\documentclass[a4paper,11pt,reqno,noindent]{amsart}
\usepackage[centertags]{amsmath}
\usepackage{amsfonts,amssymb,amsthm,dsfont,cases,amscd,esint,enumerate}
\usepackage[T1]{fontenc}
\usepackage[english]{babel}
\usepackage[applemac]{inputenc}
\usepackage{newlfont}
\usepackage{color}
\usepackage[body={15cm,21.5cm},centering]{geometry} 
\usepackage{fancyhdr}
\usepackage{enumerate}

\theoremstyle{plain}
\newtheorem{theor10}{Theorem}
\newenvironment{theor1}
  {\pushQED{\qed}\begin{theor10}}
  {\popQED\end{theor10}}
\newtheorem{prop10}{Proposition}

\newtheorem{cor10}{Corollary}

\newtheorem{theor0}{Theorem}[section]
\newenvironment{theor}
  {\pushQED{\qed}\begin{theor0}}
  {\popQED\end{theor0}}
\newtheorem{lem0}[theor0]{Lemma}
\newenvironment{lem}
  {\pushQED{\qed}\begin{lem0}}
  {\popQED\end{lem0}}
\newtheorem{prop0}[theor0]{Proposition}
\newenvironment{prop}
  {\pushQED{\qed}\begin{prop0}}
  {\popQED\end{prop0}}
\newtheorem{cor0}[theor0]{Corollary}

\newtheorem{defin0}[theor0]{Definition}
\newenvironment{defin}
  {\pushQED{\qed}\begin{defin0}}
  {\popQED\end{defin0}}
\theoremstyle{definition}
\newtheorem{rems0}[theor0]{Remarks}
\newenvironment{rems}
  {\pushQED{\qed}\begin{rems0}}
  {\popQED\end{rems0}}
\newtheorem{ex0}[theor0]{Example}
\newenvironment{ex}
  {\pushQED{\qed}\begin{ex0}}
  {\popQED\end{ex0}}
\newtheorem{rem0}[theor0]{Remark}
\newenvironment{rem}
  {\pushQED{\qed}\begin{rem0}}
  {\popQED\end{rem0}}

\theoremstyle{plain}
\newtheorem{as0}[theor0]{Assumption}

\newtheorem*{asn0*}{\assumptionnumber}
  \providecommand{\assumptionnumber}{}
  \makeatletter
  \newenvironment{asn0}[2]
   {\renewcommand{\assumptionnumber}{Assumption \!(#1) {\normalfont--- #2}}
    \begin{asn0*}
    \protected@edef\@currentlabel{{\normalfont(#1)}}}
   {\end{asn0*}}
  \makeatother
\newenvironment{asn}
  {\pushQED{\qed}\begin{asn0}}
  {\popQED\end{asn0}}

\numberwithin{equation}{section}

\newcommand{\e}{\varepsilon}

\newcommand{\Pc}{\mathcal{P}}
\newcommand{\Ec}{\mathcal E}
\newcommand{\R}{\mathbb R}
\newcommand{\Z}{\mathbb Z}
\newcommand{\Ic}{\mathcal I}
\newcommand{\Md}{\mathbb M}
\newcommand{\Nc}{\mathcal N}
\newcommand{\cvf}{\rightharpoonup}
\newcommand{\w}{\omega}
\newcommand{\loc}{{\operatorname{loc}}}
\newcommand{\Id}{\operatorname{Id}}
\newcommand{\E}{\mathbb{E}}
\newcommand{\Var}{\operatorname{Var}}
\newcommand{\per}{{\operatorname{per}}}
\newcommand{\D}{\operatorname{D}}
\newcommand{\ee}{e}
\newcommand{\supess}{\operatorname{sup\,ess}}
\newcommand{\infess}{\operatorname{inf\,ess}}
\newcommand{\Aa}{\boldsymbol a}
\newcommand{\bb}{{\boldsymbol b}}
\newcommand{\Bb}{{\boldsymbol B}}

\newcommand{\Ld}{\operatorname{L}}

\newcommand{\Div}{{\operatorname{div}}}

\newcommand{\Sym}{{\operatorname{sym}}}
\newcommand{\Skew}{{\operatorname{skew}}}

\newcommand{\step}[1]{\noindent \textit{Step} #1.}
\newcommand{\substep}[1]{\noindent \textit{Substep} #1.}
\newcommand{\Pm}{\mathbb{P}}

\newcommand{\expec}[1]{\mathbb{E}\left[ #1 \right]}
\newcommand{\expecm}[1]{\mathbb{E}\big[ #1 \big]}

\newcommand{\var}[1]{\mathrm{Var}\left[#1\right]}

\newcommand{\avoir}{\langle \nabla \rangle^\frac12}

\usepackage[colorlinks,citecolor=black,urlcolor=black]{hyperref}

\title[Sedimentation of random suspensions]{Sedimentation of random suspensions
\\and the effect of hyperuniformity}

\author[M. Duerinckx]{Mitia Duerinckx}
\address[Mitia Duerinckx]{Universit\'e Paris-Saclay, CNRS, Laboratoire de Math\'ematiques d'Orsay, 91405~Orsay, France \& Universit\'e Libre de Bruxelles, D\'epartement de Math\'ematique, 1050~Brussels, Belgium}
\email{mduerinc@ulb.ac.be}
\author[A. Gloria]{Antoine Gloria}
\address[Antoine Gloria]{Sorbonne Universit\'e, CNRS, Universit\'e de Paris, Laboratoire Jacques-Louis Lions, 75005~Paris, France \& Universit\'e Libre de Bruxelles, D\'epartement de Math\'ematique, 1050~Brussels, Belgium}
\email{gloria@ljll.math.upmc.fr}

\begin{document}
\selectlanguage{english}

\begin{abstract}
This work is concerned with the mathematical analysis of the bulk rheology of random suspensions of rigid particles settling under gravity in viscous fluids. Each particle generates a fluid flow that in turn acts on other particles and hinders their settling. In an equilibrium perspective, for a given ensemble of particle positions, we analyze both the associated mean settling speed and the velocity fluctuations of individual particles. 
In the 1970s, Batchelor gave a proper definition of the mean settling speed, a 60-year-old open problem in physics, based on the appropriate renormalization of long-range particle contributions. In the 1980s, a celebrated formal calculation by Caflisch and Luke suggested that velocity fluctuations in dimension $d=3$ should diverge with the size of the sedimentation tank, contradicting both intuition and experimental observations. The role of long-range self-organization of suspended particles in form of hyperuniformity was later put forward to explain additional screening of this divergence in steady-state observations. In the present contribution, we develop the first rigorous theory that allows to justify all these formal calculations of the physics literature.

\bigskip\noindent
{\sc MSC-class:} 35R60; 76M50; 35Q70; 76D07; 60G55.
\end{abstract}
\maketitle

\setcounter{tocdepth}{1}
\tableofcontents

\section{Introduction}

\subsection{General overview}
The present article constitutes the first rigorous analysis of the bulk sedimentation of a random suspension of rigid particles in a Stokes fluid in a large tank. 
We place ourselves in a quasistatic setting where inertial forces are neglected. Given particle positions $\{x_n\}_n$ in an (experimental) tank, the corresponding instantaneous velocities $\{V_n\}_n$ are obtained by solving the steady Stokes equation with suitable conditions at the particle boundaries, cf.~Section~\ref{sec:model}. Following the physical contributions by 
Smoluchowski~\cite{Smoluchowski-12}, Burgers~\cite{Burgers-41,Burgers-42}, and Batchelor~\cite{Batchelor-72}, we place ourselves in an equilibrium perspective, and we assume that particle positions $\{x_n\}_n$ are distributed according to a known random ensemble and we analyze the corresponding random ensemble of velocities~$\{V_n\}_n$.
We specifically focus on the mean settling speed~$\expec{V_n}$  and on velocity fluctuations $\var{V_n}$.
Under suitable mixing assumptions, we recover all the predictions of the physics literature, cf.~Sections~\ref{sec:phys} and~\ref{sec:informal}, and correctly capture in particular the effect of the long-range self-organization of particles in form of hyperuniformity.
While these predictions~\cite{Smoluchowski-11,Batchelor-72,Caflisch-Luke,KS-91} are based on a simplified linear pairwise analysis, cf.~Section~\ref{sec:heur}, the key difficulty to treat the original model stems from the multibody and nonlinear nature of hydrodynamic interactions in combination with the crucial analysis of stochastic cancellations.
The mathematical merits of this contribution are mainly threefold:
\begin{enumerate}[\quad$\bullet$]
\item \emph{Sensitivity calculus:} In the spirit of Malliavin calculus, we analyze the sensitivity of the fluid velocity with respect to the random ensemble of particle positions.
More precisely, local changes in the latter are shown to generate two effects on the fluid velocity: a linear response (similar to the simplified linear pairwise analysis used in the physics literature, cf.~Section~\ref{sec:heur}) and a nonlinear response (related to hydrodynamic interactions).
Our analysis establishes a difference of locality between these two effects: the linear response is less local by one length scale, and therefore  dominates the nonlinear response.
This suggests that assumptions such as hyperuniformity that improve the scalings in the linear analysis should yield a similar improvement for the original nonlinear model.
This is quite surprising as it would not be true in the more classical setting of stochastic homogenization of linear elliptic equations, cf.~Section~\ref{sec:comp}.
\smallskip
\item \emph{Hyperuniformity:} In order to cope with the nonlinearity of hydrodynamic interactions and rigorously exploit the above sensitivity calculus, we appeal to functional inequalities on the probability space as a convenient form of nonlinear mixing condition.
While hyperuniformity expresses the suppression of large-scale density fluctuations and is usually defined in terms of the pair correlation function of the random point process, cf.~Appendix~\ref{sec:hyper}, we need to devise functional inequalities that are compatible with hyperuniformity.
For that purpose, in line with the multiscale variance inequalities that we introduced in~\cite{DG1,DG2} for point processes, we introduce here a new family of \emph{hyperuniform variance inequalities}, cf.~assumption~\ref{Hyp+} below. We believe that these functional inequalities, which  provide a versatile tool for the rigorous analysis of nonlinear hyperuniform systems, are of independent interest.
\smallskip
\item \emph{Annealed regularity:}
Due to the nonlinearity of hydrodynamic interactions,
the analysis critically requires
fine regularity results on the steady Stokes equation with a random suspension.
Although deterministic regularity results could be used, the latter always require a large enough minimal interparticle distance --- which is physically an unsatisfying restriction ---, whether they are based on the reflection method~\cite{Jabin-Otto-04,Hofer-18,Mecherbet-19,Hofer-19} or on other perturbative ideas~\cite{Gloria-19},  cf.~Section~\ref{sec:previous}.
In contrast, in this work, we take advantage of randomness and rely on a new family of non-perturbative {\it annealed} regularity results, cf.~Section~\ref{sec:reg0}, to which the companion article \cite{DG-20+} is dedicated.
Using such refined analytical tools inspired by the quantitative theory of stochastic homogenization of divergence-form linear elliptic equations,
we fully take into account hydrodynamic interactions for the first time in a non-dilute regime.
\end{enumerate}
We start this introduction by reviewing the physics literature. Next, we introduce the precise Stokes system that we study in the present work, we give an informal statement of our main results, and briefly discuss the relation to previous works in the mathematics literature. 
Precise assumptions and rigorous statements are postponed to Section~\ref{sec:main-res}.

\subsection{Review of the physics literature}\label{sec:phys}

A linear analysis of the Stokes equation shows that each particle in a Stokes fluid generates a long-range flow disturbance, which only decays as $O(r^{2-d})$ with the distance $r$.
Any naïve summation of particle contributions would therefore diverge as $O(L^2)$ in a tank of size $L$, in particular leading to the erroneous conclusion that the mean settling speed would not be a well-defined bulk quantity and would depend on the size and shape of the tank.
This paradox remained an open problem for 60 years since the work of Smoluchowski~\cite{Smoluchowski-12} in 1912, despite several notable attempts by Burgers~\cite{Burgers-41,Burgers-42}. The correct screening of hydrodynamic interactions was first unravelled by Batchelor~\cite{Batchelor-72} in the 1970s (see also the revisited approaches by Hinch~\cite{Hinch-77} and Feuillebois~\cite{Feuillebois-84}), which formally shows that the mean settling speed is well-defined in dimension~$d>2$, and further establishes a dilute expansion at small volume fraction.
The divergence is screened by a macroscopic backflow that appears as a multiparticle effect.

\medskip\noindent
The next step towards the rheology of sedimenting suspensions is the analysis of velocity fluctuations of individual particles, which is viewed as an intermediate step towards the understanding of hydrodynamic diffusion, e.g.~\cite{NHHOG-95}.
In the 1980s, Caflisch and Luke~\cite{Caflisch-Luke} argued that for a random suspension of particles, despite Batchelor's renormalization,
velocity fluctuations should diverge linearly in the size of the tank in dimension $d=3$, which would again
contradict both intuition and steady-state experimental observations.
As pointed out by Caflisch and Luke~\cite{Caflisch-Luke}, this divergence is strongly related to the standard assumption that particle positions are maximally disordered (that is, uniformly distributed in the tank and independent up to volume exclusion), suggesting that long-range order might drastically change the conclusion.

\medskip\noindent
This was made precise by Hinch~\cite{Hinch-88} in form of the scaling analysis of a ``blob model'', where particles are assumed to be organized in large correlated regions, constituting ``blobs'' at some characteristic correlation scale: particles are maximally disordered at smaller scales,
while density fluctuations are reduced on larger scales.
For such a model, the Caflisch-Luke prediction is expected to hold only up to the correlation scale, while long-range hydrodynamic interactions would be screened on larger scales. In other words, density fluctuations are expected to drive velocity fluctuations, and the spontaneous self-organization of particles would be the key mechanism that prevents the divergence.

\medskip\noindent
This scaling analysis was later refined by Koch and Shaqfeh~\cite{KS-91}: a simple condition on the pair correlation function of the ensemble of particle positions was put forward and formally shown to ensure the boundedness of velocity fluctuations.
The condition coincides with what was later coined ``hyperuniformity'' by Torquato and Stillinger~\cite{Torquato-Stillinger-03} (see also~\cite{Torquato-16,Ghosh-Lebowitz-17}) and characterizes the suppression of density fluctuations on large scales, or equivalently the vanishing of the structure factor at vanishing wavenumber, cf.~Appendix~\ref{sec:hyper}.

\medskip\noindent
From the experimental viewpoint, it took some years to devise experiments that give 
a clear picture on velocity fluctuations. 
Observations are as follows:
when starting from a well-mixed suspension (which can be reasonably modeled as maximally disordered), initial fluctuations are diverging in agreement with the Caflisch-Luke prediction, while, after some time evolution, particles reorganize and velocity fluctuations are dramatically reduced, see e.g.~\cite{NG-95,Ladd-96,Ladd-97,SHC-97,LAT-01}.
Whether this is indeed due to long-range order is however not clear and does not seem supported by recent experiments~\cite{Guazzelli,Guazzelli2}. Other possible screening mechanisms have been proposed in the physics literature based on inertial effects, on stratification, on wall effects, or on the absorbing role of the bottom of the container, see e.g.~\cite{Brenner-99}, but still no consensus has emerged~\cite{GH-rev-10,Segre-rev-16}. Such additional effects are not taken into account in the present contribution.

\medskip\noindent
To conclude this discussion, let us emphasize that all above-mentioned predictions are based on calculations for a linearized version of the problem, which amounts to replacing multibody hydrodynamic interactions by two-body (Coulomb-type) interactions as a first-order dilute approximation, cf.~Section~\ref{sec:heur}.
The merit of the present work is to extend these results to the full model, and give definite conclusions. In particular, this contribution shows that either there is indeed long-range order in experiments,
or additional effects must be included in the sedimentation model.

\subsection{Discussion of the model}\label{sec:model}
We now describe the sedimentation model under study, which allows to compute instantaneous particle velocities from their positions in a quasistatic perspective.
We consider a tank of size $L\ge1$, which for simplicity we choose as the cube $Q_L:=(-\frac L2,\frac L2]^d$ of side length $L$ with periodic boundary conditions.
The tank is filled with a (steady) Stokes fluid, together with a monodisperse collection of disjoint spherical suspended particles,
\[\Ic_L:=\bigcup_nI_{n,L},\]
where the particle $I_{n,L}:=B(x_{n,L})$ is the unit ball centered at $x_{n,L}$ and where $\Pc_L:=\{x_{n,L}\}_n$ is a collection of positions in the tank $Q_L$. The total volume fraction is denoted by
\begin{equation}\label{eq:def-lamL}
\lambda_L:=L^{-d}\,|\Ic_L|.
\end{equation}
The fluid flow satisfies the following steady Stokes equation outside the suspended particles, with periodic boundary conditions,
\begin{equation}\label{eq:StokesL}
-\triangle \phi_L+\nabla \Pi_L = -\alpha_Le,\qquad\Div\phi_L=0,\qquad\text{in $Q_L\setminus\Ic_L$,}
\end{equation}
where the constant right-hand side $-\alpha_L e$ accounts for the multiparticle backflow in the fluid in the opposite direction to gravity $e\in\R^d$, and where we have set for abbreviation the relevant factor
\[\alpha_L:=\frac{\lambda_L}{1-\lambda_L}.\]
In the present periodic setting, this backflow is imposed by the solvability condition for the Stokes equation~\eqref{eq:StokesL} together with the boundary conditions below.
As described in Theorem~\ref{th:hom}, in sedimentation experiments, since no additional force is applied, this backflow term does not appear and is in fact compensated by the pressure.

\medskip\noindent
No-slip boundary conditions are imposed at particle boundaries. As particles are constrained to have rigid motions, this amounts to letting the velocity field $\phi_L$ be extended inside particles, with the rigidity constraint
\begin{equation}\label{eq:StokesL-BC1}
\D(\phi_L)=0,\qquad\text{in $\Ic_L$},
\end{equation}
where $\D(\phi_L)$ denotes the symmetrized gradient of $\phi_L$. In other words, this condition means that $\phi_L$ coincides with a rigid motion $V_{n,L}+\Theta_{n,L}(x-x_{n,L})$ inside each particle $I_{n,L}$, for some $V_{n,L}\in\R^d$ and skew-symmetric matrix $\Theta_{n,L}$ (cf.~$\Div \phi_L=0$).
Next, gravity $e\in\R^d$ appears through the force and torque balances on each particle,
\begin{gather}
e|I_{n,L}|+\int_{\partial I_{n,L}}\sigma(\phi_L,\Pi_L)\nu=0,\label{eq:StokesL-BC2}\\
\int_{\partial I_{n,L}}\Theta\nu\cdot\sigma(\phi_L,\Pi_L)\nu=0,\qquad\text{for all skew-symmetric matrices $\Theta$},\label{eq:StokesL-BC3}
\end{gather}
where $\sigma(\phi_L,\Pi_L)$ stands for the usual Cauchy stress tensor,
\[\sigma(\phi_L,\Pi_L)=2\D(\phi_L)-\Pi_L\Id,\]
and
where $\nu$ stands for the outward unit normal vector at particle boundaries.
As $\phi_L$ and~$\Pi_L$ are only defined up to a constant, we fix them by imposing the vanishing average conditions,
\[\int_{Q_L}\phi_{L}=0,\qquad\int_{Q_L\setminus\Ic_L}\Pi_L=0.\]
Well-posedness for~\eqref{eq:StokesL}--\eqref{eq:StokesL-BC3} with $\phi_L\in H^1_\per(Q_L)^d$ and $\Pi_L\in\Ld^2_\per(Q_L\setminus\Ic_L)$ follows from the standard theory for the steady Stokes equation, e.g.~\cite[Section~IV]{Galdi}.
In addition, regularity theory ensures that~$(\phi_L,\Pi_L)$ is smooth on $Q_L\setminus\Ic_L$, is a classical solution of~\eqref{eq:StokesL}, and that boundary conditions are satisfied in a pointwise sense.

\medskip\noindent
Given the positions~$\{x_{n,L}\}_n$ of the suspended particles $\{I_{n,L}\}_n$, the above model allows to compute their instantaneous velocities~$\{V_{n,L}\}_n$, which are given by the averaged boundary values
\begin{equation}\label{eq:def-V-phi}
V_{n,L}:=\fint_{I_{n,L}}\phi_L.
\end{equation}
(Note that the rotational or skew-symmetric part $\Theta_{n,L}$ does not contribute to the average.)

\begin{rems}\label{rem:proj}$ $
\begin{enumerate}[(a)]
\item \emph{Reformulation by projection:}\\
As checked e.g.~in~\cite{Niethammer-Schubert-19,Hofer-19}, the weak solution $\phi_L$ of the above equations~\eqref{eq:StokesL}--\eqref{eq:StokesL-BC3} can equivalently be written as~$\phi_L=\frac1{1-\lambda_L}\pi_L\phi_L^\circ$, where $\phi_L^\circ$ denotes the solution of the following ``linear'' approximation, where particle interactions are neglected,
\begin{equation}\label{eq:reform}
\quad-\triangle\phi_L^\circ+\nabla\Pi_L^\circ=\big(\mathds1_{\Ic_L}-\lambda_L\big)e,\qquad\Div\phi_L^\circ=0,\qquad\text{in $Q_L$},
\end{equation}
and where $\pi_L$ is the orthogonal projection in $H^1_\Div(Q_L):=\{\phi\in H^1_\per(Q_L):\Div\phi=0\}$ onto the subspace $\{\phi: \D(\phi)=0\text{ in $\Ic_L$}\}$.
In other words, while $\phi_L^\circ$ depends linearly on the set $\Ic_L$ of particles, the multibody nonlinear hydrodynamic interactions can be fully encoded in this projection~$\pi_L$. This reformulation could slightly simplify some calculations but will not be used in the sequel.
\smallskip\item \emph{Reflection method:}\\
As introduced by Smoluchowski~\cite{Smoluchowski-11}, the so-called reflection method aims at rewriting the complicated projection $\pi_L$ as a cluster expansion only involving single-particle operators:
denoting by $\pi_L^n$ the projection in $H^1_\Div(Q_L)$ onto $\{\phi:\D(\phi)=0\text{ in $I_{n,L}$}\}$, and setting $q_L^n:=1-\pi_L^n$, the expansion takes the form $\pi_L=\Id-\sum_nq_L^n+\sum_{n\ne m}q_L^nq_L^m-\ldots$
Such an expansion appears to be very useful as single-particle operators $\{q_L^n\}_n$ are essentially explicit.
However, as shown in~\cite{Niethammer-Schubert-19,Hofer-19}, based on deterministic arguments,
convergence is only expected in the dilute regime --- more precisely, for a large enough minimal interparticle distance.
For this reason, such simplifying tools are systematically avoided in the sequel.
\qedhere
\end{enumerate}
\end{rems}

\subsection{Informal statement of the results}\label{sec:informal}
We henceforth consider a random ensemble of suspended particles in form of a stationary point process $\Pc_L=\{x_{n,L}\}_n$ with intensity~$\rho_L$ in the periodic tank~$Q_L$, constructed on a given probability space $(\Omega,\Pm)$, and we analyze both the corresponding mean settling speed and the fluctuations of individual velocities,
\begin{equation}\label{def:eff-vel}
\bar V_L:=\tfrac{e}{|e|}\cdot \expec{V_{n,L}},\qquad\sigma_L:=|\var{V_{n,L}}\!|^\frac12,
\end{equation}
in the large-volume limit $L\uparrow\infty$, where $\expec{\cdot}$ and $\var{\cdot}$ denote the expectation and the variance with respect to $\Pm$, respectively.
Note that a direct computation from~\eqref{eq:StokesL}--\eqref{eq:StokesL-BC3} shows that the averaged settling speed is proportional to the
Dirichlet form of the fluid velocity, cf.~\eqref{eq:def-V} below,
\begin{equation}\label{eq:V-energy}
\alpha_L|e|\bar V_L~~\overset{L\uparrow\infty}\sim~~\expec{|\nabla \phi_L|^2}.
\end{equation}
Our main result states that, for a mixing ensemble of particles without long-range order, the mean settling speed and velocity fluctuations are well-defined in the large-volume limit only in dimensions $d>2$ and $d>4$, respectively. More precisely, the following bounds are expected to be sharp,
\begin{equation}\label{eq:res1}
\frac{\bar V_L}{\rho_L|e|}\lesssim\left\{\begin{array}{lll}1&:&d>2;\\(\log L)^\frac12&:&d=2;\\L^\frac12&:&d=1;\end{array}\right.\qquad\text{and}\qquad
\frac{\sigma_L}{\rho_L|e|}\lesssim\left\{\begin{array}{lll}1&:&d>4;\\(\log L)^\frac12&:&d=4;\\L^\frac12&:&d=3;\end{array}\right.
\end{equation}
cf.~Theorems~\ref{th:main1}(i) and~\ref{th:main2}(i).
In particular, the boundedness of $\bar V_L$ for $d>2$ fully justifies Batchelor's analysis~\cite{Batchelor-72,Hinch-77,Feuillebois-84},
while the linear divergence of $\sigma_L^2$ for $d=3$ provides a rigorous version of the celebrated calculation by Caflisch and Luke~\cite{Caflisch-Luke} (and extends the results of~\cite{Gloria-19} to the present much more general setting, cf.~Section~\ref{sec:previous}).

\medskip\noindent
Next, we investigate the role of the hyperuniformity of the suspension and rigorously analyze how it leads to the screening of hydrodynamic interactions. 
Under a suitable functional-analytic version of hyperuniformity, we show that the critical dimensions in~\eqref{eq:res1} are shifted by $2$,
\begin{equation}\label{eq:res2}
\frac{\bar V_L}{\rho_L|e|}\lesssim1\qquad\text{and}\qquad
\frac{\sigma_L}{\rho_L|e|}\lesssim\left\{\begin{array}{lll}1&:&d>2;\\(\log L)^\frac12&:&d=2;\\L^\frac12&:&d=1;\end{array}\right.
\end{equation}
cf.~Theorems~\ref{th:main1}(ii) and~\ref{th:main2}(ii).
This rigorously justifies the heuristic calculations in dimension~$d=3$ by Hinch~\cite{Hinch-88} and Koch and Shaqfeh~\cite{KS-91}.

\medskip\noindent
In addition, whenever the averaged settling speed remains bounded, we make sense of an infinite-volume equation describing the limit of~\eqref{eq:StokesL}--\eqref{eq:StokesL-BC3}, cf.~Theorem~\ref{th:main1}.
We also deduce a homogenization result for sedimentation experiments, cf.~Theorem~\ref{th:hom}:
in the limit of a dense suspension of small particles (with suitably rescaled gravity) in a finite tank, the velocity field of the Stokes fluid with the suspension converges weakly to that of an effective Stokes fluid. The latter is characterized by an effective viscosity, which is shown to be independent of gravity. In particular, this effective viscosity is the same as for the corresponding colloidal (non-sedimenting) system that we studied in~\cite{DG-19}. The local behavior of the fluid is however drastically impacted by sedimentation, which is expressed in form of a corrector result.

\subsection{Relation to previous works}\label{sec:previous}
The mathematical literature on settling suspensions in viscous fluids is particularly scarce.
Most contributions have been investigating the (quasistatic) dynamics:
particle positions $\{x_{n,L}\}_n$ evolve according to the set of ODEs $\{\dot x_{n,L}=V_{n,L}\}_n$, where instantaneous velocities $\{V_{n,L}\}_n$ are computed from the steady Stokes problem~\mbox{\eqref{eq:StokesL}--\eqref{eq:StokesL-BC3}}.
The first result on this topic is due to Jabin and Otto~\cite{Jabin-Otto-04} and identifies the non-interacting regime, while more recent contributions by Höfer~\cite{Hofer-18} and Mecherbet~\cite{Mecherbet-19} have studied the mean-field limit.
Those works are restricted to a perturbative dilute regime (the particle density tends to zero as the particle radii with some scaling relation)
and they only rely on deterministic arguments in form of the reflection method; see also~\cite{Niethammer-Schubert-19,Hofer-19}.

\medskip\noindent
The analysis of velocity fluctuations further requires to capture stochastic cancellations, for which a probabilistic input is crucially needed. This was first performed
by the second author
in~\cite{Gloria-19} for a scalar version of the sedimentation problem, and the corresponding version of the Caflisch-Luke bound~\eqref{eq:res1} on velocity fluctuations was succesfully established. As in the present work, the mixing assumption on the point process $\{x_{n,L}\}_n$ was conveniently formulated in form of a multiscale functional inequality as we introduced in~\cite{DG1,DG2},
and the proof borrowed ideas from quantitative stochastic homogenization for divergence-form linear elliptic equations.
As opposed to the present contribution however, the approach was based on deterministic regularity properties for (a scalar version of) the Stokes equation with a suspension, in particular in form of a perturbative Green's function estimate, which  only holds under the assumption that the minimal interparticle distance is large enough. Diluteness was also used at several other instances, together with the scalar nature of the equation and the sphericity of the particles.

\medskip\noindent
The present work widely generalizes the result of~\cite{Gloria-19}:
\begin{enumerate}[\quad$\bullet$]
\item We fully relax the diluteness requirement. For that purpose,
we resort to a random version of regularity properties in form of annealed $\Ld^p$ regularity in the spirit of a recent work of the first author and Otto~\cite{DO1} for divergence-form linear elliptic equations with random coefficients.
In addition, the analysis of stochastic cancellations relies on a new, particularly efficient buckling argument, which exploits this annealed regularity and is inspired from~\cite{Otto-Tlse}. This approach also allows to treat non-spherical particles.
\smallskip\item Treating the vectorial case of the Stokes equation further requires to control the pressure. 
To this aim, we follow an approach based on local regularity~\cite{Galdi}, which we initiated in~\cite{DG-19} in the simpler setting of colloidal (non-sedimenting) suspensions.
\end{enumerate}
In addition, the present work studies for the first time the effect of hyperuniformity of the 
random ensemble of positions.
In line with our works on multiscale functional inequalities~\cite{DG1,DG2}, we introduce the notion of hyperuniform functional inequalities, which allows us to unravel additional stochastic cancellations in sedimentation.
More generally, this work constitutes to our knowledge the first example of a nonlinear physical system for which the hyperuniformity of the input is rigorously shown to improve the scalings.

\medskip
\subsection*{Notation}
\begin{enumerate}[$\bullet$]
\item We denote by $C\ge1$ any constant that only depends on the space dimension $d$ and on controlled constants appearing in the assumptions. We use the notation $\lesssim$ (resp.~$\gtrsim$) for $\le C\times$ (resp.~$\ge\frac1C\times$) up to such a multiplicative constant $C$. We write $\simeq$ when both $\lesssim$ and $\gtrsim$ hold. In an assumption, we write $\ll$ (resp.~$\gg$) for $\le\frac1C\times$ (resp.~$\ge C\times$) with some large enough constant $C$. We add subscripts to $C,\lesssim,\gtrsim,\simeq,\ll,\gg$ in order to indicate dependence on other parameters.
\smallskip\item The ball centered at $x$ of radius $r$ in $\R^d$ is denoted by $B_r(x)$, and we simply write $B(x):=B_1(x)$, $B_r:=B_r(0)$, and $B:=B_1(0)$.
\smallskip\item For a function $f$ we write $[f]_2(x):=(\fint_{B(x)}|f|^2)^{1/2}$ for its local moving quadratic average.
\item We set $\langle x\rangle:=(1+|x|^2)^{1/2}$ for $x\in\R^d$, and we similarly define $\langle\nabla\rangle=(1-\triangle)^{1/2}$.  We denote by $|x|_L$ the Euclidean distance between $x$ and $0$ modulo $L\Z^d$ (that is, the distance on $Q_L$ viewed as the flat torus $\R^d/L\Z^d$). We set $a\wedge b:=\min\{a,b\}$ for all $a,b\in\R$. We denote by $\sharp E$ the cardinality of a locally finite set $E$.
\smallskip\item We denote by $\D(u):=\frac12(\nabla u+(\nabla u)')$ the symmetrized gradient of $u$, by $\sigma(\phi,\Pi):=2\D(\phi)-\Pi\Id$ the Stokes stress tensor, by~$\nu$ the outward unit normal vector at particle boundaries, and by $\Md^\Skew$ the set of skew-symmetric $d\times d$ matrices.
\end{enumerate}


\medskip
\section{Main results}\label{sec:main-res}

We start with suitable assumptions on the random ensemble of suspended particles, including precise mixing and hyperuniformity assumptions, and then turn to the statement of the main results on the mean settling speed, velocity fluctuations, and homogenization.
Finally, for the reader's convenience, a heuristic proof of our results in the dilute regime is included in Section~\ref{sec:heur}, justifying the scalings and illustrating the use of hyperuniformity, while in Section~\ref{sec:comp} we underline the useful analogy with stochastic homogenization of linear elliptic equations.

\subsection{Assumptions}
The following general assumption makes precise the notion of convergence for the point process $\Pc_L$ in the periodic tank $Q_L$ in the large-volume limit $L\uparrow\infty$. We also always assume that the process is hardcore, with some uniform bound $\delta>0$.

\begin{asn}{H$_\delta$}{General conditions}\label{Hd}$ $\\
The family $(\Pc_L)_{L\ge1}$ of point processes is constructed on some probability space~$(\Omega,\Pm)$ and satisfies the following properties:
\begin{itemize}
\item \emph{Periodicity in law:} For all $L\ge1$, the point process $\Pc_L=\{x_{n,L}\}_n$ is defined on the periodic cell $Q_L:=(-\frac L2,\frac L2]^d$ and is stationary with respect to shifts on the latter.\footnote{More precisely, as is standard in the field, e.g.~\cite{PapaVara} or~\cite[Section~7]{JKO94}, stationarity is understood as follows: there exists a measure-preserving group action $\{\tau_{L,x}\}_{x\in Q_L}$ of $(\R^d/L\Z^d,+)$ on the probability space~$(\Omega,\Pm)$ such that $\Pc_L^\w+x = \Pc_L^{\tau_{L,x} \w}$ for all $x,\w$.}
\item \emph{Stabilization:} For any compact set $K\subset\R^d$ the restricted point set~$\Pc_L\cap K$ converges almost surely as $L\uparrow\infty$. We denote by $\Pc$ the limiting point process, which is assumed stationary (on $\R^d$) and ergodic, and we denote by $\Ic:=\cup_nI_n$ the corresponding particle suspension.
\item \emph{Hardcore condition:} For all $L\ge1$ the point process $\Pc_L$ satisfies
\[\qquad\qquad\inf_{m\ne n}|x_m-x_n|_L\,\ge\,2(1+\delta)\quad\text{almost surely}.\qedhere\]
\end{itemize}
\end{asn}

Before stating mixing and hyperuniformity assumptions, we recall the following standard definition of the pair correlation function of the stationary random point process~$\Pc_L$.
\begin{defin}[Pair correlation function]\label{def:corr-fct}
The \emph{intensity} of $\Pc_L$ is defined by
\[\rho_L:=\expecm{L^{-d}\,\sharp\Pc_L},\]
which is related to the volume fraction $\lambda_L$ of the suspension $\Ic_L$ via $\expec{\lambda_L}=|B|\rho_L$, cf.~\eqref{eq:def-lamL}.
The \emph{pair density function} $f_{2,L}:Q_L\to\R^+$ of $\Pc_L=\{x_{n,L}\}_n$ is defined via the following relation, for all $\zeta\in C^\infty_\per(Q_L\times Q_L)$,
\[\E\bigg[\sum_{n\ne m}\zeta(x_{n,L},x_{m,L})\bigg]=\rho_L^2\iint_{Q_L\times Q_L}\zeta(x,y)\,f_{2,L}(x-y)\,dxdy.\]
The \emph{pair correlation function} is defined as $g_{2,L}:=\rho_L^{-2}(f_{2,L}-\rho_L^2)$.
The \emph{total pair correlation function} of $\Pc_L$  is defined by $h_{2,L}(x):=g_{2,L}(x)+\rho_L^{-1}\delta(x)$ and is characterized by the following relation, for all $\zeta\in C^\infty_\per(Q_L)$,
\begin{equation}\label{var-hyper}
\Var\bigg[\sum_{n}\zeta(x_{n,L})\bigg]=\rho_L^2\iint_{Q_L\times Q_L}\zeta(x)\zeta(y)\,h_{2,L}(x-y)\,dxdy.\qedhere
\end{equation}
\end{defin}

After an appropriate mechanical mixing of the particle suspension, the distribution of particle positions in the fluid typically displays fast decaying correlations and we naturally consider the following type of condition.

\begin{asn}{Mix}{Mixing condition}\label{Mix}$ $\\
The pair correlation function $g_{2,L}$ of $\Pc_L$ is integrable in the sense of
\[\sup_{L\ge1}\int_{Q_L}|g_{2,L}|\,<\,\infty.\qedhere\]
\end{asn}

Next, in the spirit of the discussion in Section~\ref{sec:phys}, see also~\cite{Hinch-88,KS-91}, we consider the assumption that the sedimenting suspension displays long-range structural order such that density fluctuations are suppressed, while still displaying fast decaying correlations. This is formalized through the notion of hyperuniformity, which we define as follows in the periodized setting. We refer to Appendix~\ref{sec:hyper} for a detailed motivation and some reformulations.

\begin{asn}{Hyp}{Mixing and hyperuniformity conditions}\label{Hyp}$ $\\
The pair correlation function $g_{2,L}$ of $\Pc_L$ has fast decay in the sense of
\[\sup_{L\ge1}\int_{Q_L}|x|_L^2\,|g_{2,L}(x)|\,dx\,<\,\infty,\]
and hyperuniformity holds in the sense that the total pair correlation $h_{2,L}$ satisfies
\[\sup_{L\ge1}\,L^2\,\Big|\int_{Q_L}h_{2,L}\Big|\,<\, \infty,\]
which is viewed as the approximate vanishing of the so-called structure factor at small wavenumber, cf.~Appendix~\ref{sec:hyper}.
\end{asn}

While the above assumptions~\ref{Mix} and~\ref{Hyp} are expressed in terms of the pair correlation function,
stronger versions quickly become necessary to analyze the effects of nonlinear multibody interactions.
First, the mixing assumption~\ref{Mix} should be replaced by a stronger mixing assumption. As in~\cite{Gloria-19}, we choose to appeal to the following nonlinear assumption in form of a \emph{multiscale variance inequality}, cf.~\cite{DG1,DG2}. As shown in~\cite[Section~3]{DG2}, examples of point processes that satisfy this assumption include for instance (periodized) hardcore Poisson processes and random parking processes.

\begin{asn}{Mix$^+$}{Improved mixing condition}\label{Mix+}$ $\\
There exists a non-increasing weight function $\pi:\R^+\to\R^+$ with superalgebraic decay (that is, $\pi(\ell)\le C_p\langle\ell\rangle^{-p}$ for all $p\ge1$) such that
for all $L\ge1$ the point process $\Pc_L$ satisfies, for all $\sigma(\Pc_L)$-measurable random variables~$Y(\Pc_L)$,
\begin{equation}\label{eq:SGL}
\var{Y(\Pc_L)}\,\le\,\expec{\int_0^L\int_{Q_L}\Big(\partial^{\operatorname{osc}}_{\Pc_L,B_\ell(x)}Y(\Pc_L)\Big)^2dx\,\langle\ell\rangle^{-d}\pi(\ell)\,d\ell},
\end{equation}
where the ``oscillation'' derivative $\partial^{\operatorname{osc}}$ is defined by
\begin{align*}
&\partial^{\operatorname{osc}}_{\Pc,B_\ell(x)}Y(\Pc):=\supess\Big\{Y(\Pc'):\Pc'|_{Q_L\setminus B_\ell(x)}=\Pc|_{Q_L\setminus B_\ell(x)}\Big\}\\
&\hspace{5cm}-\infess\Big\{Y(\Pc'):\Pc'|_{Q_L\setminus B_\ell(x)}=\Pc|_{Q_L\setminus B_\ell(x)}\Big\}.\qedhere
\end{align*}
\end{asn}

Similarly, we strengthen the hyperuniformity assumption~\ref{Hyp} in form of a suitable variance inequality. Intuitively, while the ``oscillation'' derivative in~\ref{Mix+} above allows to locally add and move points, the carr\'e-du-champ below only accounts (at leading order) for moving points, but not adding or removing any, in accordance with the definition of hyperuniformity as suppressing density fluctuations.
We refer to Appendix~\ref{sec:hyper} for a detailed discussion.

\begin{asn}{Hyp$^+$}{Improved mixing and hyperuniformity conditions}\label{Hyp+}$ $\\
There exists a non-increasing weight function $\pi:\R^+\to\R^+$ with superalgebraic decay
such that for all $L\ge1$ the point process $\Pc_L$ satisfies, for all $\sigma(\Pc_L)$-measurable random variables~$Y(\Pc_L)$,
\begin{equation}\label{eq:SGL-hyp}
\var{Y(\Pc_L)}\,\le\,\expec{\int_0^L\int_{\R^d}\Big(\partial^{\operatorname{hyp}}_{\Pc_L,B_\ell(x)}Y(\Pc_L)\Big)^2dx\,\langle\ell\rangle^{-d}\pi(\ell)\,d\ell},
\end{equation}
where the ``hyperuniform'' derivative is given by 
$$\partial^{\operatorname{hyp}}_{\Pc_L,B_\ell(x)}Y(\Pc_L)=\partial^{\operatorname{mov}}_{\Pc_L,B_\ell(x)}Y(\Pc_L)+L^{-1}\partial^{\operatorname{osc}}_{\Pc_L,B_\ell(x)}Y(\Pc_L)$$ and the 
``move-point'' derivative $\partial^{\operatorname{mov}}$ is defined by
\begin{align*}
&\partial^{\operatorname{mov}}_{\Pc,B_\ell(x)}Y(\Pc):=\supess\Big\{Y(\Pc'):\Pc'|_{Q_L\setminus B_\ell(x)}=\Pc|_{Q_L\setminus B_\ell(x)},\,\sharp\Pc'|_{B_\ell(x)}=\sharp\Pc|_{B_\ell(x)}\Big\}\\
&\hspace{2cm}-\infess\Big\{Y(\Pc'):\Pc'|_{Q_L\setminus B_\ell(x)}=\Pc|_{Q_L\setminus B_\ell(x)},\,\sharp\Pc'|_{B_\ell(x)}=\sharp\Pc|_{B_\ell(x)}\Big\}.\qedhere
\end{align*}
\end{asn}

\subsection{Mean settling speed}
Our first main result concerns the mean settling speed $\bar V_L$, cf.~\eqref{def:eff-vel}, which is shown to be well-defined under the mixing condition~\ref{Mix} only in dimension $d>2$ while under hyperuniformity~\ref{Hyp} it is well-defined in all dimensions.
We also make sense of a limiting equation in the large-volume limit $L\uparrow\infty$.
The proof is surprisingly elementary, although fully taking into account multibody hydrodynamic interactions: the argument is solely based on $\Ld^2$~theory, and the 
standard weak form~\ref{Hyp} of hyperuniformity is enough to unravel the screening.

\begin{theor1}[Mean settling speed]\label{th:main1}
Let the random point processes $(\Pc_L)_{L\ge1}$ satisfy the general assumption~\ref{Hd} for some $\delta>0$.
\begin{enumerate}[(i)]
\item Under the mixing assumption~\emph{\ref{Mix}}, there holds for all $L\ge1$,
\begin{equation*}
\qquad\frac{\bar V_L}{\rho_L|e|}\lesssim\left\{\begin{array}{lll}1&:&d>2;\\(\log L)^\frac12&:&d=2;\\L^\frac12&:&d=1.\end{array}\right.
\end{equation*}
More precisely, in dimension $d>2$,  for almost all $\w$,
\begin{equation}
\qquad\lim_{L\uparrow\infty}\frac1{\sharp\Pc_L}\sum_n\tfrac{e}{|e|}\cdot V_{n,L}^\w~=~\lim_{L\uparrow\infty}\bar V_L~=~\bar V,\label{eq:barV}
\end{equation}
in terms of
\begin{equation}\label{eq:def-alpha}
\qquad\bar V:=\frac1{\alpha |e|}\expec{|\nabla\phi|^2},\qquad \alpha:=\frac\lambda{1-\lambda},\qquad\lambda:=\expec{\mathds1_{\Ic}}=\lim_{L\uparrow\infty}\lambda_L^\w,
\end{equation}
where the random field $\phi\in \Ld^2(\Omega;H^1_\loc(\R^d)^d)$ denotes the unique solution of the following infinite-volume problem:
\begin{enumerate}[\quad$\bullet$]
\smallskip\item For almost all $\w$ the realization $\phi^\w\in H^1_\loc(\R^d)^d$ satisfies in the weak sense, for some pressure field $\Pi^\w\in\Ld^2_\loc(\R^d)$,
\begin{equation}\label{eq:cor-sed}
\qquad\quad\left\{\begin{array}{ll}
-\triangle \phi^\w+\nabla \Pi^\w=-\alpha e,&\text{in $\R^d\setminus\Ic^\w$},\\
\Div \phi^\w=0,&\text{in $\R^d\setminus\Ic^\w$},\\
\D(\phi^\w)=0,&\text{in $\Ic^\w$},\\
e|I_n^\w|+\int_{\partial I_n^\w}\sigma(\phi^\w,\Pi^\w)\nu=0,&\forall n,\\
\int_{\partial I_n^\w}\Theta\nu\cdot\sigma(\phi^\w,\Pi^\w)\nu=0,&\forall n,\,\forall\Theta\in\Md^\Skew.
\end{array}\right.
\end{equation}
\item The gradient field $\nabla\phi$ and the pressure field $\Pi$ are stationary, they have vanishing expectations $\expecm{\nabla\phi}=0$ and $\expecm{\Pi\mathds1_{\R^d\setminus\Ic}}=0$, they have bounded second moments
\[\qquad\quad\expecm{|\nabla\phi|^2}+\expecm{|\Pi|^2\mathds1_{\R^d\setminus\Ic}}\lesssim|e|^2,\]
and $\phi$ is anchored at the origin in the sense of $\fint_B \phi^\w=0$ for almost all $\w$.
\end{enumerate}
\smallskip\item Under the mixing and hyperuniformity assumption~\ref{Hyp}, in any dimension $d\ge1$, there holds for all $L\ge1$,
\begin{equation*}
\qquad\frac{\bar V_L}{\rho_L|e|}\lesssim1,
\end{equation*}
 the limit~\eqref{eq:barV}  holds, and the infinite-volume problem~\eqref{eq:cor-sed} is always well-posed.
\qedhere
\end{enumerate}
\end{theor1}

\subsection{Velocity fluctuations}
Our next main result concerns the estimation of velocity fluctuations, which
requires a much finer use of stochastic cancellations.
In this context, the analysis of the effects of nonlinear multibody interactions requires a suitable strengthening of the standard mixing and hyperuniformity assumptions~\ref{Mix} and~\ref{Hyp},
and we rather appeal to their nonlinear functional-analytic versions~\ref{Mix+} and~\ref{Hyp+}. In addition, this result crucially relies on annealed regularity properties for the steady Stokes equation in presence of a random suspension, cf.~Section~\ref{sec:reg0}.
Rather than focussing on the variance~$\sigma_L^2$, cf.~\eqref{def:eff-vel}, we further consider higher moments of the velocity field $\phi_L$.

\begin{theor1}[Velocity fluctuations]\label{th:main2}
Let the random point processes $(\Pc_L)_{L\ge1}$ satisfy the general assumption~\ref{Hd} for some $\delta>0$.
\begin{enumerate}[(i)]
\item Under the improved mixing assumption~\ref{Mix+}, in any dimension $d>2$, we have for all $L\ge1$ and $1\le p<\infty$,
\begin{equation}\label{eq:nablaphi-bnd}
\qquad\|\nabla\phi_L\|_{\Ld^{p}(\Omega)}\lesssim_p |e|,
\end{equation}
and
\begin{equation}\label{eq:phi-bnd}
\qquad\|\phi_L(x)\|_{\Ld^{p}(\Omega)}\lesssim_p |e|\times\begin{cases}
1,&\text{if $d>4$};\\
\log(2+|x|)^\frac12,&\text{if $d=4$};\\
\langle x\rangle^\frac12,&\text{if $d=3$}.
\end{cases}
\end{equation}
In particular, in dimension $d>4$, up to relaxing the anchoring condition, the solution~$\phi$ of the infinite-volume problem~\eqref{eq:cor-sed} can be uniquely constructed as a stationary object with vanishing expectation.
\smallskip\item Under the improved mixing and hyperuniformity assumption~\ref{Hyp+}, in any dimension $d\ge1$, we have for all $L\ge1$ and $1\le p<\infty$,
\begin{equation}\label{eq:nablaphi-bnd+}
\qquad\|\nabla\phi_L\|_{\Ld^{p}(\Omega)}\lesssim_p |e|,
\end{equation}
and
\begin{equation}\label{eq:phi-bnd+}
\qquad\|\phi_L(x)\|_{\Ld^{p}(\Omega)}\lesssim_p |e|\times\begin{cases}
1,&\text{if $d>2$};\\
\log(2+|x|)^\frac12,&\text{if $d=2$};\\
\langle x\rangle^\frac12,&\text{if $d=1$}.
\end{cases}
\end{equation}
In particular, in dimension $d>2$, up to relaxing the anchoring condition, the solution~$\phi$ of the infinite-volume problem~\eqref{eq:cor-sed} can be uniquely constructed as a stationary object with vanishing expectation.
\qedhere
\end{enumerate}
\end{theor1}

Stochastic cancellations are conveniently exploited in form of the fluctuation scaling of large-scale averages of $\nabla \phi_L$, from which the above moment bounds are deduced as consequences.
More precisely, we show that,
under the improved mixing assumption~\ref{Mix+}, in dimension $d>2$, fluctuations of~$\nabla\phi_L$ miss the usual central limit theorem scaling by a length scale (as encoded by the norm of the test function in $\Ld^{\frac{2d}{d+2}}$ instead of $\Ld^2$):  for all $g\in C^\infty_\per(Q_L)^{d\times d}$ and $1\le p<\infty$, we have
\begin{equation}\label{eq:fluc-scaling-1}
\Big\|\int_{Q_L}g:\nabla \phi_{L}\Big\|_{\Ld^{2p}(\Omega)}\,\lesssim_p\,\|[\avoir g]_2\|_{\Ld^{\frac{2d}{d+2}}(Q_L)},
\end{equation}
while under the improved mixing and hyperuniformity assumption~\ref{Hyp+} in any dimension $d\ge1$ the usual central limit theorem scaling is recovered in form of
\begin{equation}\label{eq:fluc-scaling-2}
\Big\|\int_{Q_L}g:\nabla \phi_{L}\Big\|_{\Ld^{2p}(\Omega)}\,\lesssim_p\,\|\avoir g\|_{\Ld^2(Q_L)}.
\end{equation}
(Note that the additional gradient $\langle \nabla \rangle^\frac12$ in the bounds plays no role on large scales.)

\subsection{Homogenization result}
We consider a steady Stokes fluid in a bounded domain with internal forces and a dense suspension of small particles: we analyze the non-dilute homogenization regime with vanishing particle radii but fixed volume fraction $\lambda_\e^\w\to\lambda>0$.
Suspended particles in the fluid act as obstacles, hindering the fluid flow and therefore increasing the flow resistance, that is, the viscosity. The fluid with the suspension is then expected to behave approximately like a Stokes fluid with some effective viscosity --- which was the basis of Perrin's celebrated experiment to estimate the Avogadro number as inspired by Einstein's PhD thesis~\cite{Einstein-06}. The upcoming theorem shows that the effective viscosity for a sedimenting suspension exactly coincides with that for a colloidal (non-sedimenting) suspension, although the local behavior of the fluid flow is drastically different as expressed via the \emph{corrector result}.
This constitutes the counterpart for sedimenting suspensions of our recent work~\cite[Theorem~1]{DG-19} on colloidal suspensions.
Strikingly, although the result is only qualitative, it requires strong mixing conditions and quantitative estimates, in particular relying on Theorems~\ref{th:main1} and~\ref{th:main2} above, as opposed to the much simpler situation of colloidal suspensions in~\cite{DG-19}. An optimal convergence rate (which relies on a further use of annealed regularity) will be given in the companion article~\cite{DG-20+}.
This homogenization result is the very first of its kind in the context of sedimentation.
In particular, as opposed to contributions such as~\cite{Hillairet-18,CH-20,HMS-19}, where particle velocities are prescribed a priori rather than deduced from the steady Stokes equation, there is no Brinkman term in the effective equation. In addition, we consider a non-dilute regime (cf.~$\lambda_\e^\w\to\lambda>0$), where the multiparticle effect of hydrodynamic interactions is not negligible.

\medskip\noindent
We start with some notation: Given a reference bounded Lipschitz domain $U$, we consider the set $\Nc_\e^\w(U)$ of all indices $n$ 
such that $\e(I_n^\w+\delta B)\subset U$, and we define the corresponding rescaled particle suspension $\Ic_\e^\w(U)$ in $U$,
\[\Ic_\e^\w(U):=\bigcup_{n\in\Nc_\e^\w(U)}\e I_n^\w.\]
Note that particles in this collection are at least at distance $\e\delta$ from one another and from the boundary $\partial U$.

\begin{theor1}[Homogenization of steady Stokes flow with sedimenting suspension]\label{th:hom}
Let the stationary random point process $\Pc$ be as in~\emph{\ref{Hd}}. Under the improved mixing assumption~\ref{Mix+} in dimension $d>2$, or under the improved mixing and hyperuniformity assumption~\ref{Hyp+} in any dimension $d\ge1$,
given a bounded Lipschitz domain $U\subset\R^d$,  an internal force $f\in\Ld^2(U)$, and gravity $e\in\R^d$, we consider for all $\e>0$ and $\w\in\Omega$ the unique weak solution $u_\e^\w\in H^1_0(U)$ of the following steady Stokes problem,
\begin{equation}\label{eq:Stokes}
\left\{\begin{array}{ll}
-\triangle u_\e^\w+\nabla P_\e^\w=f,&\text{in $U\setminus\Ic_\e^\w(U)$},\\
\Div u_\e^\w=0,&\text{in $U\setminus\Ic_\e^\w(U)$},\\
u_\e^\w=0,&\text{on $\partial U$},\\
\D(u_\e^\w)=0,&\text{in $\Ic_\e^\w(U)$},\\
\e^{d-1}e |I_n^\w|+\int_{\e\partial I_n^\w}\sigma(u_\e^\w,P_\e^\w)\nu=0,&\forall n\in\Nc_\e^\w(U),\\
\int_{\e\partial I_n^\w}\Theta\nu\cdot\sigma(u_\e^\w,P_\e^\w)\nu=0,&\forall n\in\Nc_\e^\w(U),\,\forall\Theta\in\Md^\Skew.
\end{array}\right.
\end{equation}
Then the following results hold,
\begin{enumerate}[(i)]
\item \emph{Homogenization:} For almost all $\w$, $u_\e^\w\cvf\bar u$ weakly in $H^1_0(U)$, where $\bar u\in H^1_0(U)$ is the unique weak solution of the  homogenized Stokes problem
\begin{equation}\label{eq:Stokes-hom}
\quad\left\{\begin{array}{ll}
-\Div\bar\Bb \D( \bar u)+\nabla\bar P=(1-\lambda) f ,&\text{in $U$},\\
\Div\bar u=0,&\text{in $U$},\\
\bar u=0,&\text{on $\partial U$},
\end{array}\right.
\end{equation}
where $\lambda:=\expec{\mathds 1_{\Ic}}$ is the intensity of the inclusion process, and the effective viscosity tensor $\bar\Bb$ is positive definite on trace-free matrices and is given by
\begin{equation}\label{eq:def-B}
\quad\bar\Bb:= \sum_{E,E'\in\Ec}(E'\otimes E)~\expec{(\nabla\psi_{E'}+E'):(\nabla\psi_{E}+E)},
\end{equation}
where the sum runs over an orthonormal basis $\Ec$ of trace-free $d\times d$ matrices, and where the random field $\psi_E\in \Ld^2(\Omega;H^1_\loc(\R^d)^{d\times d})$ denotes the unique solution of the following (infinite-volume) ``colloidal corrector'' problem:
\begin{enumerate}[\quad$\bullet$]
\item For almost all $\w$, the realization $\psi_E^\w\in H^1_\loc(\R^d)^d$ satisfies in the weak sense, for some pressure field $\Sigma_E^\w\in\Ld^2_\loc(\R^d)$,
\begin{equation}\label{eq:corr-Stokes}
\qquad\qquad\left\{\begin{array}{ll}
-\triangle\psi_E^\w+\nabla\Sigma_E^\w=0,&\text{in $\R^d\setminus\Ic^\w$},\\
\Div\psi_E^\w=0,&\text{in $\R^d\setminus\Ic^\w$},\\
\D\big(\psi_E^\w+E(x-x_n^\w)\big)=0,
&\text{in $\Ic^\w$},\\
\fint_{\partial I_n^\w}\sigma\big(\psi_E^\w+E(x-x_n^\w),\Sigma_E^\w\big)\nu=0,&\forall n,\\
\fint_{\partial I_n^\w}\Theta\nu\cdot\sigma\big(\psi_E^\w+E(x-x_n^\w),\Sigma_E^\w\big)\nu=0,&\forall n,\,\forall\Theta\in\Md^\Skew.
\end{array}\right.
\end{equation}
\item The gradient field $\nabla\psi_E$ and the pressure field $\Sigma_E$ are stationary, they have vanishing expectations $\expecm{\nabla\psi_E}=0$ and $\expecm{\Sigma_E\mathds1_{\R^d\setminus\Ic}}=0$, they have bounded second moments, and $\psi_E$ satisfies the anchoring condition $\fint_B \psi_E^\w=0$ for almost all $\w$.
\end{enumerate}
\smallskip\item \emph{Convergence of pressure:}
For almost all~$\w$, the pressure field $P_\e^\w$ converges weakly in~$\Ld^2(U)$ up to suitable renormalization in the following sense,
\begin{multline*}
\qquad\bigg((P_\e^\w-\tfrac1\e \lambda e \cdot x) -\fint_{U\setminus\Ic_\e^\w(U)}(P_\e^\w-\tfrac1\e \lambda e \cdot x) \bigg)\mathds1_{U\setminus\Ic_\e^\w(U)}\\
~\cvf~(1-\lambda)\bigg(\bar P+\bb:\D(\bar u)-\fint_U\bar P\bigg),
\end{multline*}
where $\bb$ is a symmetric trace-free matrix given for all $E\in\Md_0^\Sym$ by
\begin{equation}\label{eq:def-b}
\bb:E\,:=\,\frac1d\,\E\bigg[{\sum_n\frac{\mathds1_{I_n}}{|I_n|}\int_{\partial I_n}(x-x_n)\cdot\sigma(\psi_E+Ex,\Sigma_E)\nu}\bigg];
\end{equation}
\item \emph{Corrector result:}
Provided $f\in\Ld^p(U)$ for some $p>d$, for almost all~$\w$, a corrector result holds in the following form
for the velocity field $u_\e^\w$,
\[\quad\bigg\|u_\e^\w-\bar u-\e(1-\lambda)\phi^\w(\tfrac\cdot \e)-\e\sum_{E\in\Ec}\psi_E^\w(\tfrac\cdot\e)\nabla_E\bar u\bigg\|_{H^1(U)}\,\to\,0,\]
and for the pressure field $P_\e^\w$,
\begin{multline*}
\qquad\inf_{\kappa\in\R}\bigg\|P_\e^\w-\tfrac1\e \lambda e \cdot x-\bar P-\bb:\D(\bar u)\\
-(1-\lambda)(\Pi^\w \mathds1_{\R^d\setminus\Ic^\w})(\tfrac\cdot\e)-\sum_{E\in\Ec}(\Sigma_E^\w\mathds1_{\R^d\setminus\Ic^\w})(\tfrac\cdot\e)\nabla_E\bar u-\kappa\bigg\|_{\Ld^2(U\setminus\Ic_\e^\w(U))}\,\to\,0,
\end{multline*}
where $(\phi,\Pi)$ is  the (infinite-volume) ``sedimentation corrector'' of Theorem~\ref{th:main1}.
\qedhere
\end{enumerate}
\end{theor1}
\begin{rem}
We briefly comment on the scaling in~\eqref{eq:Stokes}.
In the force balance, while gravity appears as a bulk term $\e^d e|I_n|$ and the drag force as a surface term $\int_{\e\partial I_n}\sigma(u_\e,P_\e)\nu$, gravity is naturally rescaled by a factor $\frac1\e$ so that both contributions have the same order of magnitude.
This appears as the natural setup for sedimentation experiments: by scaling, it is equivalent to considering a fixed gravity and particles with fixed size and volume fraction in a tank of increasing size.
As gravity is rescaled by a diverging factor $\frac1\e$, it is compensated by a diverging backflow $-\frac1\e \lambda e$ generated by the pressure, which explains why the convergence of the pressure only holds up to this corresponding renormalization.
\end{rem}

\subsection{Extensions}

We mention possible relaxations of the set of general assumptions on the suspension; details are omitted.
\begin{enumerate}[\quad$\bullet$]
\item \emph{Polydisperse suspensions:} Spherical particles $I_{n,L}=B(x_{n,L})$ can be replaced by other bounded shapes, or even by iid bounded random shapes (provided that particle boundaries are uniformly of class~$C^2$).
In the hyperuniform setting, if particles have varying volumes, the condition on the conservation of the number of points in the move-point derivative has naturally to be replaced by a condition on the conservation of the total volume of the particles upon perturbation.
\smallskip
\item \emph{Weakened hardcore condition:} The deterministic hardcore condition in~\ref{Hd} could be relaxed into a lower bound of the type
\[\E\Big[\mathds1_{x_n\in B}\sup_{m:m\ne n}\big(|x_m-x_n|-2\big)^{-r}\Big]\,<\,\infty,\]
for some large enough power $r\ge1$, at the price of appealing more substantially to Meyers type estimates.
\end{enumerate}

\subsection{Heuristic proof: the linear response}\label{sec:heur}
The main difficulty of any rigorous approach to sedimentation is to account for multibody nonlinear hydrodynamic interactions.
In this paragraph, we briefly show how the scalings for the mean settling speed and for velocity fluctuations can be motivated by a formal linear analysis in the dilute regime and we explicitly emphasize the role of hyperuniformity in this simple setting.
This constitutes a reformulation of the formal calculations by Batchelor~\cite{Batchelor-72}, Caflisch and Luke~\cite{Caflisch-Luke} (see also \cite[Section~1.3]{Gloria-19}), and Koch and Shaqfeh~\cite{KS-91}. A heuristic discussion of the nonlinear contribution is postponed to Section~\ref{sec:comp}.

\medskip\noindent
Neglecting the multibody interactions in the dilute regime $\lambda_L\ll1$, the Stokes problem~\eqref{eq:StokesL}--\eqref{eq:StokesL-BC3} is formally reduced to $\phi_L\approx \phi_L^\circ$, see also Remark~\ref{rem:proj}(a),
\begin{equation}\label{eq:linear-approx}
-\triangle\phi_L^\circ+\nabla\Pi_L^\circ=\Big(\sum_n\mathds1_{I_{n,L}}-\lambda_L\Big)e,\qquad\Div\phi_L^\circ=0,\qquad\text{in $Q_L$},
\end{equation}
and particle velocities are approximated by $V_{n,L}\approx V_{n,L}^\circ:=\int_{I_{n,L}}\phi_L^\circ$.
For this simplified linear model, in view of~\eqref{eq:V-energy}, the mean settling speed is explicitly given by
\begin{equation}\label{eq:pre-bnd-barV}
\lambda_L|e|\bar V_L^\circ~~\overset{L\uparrow\infty}\sim~~\expec{|\nabla\phi_L^\circ|^2}\,=\,\E\bigg[{\Big|\sum_n\nabla U_L(x_{n,L})\Big|^2}\bigg]\,=\,\bigg|\!\Var\bigg[{\sum_n\nabla U_L(x_{n,L})}\bigg]\bigg|,
\end{equation}
in terms of the periodic (locally averaged) Stokeslet,
\[-\triangle U_L+\nabla P_L=\big(\mathds1_B-L^{-d}|B|\big)e,\qquad\text{in $Q_L$}.\]
Likewise, velocity fluctuations are formally computed as follows,
\begin{equation}\label{eq:pre-bnd-sigma}
(\sigma_L^\circ)^2\,\approx\,|\var{\phi^\circ_L}\!|=\bigg|\!\Var\bigg[\sum_nU_L(x_{n,L})\bigg]\bigg|.
\end{equation}
In order to estimate~\eqref{eq:pre-bnd-barV} and~\eqref{eq:pre-bnd-sigma}, we distinguish between random point processes $\Pc_L$ without or with long-range order in form of hyperuniformity.

\begin{enumerate}[\quad$\bullet$]
\item \emph{Case without long-range order:}\\
If correlations of the point process $\Pc_L=\{x_{n,L}\}_n$ display an integrable decay~\ref{Mix}, the following variance estimate is easily obtained, cf.~\eqref{eq:var-est00},
\begin{equation}\label{eq:nohyp}
\qquad\Var\bigg[\sum_n\zeta(x_{n,L})\bigg]~\lesssim~\rho_L^2\int_{Q_L}|\zeta|^2.
\end{equation}
Recalling that the decay of the Stokeslet is given by
\[|U_L(x)|\lesssim|e|(1+|x|_L)^{2-d},\qquad|\nabla U_L(x)|\lesssim|e|(1+|x|_L)^{1-d},\]
identities~\eqref{eq:pre-bnd-barV} and~\eqref{eq:pre-bnd-sigma} then lead to the expected scaling,
\begin{eqnarray*}
\qquad\lambda_L|e|\bar V_L^\circ~&\lesssim~\rho_L^2\displaystyle\int_{Q_L}|\nabla U_L|^2~&\lesssim~\rho_L^2|e|^2,\qquad\text{provided $d>2$},\\
\qquad(\sigma_L^\circ)^2~&\lesssim~\rho_L^2\displaystyle\int_{Q_L}|U_L|^2~&\lesssim~\rho_L^2|e|^2,\qquad\text{provided $d>4$}.
\end{eqnarray*}
\smallskip\item \emph{Case with hyperuniformity:}\\
Hyperuniformity is naturally expected to translate into a strong improvement of~\eqref{eq:nohyp}: indeed, given an independent copy $\{x_{n,L}'\}_n$ of $\{x_{n,L}\}_n$, we may represent
\[\qquad\Var\bigg[\sum_n \zeta(x_{n,L})\bigg]\,=\,\E\E'\bigg[\frac12\Big(\sum_n \zeta(x_{n,L})-\sum_n \zeta(x_{n,L}')\Big)^2\bigg],\]
and the suppression of density fluctuations would formally allow to locally couple the point sets~$\{x_{n,L}'\}_n$ and $\{x_{n,L}\}_n$, hence only compare close points, which would ideally translate into the gain of a derivative: for all $\zeta\in C^\infty_\per(Q_L)$ with $\int_{Q_L}\zeta=0$,
\begin{equation}\label{eq:withhyp}
\qquad\Var\bigg[\sum_n \zeta(x_{n,L})\bigg]\,\lesssim\,\rho_L^2\int_{Q_L}|\nabla \zeta|^2.
\end{equation}
As shown in Lemma~\ref{lem:var-lin}, this functional inequality is indeed essentially equivalent to hyperuniformity together with a suitable decay of correlations.
Identities~\eqref{eq:pre-bnd-barV} and~\eqref{eq:pre-bnd-sigma} then lead to the expected improved scaling,
\begin{eqnarray*}
\qquad\lambda_L|e|\bar V_L^\circ~&\lesssim~\rho_L^2\displaystyle\int_{Q_L}|\nabla^2 U_L|^2~&\lesssim~\rho_L^2|e|^2,\qquad\text{for any $d\ge1$},\\
\qquad(\sigma_L^\circ)^2~&\lesssim~\rho_L^2\displaystyle\int_{Q_L}|\nabla U_L|^2~&\lesssim~\rho_L^2|e|^2,\qquad\text{provided $d>2$}.
\end{eqnarray*}
\end{enumerate}

\subsection{Analogies to stochastic homogenization: the nonlinear response}\label{sec:comp}
This section is devoted to
analogies between the sedimentation problem and corrector equations
for linear elliptic operators, both in divergence form or in non-divergence form (cf.~\cite{GO1,GNO1,AKM2,AKM-book,AL-17} e.g.).
In particular, we explain the differences in the corresponding critical dimensions, as well as the surprising fact that the sedimentation problem benefits from hyperuniformity whereas
elliptic corrector equations do not in general.
We first recall the corrector equations for linear elliptic operators both in divergence form, cf.~\eqref{eq:compB}, and in non-divergence form, cf.~\eqref{eq:compC}, associated with a uniformly elliptic random coefficient field~$\Aa_L$; to make the comparison with the sedimentation problem more transparent, we also introduce some hybrid corrector equation, cf.~\eqref{eq:compD}:
\begin{align*}
-\Div(\Aa_L\nabla\psi_{1})&\,=\,\Div(\Aa_L e),\label{eq:compB}\tag{Div}\\
 -\Aa_L:\nabla^2\psi_{2}&\,=\,(\Aa_L-\tilde\Aa_L):E,\label{eq:compC}\tag{NDiv}\\
-\Div(\Aa_L\nabla\psi_{3})&\,=\,(\Aa_L-\expec{\Aa_L}):E,\label{eq:compD}\tag{Hyb}
\end{align*}
for some fixed directions $e \in \R^d$ and $E \in \R^{d\times d}$, and some suitable constant $\tilde\Aa_L\in\R^{d\times d}$ that ensures solvability. In each of these corrector equations, the right-hand side displays a linear random input either in divergence or in non-divergence form, while the nonlinearity with respect to randomness arises from the solution operator associated with the elliptic operator in the left-hand side, which is again either in divergence or in non-divergence form.
We compare these elliptic models with the sedimentation problem~\eqref{eq:StokesL}--\eqref{eq:StokesL-BC3}, which is conveniently rewritten as follows, cf.~Remark~\ref{rem:proj}(a),
\begin{equation*}
\phi_L=\tfrac1{1-\lambda_L}\pi_L\phi_L^\circ,\qquad-\triangle\phi_L^\circ+\nabla\Pi_L^\circ=(\mathds1_{\Ic_L}-\lambda_L)e,\quad\Div\phi_L^\circ=0,\quad\text{in $Q_L$}.
\end{equation*}
The linear random input $\mathds1_{\Ic_L}$ in the equation for $\phi_L^\circ$ is in non-divergence form, and the nonlinearity arises from the projection $\pi_L$. Although the latter has a very different structure from elliptic solution operators, our results in the present work indicate that it behaves quite similarly and displays exactly the same nonlocality as the divergence-form elliptic solution operator $(-\Div\Aa_L\nabla)^{-1}$. In this respect, the sedimentation problem appears comparable to the hybrid model~\eqref{eq:compD}.

\medskip\noindent
We now compare the critical dimensions for the different elliptic models. While a linearized analysis would be misleading when unraveling the role of hyperuniformity, we appeal to sensitivity calculus: we analyze how correctors are modified upon infinitesimal local changes $\delta \Aa_L$ of the random coefficients $\Aa_L$ and we consider both the linear and nonlinear responses.
Formally differentiating~\eqref{eq:compB}, \eqref{eq:compC}, and~\eqref{eq:compD} with respect to~$\Aa_L$ in the direction~$\delta\Aa_L$ yields, respectively,
\begin{eqnarray*}
-\Div(\Aa_L\nabla \delta \psi_{1})&=&\Div(\delta\Aa_L e)+\Div(\delta\Aa_L \nabla \psi_1),
\\
 -\Aa_L:\nabla^2\delta \psi_{2}&=&\delta \Aa_L :E+\delta \Aa_L :\nabla^2\psi_{2},
\\
-\Div(\Aa_L\nabla \delta \psi_{3})&=&\delta \Aa_L:E+\Div(\delta\Aa_L \nabla \psi_3),
\end{eqnarray*}
where in each line the first right-hand side term is the linear response and the second one is the nonlinear response.
We denote by $G_1$ and $G_2$ the Green's functions associated with $-\Div(\Aa_L\nabla)$ and $-\Aa_L:\nabla^2$, respectively, and we recall that they behave on large scales like the Green's function for the Laplacian (up to second mixed derivative for~$G_1$, cf.~\cite{Delmotte-Deuschel-05,MaO,AKM-book,GNO-reg}, and up to first derivative for~$G_2$, cf.~\cite{AL-17}).
In order to assess the locality of the above contributions,
we appeal to Green's representation formula,
\begin{eqnarray}\label{eq:compBG}
\delta \psi_{1}(x)&=&-\int_{Q_L} \nabla_2 G_{1}(x,\cdot) \cdot \delta\Aa_L e-\int_{Q_L} \nabla_2G_{1}(x,\cdot) \cdot  \delta\Aa_L \nabla \psi_1,
\\
 \delta \psi_{2}(x)&=&\int_{Q_L} G_{2} (x,\cdot)\,\delta \Aa_L :E+\int_{Q_L} G_{2} (x,\cdot)\,\delta \Aa_L:\nabla^2\psi_{2},\label{eq:compCG}
\\
\label{eq:compDG}
 \delta \psi_{3}(x)&=&\int_{Q_L} G_{1} (x,\cdot)\,\delta \Aa_L:E-\int_{Q_L} \nabla_2G_{1}(x,\cdot) \cdot  \delta\Aa_L \nabla \psi_3.
\end{eqnarray}
Locality is measured in terms of
the power decay of $\delta\psi_1(x),\delta\psi_2(x),\delta\psi_3(x)$
when the perturbation $\delta\Aa_L$ of the coefficient field is localized in a ball $B(y)$ at a far-away point $y$.
In this informal discussion, we focus on a self-consistency analysis and assume that $\nabla \psi_1$, $\nabla^2 \psi_2$, and $\nabla\psi_3$ are already known to be well-defined stationary objects with bounded moments, while rigorous analysis would require a suitable buckling argument, cf.~Section~\ref{sec:fluc-str}.
First note that the divergence-form structure yields an additional gradient on the Green's function, hence a better locality. Decay is indeed $|x-y|^{1-d}$ for~\eqref{eq:compB} and only $|x-y|^{2-d}$ for~\eqref{eq:compC}, which explains the shift in critical dimensions: combined with suitable functional inequalities, in the spirit of Malliavin calculus, this formally entails that $\psi_1$ has bounded moments in dimension $d>2$, while  $\psi_2$ only has bounded moments in dimension $d>4$.
Next, note that both in~\eqref{eq:compBG} and in~\eqref{eq:compCG} the linear and nonlinear responses display the same locality (Green's functions have the same number of derivatives),
while this is not the case in the hybrid model: the nonlinear response in~\eqref{eq:compDG} has better locality and the scaling is thus dominated by the linear part in non-divergence form. The same property is shown to hold for the original sedimentation problem, cf.~Proposition~\ref{prop:CLTscaling}(i), and this explains in particular why the critical dimension is in general the same as for~\eqref{eq:compC}.

\medskip\noindent
Finally, we investigate the role of hyperuniformity.
As this statistical property consists of the suppression of density fluctuations of $\Aa_L$, this is naturally expressed by restricting to perturbations $\delta\Aa_L$ having vanishing average.
In the spirit of~\eqref{eq:withhyp}, this restriction leads to the gain of a derivative in the following form, for all $F \in C^\infty_\per(Q_L)^{d\times d}$,
$$
\Big| \int_{Q_L} F : \delta\Aa_L\Big|\,=\,\Big|\int_{B(y)} \Big(F - \fint_{B(y)} F\Big):\delta \Aa_L \Big|\, \lesssim \, \|\delta\Aa_L\|_{\Ld^\infty}\int_{B(y)} |\nabla F|.
$$
In the linear terms in~\eqref{eq:compBG}--\eqref{eq:compDG}, this yields an additional derivative on the Green's function, but no such gain appears in the nonlinear terms due to the presence of the stationary factors $\nabla \psi_1$, $\nabla^2 \psi_2$, and $\nabla\psi_3$.
Alternatively, this is understood as follows: fluctuations of elliptic solution operators are known to depend not only on density fluctuations of~$\Aa_L$ (which would be suppressed by hyperuniformity), but also on its geometry (which is unrelated to hyperuniformity), e.g.~\cite[Section~7.3]{JKO94}.
In the models~\eqref{eq:compB} and~\eqref{eq:compC}, the nonlinear responses are therefore limitant and no benefit is expected from hyperuniformity --- except in the particular case of dimension $d=1$ for~\eqref{eq:compB} and of isotropic coefficients $\Aa(x)=\alpha(x) \Id$ for~\eqref{eq:compC}, in which case the nonlinearity with respect to randomness is drastically simplified.
In contrast, in the hybrid model~\eqref{eq:compD}, as the nonlinear response has already better locality, hyperuniformity leads to a nontrivial improvement; this explains why under hyperuniformity the critical dimension for the sedimentation problem as for~\eqref{eq:compD} becomes the same as for~\eqref{eq:compB}.


\section{Mean settling speed}\label{sec:sedim}
This section is devoted to the analysis of the mean settling speed, cf.~Theorem~\ref{th:main1}.
First recall that the flow $\phi_L$ in the periodized tank $Q_L$ is defined via~\eqref{eq:StokesL}--\eqref{eq:StokesL-BC3}, that is, for all $\w\in\Omega$, the velocity field $\phi_L^\w\in H^1_\per(Q_L)^d$ and pressure field $\Pi_L^\w\in\Ld^2_\per(Q_L)$ are the unique periodic solutions of
\begin{equation}\label{eq:StokesL+}
\left\{\begin{array}{ll}
-\triangle \phi_{L}^\w+\nabla \Pi_{L}^\w=-\alpha_L^\w e,&\text{in $Q_L\setminus\Ic_L^\w$},\\
\Div \phi_{L}^\w=0,&\text{in $Q_L\setminus\Ic_L^\w$},\\
\D(\phi_{L}^\w)=0,&\text{in $\Ic_L^\w$},\\
e|I_{n,L}^\w|+\int_{\partial I_{n,L}^\w}\sigma(\phi_{L}^\w,\Pi_{L}^\w)\nu=0,&\forall n,\\
\int_{\partial I_{n,L}^\w}\Theta\nu\cdot\sigma(\phi_{L}^\w,\Pi_{L}^\w)\nu=0,&\forall n,\,\forall\Theta\in\Md^\Skew,
\end{array}\right.
\end{equation}
with vanishing average $\int_{Q_L}\phi_{e,L}^\w=0$ and $\int_{Q_L\setminus\Ic_L^\w}\Pi_L^\w=0$, where solvability imposes that the constant backflow $-\alpha_L^\w e$ be given by
\[\alpha_L^\w=\frac{\lambda_L^\w}{1-\lambda_L^\w},\qquad\lambda_L^\w=L^{-d}|\Ic_L^\w|.\]
Note that the hardcore assumption in~\ref{Hd} ensures $\alpha_L^\w\lesssim1$.
We may now proceed to the proof of Theorem~\ref{th:main1}.

\begin{proof}[Proof of Theorem~\ref{th:main1}]
We split the proof into three steps.
After some general preparation in the first step, the second step is devoted to the analysis of the periodized problem~\eqref{eq:StokesL+} and the proof of the main estimates.
In the last step, borrowing arguments from~\cite{DG-19}, we argue that the infinite-volume limit is well-defined and that the corresponding Stokes problem~\eqref{eq:cor-sed} is well-posed.

\medskip
\step1 Reformulation of the equation: we prove that equation~\eqref{eq:StokesL+} for $\phi_{L}^\w$ yields in the weak sense on the whole periodic cube $Q_L$,
\begin{equation}\label{eq:reform-sedim}
-\triangle \phi_{L}^\w+\nabla(\Pi_{L}^\w\mathds1_{Q_L\setminus\Ic_L^\w})=-\alpha_L^\w e\mathds1_{Q_L\setminus\Ic_L^\w}-\sum_n\delta_{\partial I_{n,L}^\w}\sigma(\phi_{L}^\w,\Pi_{L}^\w)\nu,
\end{equation}
where $\delta_{\partial I_{n,L}^\w}$ stands for the Dirac measure on the boundary of $I_{n,L}^\w$.

\medskip
\noindent
Given $\zeta\in C^\infty_\per(Q_L)^d$, testing equation~\eqref{eq:StokesL+} with $\zeta$ and integrating by parts on $Q_L\setminus\Ic_L^\w$, we find
\begin{multline}\label{eq:reform-prepre}
\int_{Q_L\setminus\Ic_L^\w}\nabla\zeta:\nabla \phi_L^\w - \int_{Q_L\setminus\Ic_L^\w}(\Div\zeta)\,\Pi_L^\w\\
\,=\,-\alpha_L^\w e\cdot\int_{Q_L\setminus\Ic_L^\w}\phi_L^\w-\sum_{n}\int_{\partial I_{n,L}^\w}(\zeta\otimes\nu):(\nabla\phi_L^\w-\Pi_L^\w\Id).
\end{multline}
The claim~\eqref{eq:reform-sedim} would follow provided we prove that
\begin{equation}\label{ag+1}
\int_{\Ic_L^\w}\nabla\zeta:\nabla \phi_L^\w \,=\,-\sum_{n}\int_{\partial I_{n,L}^\w}(\nu \otimes \zeta):\nabla\phi_L^\w.
\end{equation}
Indeed, adding~\eqref{ag+1} to~\eqref{eq:reform-prepre} yields the claim~\eqref{eq:reform-sedim} in view of
\[\int_{\partial I_{n,L}^\w}(\nu \otimes \zeta+\zeta\otimes\nu):\nabla\phi_L^\w=\int_{\partial I_{n,L}^\w}\zeta\otimes\nu:2\D(\phi_L^\w).\]
We turn to the proof of~\eqref{ag+1}.
Using that $\phi_L^\w$ is affine in $I_{n,L}^\w$, we obtain for all $n$,
\[\int_{\partial I_{n,L}^\w}(\nu\otimes\zeta):\nabla\phi_L^\w=\int_{\partial I_{n,L}^\w}\zeta_i\nu\cdot\nabla_i\phi_L^\w=\int_{I_{n,L}^\w}\Div(\zeta_i\nabla_i\phi_L^\w)=\int_{I_{n,L}^\w}\nabla\zeta_i\cdot\nabla_i\phi_L^\w.\]
Since $\D(\phi_L^\w)=0$ on $I_{n,L}^\w$, we can write $\phi_L^\w=V_{n,L}^\w+\Theta_{n,L}^\w(x-x_{n,L}^\w)$ on $I_{n,L}^\w$ for some $V_{n,L}^\w\in\R^d$ and $\Theta_{n,L}^\w\in\Md^\Skew$, so that the above becomes
\[\int_{\partial I_{n,L}^\w}(\nu\otimes\zeta):\nabla\phi_L^\w=\int_{I_{n,L}^\w}\nabla\zeta_i\cdot\nabla_i(\Theta_{n,L}^\w x)=\int_{\partial I_{n,L}^\w}\nu\cdot(\zeta\cdot\nabla)(\Theta_{n,L}^\w x)=\int_{\partial I_{n,L}^\w}\nu\cdot\Theta_{n,L}^\w \zeta.\]
Likewise, we find
\[\int_{I_{n,L}^\w}\nabla\zeta:\nabla\phi_L^\w=\int_{I_{n,L}^\w}\nabla\zeta:\nabla(\Theta_{n,L}^\w x)=\int_{\partial I_{n,L}^\w}\zeta\cdot(\nu\cdot\nabla)(\Theta_{n,L}^\w x)=\int_{\partial I_{n,L}^\w}\zeta\cdot\Theta_{n,L}^\w\nu.\]
By skew-symmetry of $\Theta_{n,L}^\w$, this yields~\eqref{ag+1}, hence~\eqref{eq:reform-sedim}.

\medskip
\step2 Bounds on periodized problem \eqref{eq:StokesL+}: we establish the identity
\begin{equation}\label{eq:def-V}
\frac1{\sharp\Pc_L^\w}\sum_n\tfrac{e}{|e|}\cdot V_{n,L}^\w=\frac1{\alpha_L^\w|e|}\fint_{Q_L}|\nabla\phi_{L}^\w|^2,
\end{equation}
and show that under the mixing assumption~\ref{Mix} there holds
\begin{equation}\label{eq:bound0-mix}
\|\nabla\phi_L\|_{\Ld^2(\Omega)}\lesssim \rho_L|e|\times\left\{\begin{array}{lll}
1&:&d>2;\\
(\log L)^\frac12&:&d=2;\\
L^\frac12&:&d=1;\\
\end{array}\right.
\end{equation}
whereas under the mixing and hyperuniformity assumption~\ref{Hyp} this is improved  in all dimensions to
\begin{equation}\label{eq:bound0-hyp}
\|\nabla\phi_L\|_{\Ld^2(\Omega)}\lesssim \rho_L|e|.
\end{equation}

\medskip\noindent
Testing the reformulation~\eqref{eq:reform-sedim} of the equation for $\phi_L^\w$ with $\phi_L^\w$ itself and using the divergence-free condition for $\phi_L^\w$, we obtain
\[\int_{Q_L}|\nabla\phi_{L}^\w|^2=-\alpha_L^\w e\cdot\int_{Q_L\setminus\Ic_L^\w} \phi_{L}^\w-\sum_n\int_{\partial I_{n,L}^\w}\phi_L^\w\cdot\sigma(\phi_L^\w,\Pi_L^\w)\nu.\]
Since $\D(\phi_L^\w)=0$ on $I_{n,L}^\w$, we can write $\phi_L^\w=V_{n,L}^\w+\Theta_{n,L}^\w(x-x_{n,L}^\w)$ on $I_{n,L}^\w$ for some $V_{n,L}^\w\in\R^d$ and $\Theta_{n,L}^\w\in\Md^\Skew$, so that the boundary conditions for $\phi_L^\w$ lead to
\begin{multline*}
\int_{\partial I_{n,L}^\w}\phi_L^\w\cdot\sigma(\phi_L^\w,\Pi_L^\w)\nu\,=\,V_{n,L}^\w\cdot\int_{\partial I_{n,L}^\w}\sigma(\phi_L^\w,\Pi_L^\w)\nu+\int_{\partial I_{n,L}^\w}\Theta_{n,L}^\w\nu\cdot\sigma(\phi_L^\w,\Pi_L^\w)\nu\\
\,=\,-e\cdot |I_{n,L}^\w|V_{n,L}^\w\,=\,-e\cdot\int_{I_{n,L}^\w}\phi_{L}^\w,
\end{multline*}
and the above becomes
\[\int_{Q_L}|\nabla\phi_{L}^\w|^2=-\alpha_L^\w e\cdot\int_{Q_L\setminus\Ic_L^\w} \phi_{L}^\w+\sum_ne\cdot\int_{I_{n,L}^\w}\phi_{L}^\w.\]
Since $\int_{Q_L}\phi_{L}^\w=0$, this energy identity takes the form
\begin{equation}\label{eq:ident-energy-surpr}
\int_{Q_L}|\nabla\phi_{L}^\w|^2=(1+\alpha_L^\w)\sum_ne\cdot\int_{I_{n,L}^\w}\phi_{L}^\w.
\end{equation}
In terms of particle velocities $V_{n,L}^\w=\fint_{I_{n,L}^\w}\phi_{L}^\w$, cf.~\eqref{eq:def-V-phi}, noting that $(1+\alpha_L^\w)|B|\sharp\Pc_L^\w=L^d\alpha_L^\w$, this turns into the claim~\eqref{eq:def-V}.

\medskip\noindent
Next, recalling $\int_{Q_L}\phi_L^\w=0$ and integrating by parts, the energy identity~\eqref{eq:ident-energy-surpr} is alternatively written as
\begin{eqnarray}
\int_{Q_L}|\nabla\phi_{L}^\w|^2&=& (1+\alpha_L^\w) \int_{Q_L}\phi_{L}^\w \cdot \Big(\sum_n e(\mathds1_{I_{n,L}^\w}-L^{-d}|B|)\Big)\nonumber
\\
&=&(1+\alpha_L^\w)\int_{Q_L}\nabla\phi_{L}^\w:\Big(e\otimes\sum_n\nabla(-\triangle)^{-1}(\mathds1_{I_{n,L}^\w}-L^{-d}|B|)\Big).\label{eq:weak-strong}
\end{eqnarray}
Hence, by Cauchy-Schwarz' inequality,
\begin{equation*}
\int_{Q_L}|\nabla\phi_{L}^\w|^2\le(1+\alpha_L^\w)^2|e|^2\int_{Q_L}\Big|\sum_n\nabla\triangle^{-1}(\mathds1_{I_{n,L}^\w}-L^{-d}|B|)\Big|^2,
\end{equation*}
so that we find, by the hardcore assumption in the form $\alpha_L^\w\lesssim1$ and by stationarity of $\Pc_L$,
\begin{equation*}
\|\nabla\phi_{L}\|_{\Ld^2(\Omega)}^2\lesssim|e|^2\,\E\bigg[\Big|\sum_n\nabla\triangle^{-1}(\mathds1_{I_{n,L}}-L^{-d}|B|)\Big|^2\bigg].
\end{equation*}
Denoting by $G_L$ the Green's function for the Laplacian on the periodic cell $Q_L$, that is, the unique distributional solution of
$-\triangle G_L=\delta-L^{-d}$ on $Q_L$, and setting $F_L(x):=\int_{B(x)}\nabla G_L$, we may rewrite the above as
\begin{equation*}
\|\nabla\phi_{L}\|_{\Ld^2(\Omega)}^2\,\lesssim\, |e|^2\,\E\bigg[{\Big|\sum_nF_L(x_{n,L})\Big|^2}\bigg]
\,=\,|e|^2\,\bigg|\!\Var\bigg[\sum_nF_L(x_{n,L})\bigg]\bigg|,
\end{equation*}
where the last equality follows from noting that $\expec{\,\sum_nF_L(x_{n,L})}=\rho_L\int_{Q_L}F_L=0$.
Under the mixing assumption~\ref{Mix}, in terms of the pair correlation function $g_{2,L}$, we then deduce
\begin{eqnarray*}
\|\nabla\phi_{L}\|_{\Ld^2(\Omega)}^2
&\lesssim&\rho_L^2|e|^2\bigg|\iint_{Q_L\times Q_L}F_L(x)F_L(y)\,g_{2,L}(x-y)\,dxdy\bigg|\\
&\le&\rho_L^2|e|^2\Big(\int_{Q_L}|F_L|^2\Big)\Big(\int_{Q_L}|g_{2,L}|\Big)\,\lesssim\,\rho_L^2|e|^2\int_{Q_L}|F_L|^2,
\end{eqnarray*}
and the claim~\eqref{eq:bound0-mix} follows from the standard decay of the Green's function $G_L$
in form of $|F_L(x)|\lesssim (1+|x|_L)^{1-d}$.
Under the hyperuniformity assumption~\ref{Hyp}, we rather appeal to the variance estimate~\eqref{eq:withhyp}, cf.~Lemma~\ref{lem:var-lin}, to the effect of
\begin{equation*}
\|\nabla\phi_{L}\|_{\Ld^2(\Omega)}^2
\,\lesssim\,\rho_L^2|e|^2\int_{Q_L}|\nabla F_L|^2,
\end{equation*}
and the claim~\eqref{eq:bound0-hyp} now follows from the standard decay of the Green's function $G_L$ in form of $|\nabla F_L(x)|=|\int_{B(x)} \nabla^2 G_L|\,\lesssim\, (1+|x|_L)^{-d}$.

\medskip
\step3 Infinite-volume limit:
under~\ref{Mix} for $d>2$ or under~\ref{Hyp} for any $d\ge 1$,
we argue that $\nabla\phi_L$ converges strongly in $\Ld^2(\Omega)$ as $L\uparrow\infty$ to the unique gradient solution $\nabla\phi$ of the corresponding infinite-volume problem~\eqref{eq:cor-sed}.

\nopagebreak\medskip\noindent
Uniqueness for~\eqref{eq:cor-sed} is already contained in~\cite{DG-19} since the difference of two solutions of~\eqref{eq:cor-sed} is a solution of~\eqref{eq:cor-sed} without gravity. It only remains to establish the strong convergence result. For that purpose, in terms of
\[\gamma_L^\w:=\sum_n\nabla(-\triangle)^{-1}(\mathds1_{I_{n,L}^\w}-L^{-d}|B|),\]
we appeal to identity~\eqref{eq:weak-strong} in the form
\[\expec{|\nabla\phi_{L}|^2}\,=\,\E\Big[(1+\alpha_L)\fint_{Q_L}\nabla\phi_{L} : (e\otimes \gamma_L)\Big],\]
or equivalently, by stationarity of $\nabla\phi_{L} : (e\otimes \gamma_L)$ and by invariance of $\alpha_L$ with respect to shifts on the periodic cell (shifting $\Pc_L$ does indeed not change the number of points),
\begin{equation}\label{eq:idesw}
\expec{|\nabla\phi_{L}|^2}\,=\,\E\big[(1+\alpha_L)\nabla\phi_{L} : (e\otimes \gamma_L)\big].
\end{equation}
The stabilization condition for $\Pc_L$ in~\ref{Hd} ensures that $\gamma_L$ converges strongly in $\Ld^2(\Omega)$ to $\gamma := \sum_n \nabla(-\triangle)^{-1}\mathds1_{I_{n}}$,
which is indeed well-defined under~\ref{Mix} for $d>2$ or under~\ref{Hyp} for any $d\ge 1$. Since the bounded random variable $\alpha_L$ converges almost surely to $\alpha$,
we conclude that $(1+\alpha_L)  (e\otimes \gamma_L)$ converges strongly to $(1+\alpha)(e\otimes \gamma)$ in $\Ld^2(\Omega)$. Therefore, identity~\eqref{eq:idesw} entails that the strong convergence $\nabla \phi_L \to \nabla \phi$ in $\Ld^2(\Omega)$ follows from the corresponding weak convergence.

\medskip\noindent
Using the uniform bound~\eqref{eq:bound0-mix} or~\eqref{eq:bound0-hyp} of Step~1, weak compactness in $\Ld^2(\Omega;\Ld^2_\loc(\R^d))$ ensures that $\nabla\phi_L$ converges weakly to some stationary random field $\Phi \in \Ld^2(\Omega;\Ld^2_\loc(\R^d)^{d\times d})$ (along a subsequence, not relabelled).
As a weak limit of gradients, $\Phi$ is necessarily gradient-like, hence we may write
$\Phi^\w(x)=\nabla \phi^\omega (x)$ for some $\phi\in\Ld^2(\Omega;H^1_\loc(\R^d)^d)$.
Combining this with the stabilization condition in~\ref{Hd} in form of $\mathds1_{\Ic_L^\w}\to \mathds1_{\Ic^\w}$ in $\Ld^2_\loc(\R^d)$ for almost all $\w$,
we may then pass to the limit in the weak formulation of equation~\eqref{eq:StokesL+}. More precisely, we first easily deduce for almost all $\w$ that $\D(\phi^\w)=0$ in $\Ic^\w$ and that $\Div\phi^\w=0$. Next, a similar standard argument as in~\cite[Step~3 of proof of Proposition~2.1]{DG-19} allows to deduce that $\phi$ is a solution of~\eqref{eq:cor-sed} in the following weak sense: for almost all $\w$, for all test functions $\psi\in C^\infty_c(\R^d)^d$ with $\D(\psi)=0$ on $\Ic^\w$ and $\Div\psi=0$, there holds
\[\int_{\R^d}\nabla\psi:\nabla\phi^\w=-\alpha e\cdot\int_{\R^d\setminus\Ic^\w}\psi+e\cdot\int_{\Ic^\w}\psi.\]
Finally, arguing as in~\cite[Step~4 of proof of Proposition~2.1]{DG-19}, a stationary pressure $\Pi^\omega$ can be reconstructed with vanishing expectation and finite second moments, while the regularity theory for the steady Stokes equation ensures that~$(\phi^\w,\Pi^\w)$ is in fact a classical solution of~\eqref{eq:cor-sed} and that boundary conditions are satisfied in a pointwise sense.
\end{proof}


\medskip
\section{Velocity fluctuations}\label{sec:fluc}
This section is devoted to the proof of Theorem~\ref{th:main2}, that is, the estimation of fluctuations of individual particle velocities, which requires a fine analysis of stochastic cancellations.
The proof is particularly demanding and strongly relies on a novel annealed regularity theory for the steady Stokes equation with a random suspension, which is briefly described in Section~\ref{sec:reg0} and mainly postponed to a forthcoming companion contribution~\cite{DG-20+}.
Interestingly, the hyperuniform setting is easier to treat as it only requires a perturbative regularity result.
We also strongly rely on local regularity statements and pressure estimates borrowed from our previous work~\cite{DG-19} on colloidal (non-sedimenting) suspensions.
By scaling, we may henceforth assume $|e|=1$.

\subsection{Structure of the proof}\label{sec:fluc-str}
We start with the following key estimates on the optimal decay of large-scale averages of the gradient field $\nabla\phi_{L}$.
Due to the nonlinearity with respect to randomness, local norms of $\nabla\phi_L$ also appear in the right-hand side; this will be subsequently absorbed by a buckling argument. Yet, the correct fluctuation scaling is already manifest:
in particular, the first term in~\eqref{eq:mix-res-scal} below displays the CLT scaling multiplied by a length scale, cf.~\eqref{eq:fluc-scaling-1}, and is the reason why the critical dimension in Theorem~\ref{th:main2}(i) is $d=4$ instead of $d=2$.
Importantly, the nonlinear contribution of $\nabla\phi_L$ in the right-hand side is multiplied by the CLT scaling without loss, which is key to our buckling argument to prove Theorem~\ref{th:main2}(i): the worse nonlocality only appears in the linear part, while the nonlinear part always behaves as in homogenization for divergence-form linear elliptic equations, cf.~Section~\ref{sec:comp}.
In the hyperuniform setting~(ii), the suppression of density fluctuations exactly allows to avoid the worse scaling of the linear part, cf.~Section~\ref{sec:comp}, and we recover the CLT scaling~\eqref{eq:fluc-scaling-2}.

\begin{prop}[Fluctuation scaling]\label{prop:CLTscaling}
Let the random point processes $(\Pc_L)_{L\ge1}$ satisfy the general assumptions~\ref{Hd} for some $\delta>0$.
\begin{enumerate}[(i)]
\item Under the improved mixing assumption~\ref{Mix+},
in dimension $d>2$, there holds for all $g\in C^\infty_\per(Q_L)^{d\times d}$, $1\le R\le L$, $q\ge1$, and $1\ll p<\infty$,
\begin{multline}\label{eq:mix-res-scal}
\qquad\Big\|\int_{Q_L}g:\nabla \phi_{L}\Big\|_{\Ld^{2p}(\Omega)}^{2}
\,\lesssim_p\,\|g\|_{\Ld^{\frac{2d}{d+2}}(Q_L)}^2 \wedge\, \|\langle\cdot\rangle g\|_{\Ld^{2}(Q_L)}^2\\
+\|\avoir g\|_{\Ld^2(Q_L)}^2\,\bigg\|\Big(1+\int_{B_{R}}[\nabla\phi_{L}]_2^{2q}\Big)^\frac1q\bigg\|_{\Ld^p(\Omega)}.
\end{multline}
\item Under the improved mixing and hyperuniformity assumption~\ref{Hyp+}, in any dimension $d\ge1$, there holds for all $g\in C^\infty_\per(Q_L)^{d\times d}$, $1\le R\le L$, $q\ge1$, and~$1\ll p<\infty$,
\begin{equation}\label{eq:hyp-res-scal}
\qquad\Big\|\int_{Q_L}g:\nabla \phi_{L}\Big\|_{\Ld^{2p}(\Omega)}^{2}
\,\lesssim_p\,\|\avoir g\|_{\Ld^2(Q_L)}^2\,\bigg\|\Big(1+\int_{B_{R}}[\nabla\phi_{L}]_2^{2q}\Big)^\frac1q\bigg\|_{\Ld^p(\Omega)}.\qedhere
\end{equation}
\end{enumerate}
\end{prop}

In preparation for a buckling argument, the following result allows to bound local norms of $\nabla\phi_L$ as appearing in the right-hand sides of~\eqref{eq:mix-res-scal}--\eqref{eq:hyp-res-scal} by corresponding large-scale averages.
This statement is inspired by~\cite{Otto-Tlse} in the context of homogenization for divergence-form linear elliptic equations, and constitutes a compact improved version of~\cite[Lemma~2.2]{Gloria-19}.

\begin{prop}\label{prop:interpol}
Let the random point processes $(\Pc_L)_{L\ge1}$ satisfy the general assumptions~\ref{Hd} for some $\delta>0$.
Choose $\chi\in C^\infty_c(B)$ with $\int_B\chi=1$ and set $\chi_r(x):=r^{-d}\chi(\frac xr)$.
There exists $\eta_0>0$ (only depending on $d,\delta$) such that there holds for all $1\le r\ll_\chi R\le L$, $1\le q\le1+\eta_0$, and~$p\ge1$,
\[\bigg\|\Big(\fint_{B_{R}}|\nabla \phi_{L}|^{2q}\Big)^\frac1q\bigg\|_{\Ld^p(\Omega)}\,\lesssim_\chi\, R^{2}+\Big\|\int_{Q_L}\chi_r\nabla\phi_{L}\Big\|_{\Ld^{2p}(\Omega)}^2.\qedhere\]
\end{prop}

We are now in position to prove Theorem~\ref{th:main2}.
Based on the above two propositions together with a buckling argument, we first deduce moment bounds on $\nabla\phi_L$.
Combining this again with Proposition~\ref{prop:CLTscaling}, we  deduce the optimal fluctuation scaling for large-scale averages of $\nabla\phi_L$, cf.~\eqref{eq:fluc-scaling-1}--\eqref{eq:fluc-scaling-2}.
Finally, moment bounds on the velocity field $\phi_L$ simply follow by integration.

\begin{proof}[Proof of Theorem~\ref{th:main2}]
By local regularity for the steady Stokes equation, e.g.~\cite[Section~IV]{Galdi}, we have
\begin{equation}\label{eq:loc-reg00}
\sup_{B(x)} |\nabla \phi_L^\w| \lesssim \Big(\fint_{B_2(x)} |\nabla \phi_L^\w|^2\Big)^\frac12,\qquad \sup_{B(x)} |\phi_L^\w| \,\lesssim\, \Big( \fint_{B_2(x)} |\phi_L^\w|^2+|\nabla \phi_L^\w|^2\Big)^\frac12,
\end{equation}
so that it is enough to control moments of local quadratic averages $[\nabla\phi_L]_2$ and $[\phi_L]_2$; we omit the detail.
We split the proof into three steps.

\medskip
\step1 Proof that for all $1\le R\le L$,
\[\|[\nabla\phi_{L}]_2\|_{\Ld^{2p}(\Omega)}^2\lesssim(R^d)^{1-\frac1p}\Big\|\fint_{B_{R}}|\nabla\phi_{L}|^2\Big\|_{\Ld^p(\Omega)}.\]
By the discrete $\ell^1-\ell^p$ inequality in form of the  reverse Jensen's inequality
\[\Big(\fint_{B_R(x)}[\nabla\phi_{L}]_2^{2p}\Big)^\frac1p\lesssim (R^d)^{1-\frac1p}\fint_{B_{2R}(x)}|\nabla \phi_{L}|^2,\]
the claim follows in combination with stationarity of $[\nabla \phi_L]_2$.

\medskip
\step2 Moment bounds on $\nabla\phi_L$: proof of~\eqref{eq:nablaphi-bnd} and~\eqref{eq:nablaphi-bnd+}.\\
Under~\ref{Mix+} in dimension $d>2$, combining the results of Propositions~\ref{prop:CLTscaling}(i) and~\ref{prop:interpol}, we find for all $1\le r\ll_\chi R\le L$, $1<q\le1+\eta_0$, and $1\ll p<\infty$,
\begin{eqnarray*}
\bigg\|\Big(\fint_{B_{R}}|\nabla \phi_{L}|^{2q}\Big)^\frac1q\bigg\|_{\Ld^p(\Omega)}&\lesssim_\chi& R^{2}+\Big\|\int_{Q_L}\chi_r\nabla\phi_{L}\Big\|_{\Ld^{2p}(\Omega)}^2\\
&\lesssim_{p,\chi}& R^{2}+r^{2-d}+\big(\tfrac Rr\big)^dR^{-\frac d{q'}}\bigg\|\Big(1+\fint_{B_R}[\nabla\phi_{L}]_2^{2q}\Big)^\frac1q\bigg\|_{\Ld^{p}(\Omega)},
\end{eqnarray*}
which yields after optimization in $r$, for $R\gg1$ and $q>1$,
\[\bigg\|\Big(\fint_{B_{R}}|\nabla \phi_{L}|^{2q}\Big)^\frac1q\bigg\|_{\Ld^p(\Omega)}\lesssim_p R^{2},\]
and the conclusion~\eqref{eq:nablaphi-bnd} then follows from the result of Step~1.

\medskip\noindent
Under~\ref{Hyp+} in any dimension $d\ge1$, combining the results of Propositions~\ref{prop:CLTscaling}(ii) and~\ref{prop:interpol}, we rather find for all $1\le r\ll R\le L$, $1\le q<1+\eta_0$, and $1\ll p<\infty$,
\begin{equation*}
\bigg\|\Big(\fint_{B_{R}}|\nabla \phi_{L}|^{2q}\Big)^\frac1q\bigg\|_{\Ld^p(\Omega)}
\,\lesssim_p\, R^{2}+\big(\tfrac Rr\big)^{d}R^{-\frac d{q'}}\bigg\|\Big(1+\fint_{B_R}[\nabla\phi_{L}]_2^{2q}\Big)^\frac1q\bigg\|_{\Ld^{p}(\Omega)},
\end{equation*}
and the conclusion~\eqref{eq:nablaphi-bnd+} follows in the same way.

\medskip
\step3 Moment bounds on $\phi_L$: proof of~\eqref{eq:phi-bnd} and~\eqref{eq:phi-bnd+}.\\
Poincar\'e's inequality yields
\begin{equation}\label{eq:phi-bnd0}
\bigg\|\Big[\phi_{L}-\fint_{B}\phi_{L}\Big]_2(x)\bigg\|_{\Ld^{2p}(\Omega)}\lesssim\|[\nabla\phi_{L}]_2\|_{\Ld^{2p}(\Omega)}+\Big\|\fint_{B(x)}\phi_{L}-\fint_{B}\phi_{L}\Big\|_{\Ld^{2p}(\Omega)},
\end{equation}
and in view of~\eqref{eq:nablaphi-bnd} and~\eqref{eq:nablaphi-bnd+} it remains to estimate the last right-hand side term.
For that purpose, we write
\[\fint_{B(x)}\phi_{L}-\fint_{B}\phi_{L}=\int_{Q_L}\nabla\phi_{L}\cdot\nabla h_{x,L},\]
in terms of $\nabla h_{x,L}:=\nabla h_L(\cdot -x)-\nabla h_L$, where $h_L$ denotes the unique solution in $Q_L$ of
\[-\triangle h_L=\frac{\mathds1_{B}}{|B|}-L^{-d}.\]
Under~\ref{Mix+} in dimension $d>2$, appealing to Proposition~\ref{prop:CLTscaling}(i) together with~\eqref{eq:nablaphi-bnd} yields for all $p\ge1$,
\begin{multline*}
\Big\|\int_{Q_L}\nabla\phi_{L}\cdot\nabla h_{x,L}\Big\|_{\Ld^{2p}(\Omega)}
\\
\lesssim_p \|\nabla h_{x,L}\|_{\Ld^\frac{2d}{d+2}(Q_L)}
\wedge\, \|\langle\cdot\rangle\nabla h_{x,L}\|_{\Ld^2(Q_L)}
+\|\avoir \nabla h_{x,L}\|_{\Ld^2(Q_L)}.
\end{multline*}
Noting that $\|\nabla^2 h_{x,L}\|_{\Ld^2(Q_L)}\lesssim1$ and that for $d>2$ Riesz potential theory further yields $\|\nabla h_{x,L}\|_{\Ld^2(Q_L)}\lesssim1$,
we are reduced to
\[\Big\|\int_{Q_L}\nabla\phi_{L}\cdot\nabla h_{x,L}\Big\|_{\Ld^{2p}(\Omega)}\lesssim_p 1+\|\nabla h_{x,L}\|_{\Ld^{\frac{2d}{d+2}}(Q_L)} \wedge\, \|\langle\cdot\rangle\nabla h_{x,L}\|_{\Ld^2(Q_L)},\]
while a direct computation with Green's kernel gives
\begin{equation}\label{e.ant-need-more}
\|\nabla h_{x,L}\|_{\Ld^{\frac{2d}{d+2}}(Q_L)} \wedge\, \|\langle\cdot\rangle\nabla h_{x,L}\|_{\Ld^2(Q_L)}\lesssim\begin{cases}
1,&\text{if $d>4$};\\
\log(2+|x|)^\frac12,&\text{if $d=4$};\\
\langle x\rangle^\frac12,&\text{if $d=3$}.
\end{cases}
\end{equation}
Inserting this into~\eqref{eq:phi-bnd0}, the conclusion~\eqref{eq:phi-bnd} follows.
Under~\ref{Hyp+}, rather appealing to Proposition~\ref{prop:CLTscaling}(ii), the conclusion~\eqref{eq:phi-bnd+} follows in the same way.
\end{proof}

\subsection{Preliminaries}
We introduce a number of general tools that play an important role in the proof of Propositions~\ref{prop:CLTscaling} and~\ref{prop:interpol}.

\subsubsection{Multiscale inequalities for higher moments}
The following shows that the multiscale variance inequalities in assumptions~\ref{Mix+} and~\ref{Hyp+} imply corresponding functional inequalities for higher moments.
This is proven in~\cite[Proposition~1.10(ii)]{DG1} for~\ref{Mix+}, and the same proof applies for~\ref{Hyp+}.

\begin{lem}\label{lem:p}
If the random point processes $\{\Pc_L\}_{L\ge 1}$ satisfy~\ref{Mix+}, then there holds for all $p\ge 1$ and all $\sigma(\Pc_L)$-measurable random variables~$Y(\Pc_L)$
with $\expec{Y(\Pc_L)}=0$,
\begin{equation}\label{eq:SGL-p}
\expec{|Y(\Pc_L)|^{2p}}^\frac1{p}\,\lesssim \,p^2 \,\E\bigg[{\int_0^L \bigg(\int_{Q_L}\Big(\partial^{\operatorname{osc}}_{\Pc_L,B_\ell(x)}Y(\Pc_L)\Big)^2dx\bigg)^p \langle\ell\rangle^{-dp}\pi(\ell)\,d\ell}\bigg]^\frac1{p}.
\end{equation}
Likewise, if $\{\Pc_L\}_{L\ge 1}$ satisfy~\ref{Hyp+}, then there holds for all $p\ge 1$ and all $\sigma(\Pc_L)$-measurable random variables~$Y(\Pc_L)$
with $\expec{Y(\Pc_L)}=0$,
\begin{equation}\label{eq:SGL-hyp-p}
\expec{|Y(\Pc_L)|^{2p}}^\frac1{p}\,\lesssim \,p^2 \,\E\bigg[{\int_0^L \bigg(\int_{Q_L}\Big(\partial^{\operatorname{hyp}}_{\Pc_L,B_\ell(x)}Y(\Pc_L)\Big)^2dx\bigg)^p \langle\ell\rangle^{-dp}\pi(\ell)\,d\ell}\bigg]^\frac1{p}.\qedhere
\end{equation}
\end{lem}

\subsubsection{Annealed regularity theory}\label{sec:reg0}
Another key tool consists of annealed regularity properties for the steady Stokes equation with a random colloidal (non-sedimenting) suspension.
More precisely, given a random forcing $g\in\Ld^\infty(\Omega;C^\infty_\per(Q_L)^{d\times d})$, we consider the unique solution $v_L\in\Ld^\infty(\Omega;H^1_\per(Q_L))$ of the following heterogeneous problem, for almost all $\w$,
\begin{equation}\label{eq:test-v0}
\left\{\begin{array}{ll}
-\triangle v_L^\w+\nabla P_L^\w=\Div g^\w,&\text{in $Q_L\setminus\Ic_L^\w$},\\
\Div v_L^\w=0,&\text{in $Q_L\setminus\Ic_L^\w$},\\
\D(v_L^\w)=0,&\text{in $\Ic^\w_L$},\\
\int_{\partial I_{n,L}^\w}\big(g^\w+\sigma(v_L^\w,P_L^\w)\big)\nu=0,&\forall n,\\
\int_{\partial I_{n,L}^\w}\Theta\nu\cdot\big(g^\w+\sigma(v_L^\w,P_L^\w)\big)\nu=0,&\forall n,\,\forall\Theta\in\Md^\Skew,
\end{array}\right.
\end{equation}
with $\int_{Q_L}v_L^\w=0$.
The energy inequality takes the form
\begin{equation}\label{eq:energy}
\|\nabla v_L^\w\|_{\Ld^2(Q_L)}\le\|g^\w\|_{\Ld^2(Q_L\setminus\Ic_L^\w)}.
\end{equation}
Aside from perturbative Meyers type estimates, no corresponding (deterministic) $\Ld^p$ estimate is expected to hold in general due to heterogeneities --- unless particles are assumed to be sufficiently far apart, cf.~Remark~\ref{rem:reg-hofer} below.
However, in view of homogenization, the heterogeneous Stokes operator can be replaced on large scales by the following ``homogenized'' one, cf.~\cite[Theorem~1]{DG-19},
\begin{equation*}
\left\{\begin{array}{ll}
-\nabla\cdot\bar\Bb\nabla \bar v_L^\w+\nabla\bar P_L^\w=\Div g^\w,&\text{in $Q_L$},\\
\Div\bar v_L^\w=0,&\text{in $Q_L$},
\end{array}\right.
\end{equation*}
for which standard constant-coefficient elliptic regularity theory is available. In this spirit, compared to a generic situation, the solution of~\eqref{eq:test-v0} is expected to have much better regularity when the suspension is sampled by an ergodic ensemble.
This type of result was pioneered by Avellaneda and Lin~\cite{Avellaneda-Lin-87,Avellaneda-Lin-91} in the context of periodic homogenization for divergence-form linear elliptic equations.
In the stochastic setting, while early contributions in form of annealed Green's function estimates appeared in~\cite{Delmotte-Deuschel-05,MaO}, a (quenched) large-scale regularity theory was first outlined by Armstrong and Smart~\cite{AS} (see also \cite{Armstrong-Daniel-16}), and later fully developed in~\cite{AKM-book,GNO-reg}.
For the steady Stokes problem~\eqref{eq:test-v0}, the development of a corresponding large-scale regularity theory is postponed to a forthcoming companion contribution~\cite{DG-20+} that is devoted to the quantitative homogenization of~\eqref{eq:test-v0}.
In the present work, in the spirit of~\cite{DO1}, we appeal to large-scale regularity in form of the following convenient annealed $\Ld^p$ regularity estimate established in~\cite{DG-20+}.

\begin{theor}[Annealed $\Ld^p$ regularity~\cite{DG-20+}]\label{CZ}
Let the random point process $\Pc_L$ be stationary on $Q_L$ and satisfy the hardcore condition in~\ref{Hd} for some $\delta>0$, as well as the improved mixing assumption~\ref{Mix+}.
Given $g\in\Ld^\infty(\Omega;C^\infty_\per(Q_L)^{d\times d})$, the unique solution $v_L\in\Ld^\infty(\Omega;H^1_\per(Q_L)^d)$ of~\eqref{eq:test-v0} satisfies for all $1<p,q<\infty$ and $\eta>0$,
\begin{equation*}
\|[\nabla v_L]_2\|_{\Ld^q(Q_L;\Ld^p(\Omega))}\,\lesssim_{p,q,\eta}\,\|[g]_2\|_{\Ld^q(Q_L;\Ld^{p+\eta}(\Omega))},
\end{equation*}
and for all $0\le r<d(1-\frac1q)$,
\begin{equation*}
\|\langle\cdot\rangle^r[\nabla v_L]_2\|_{\Ld^q(Q_L;\Ld^p(\Omega))}\,\lesssim_{p,q,r,\eta}\,\|\langle\cdot\rangle^r[g]_2\|_{\Ld^q(Q_L;\Ld^{p+\eta}(\Omega))}.
\qedhere
\end{equation*}
\end{theor}

In the perturbative setting $|p-2|,|q-2|\ll1$, the loss of stochastic integrability can be avoided, that is, the above holds with $\eta=0$, which happens to be a useful tool in the proof of Proposition~\ref{prop:CLTscaling}. In addition, such a perturbative statement can be established under mere stationarity and ergodicity assumptions, without any mixing. The proof is based on a simple Meyers type argument and is given in \cite{DG-20+}.

\begin{theor}[Perturbative annealed $\Ld^p$ regularity~\cite{DG-20+}]\label{th:annealed}
Let the random point process~$\Pc_L$ be stationary on $Q_L$ and satisfy the hardcore condition in~\ref{Hd} for some $\delta>0$.
Then, there exists a constant $\eta_0>0$ (only depending on $d,\delta$) such that the following holds:
Given $g\in\Ld^\infty(\Omega;C^\infty_\per(Q_L)^{d\times d})$, the unique solution $v_L\in\Ld^\infty(\Omega;H^1_\per(Q_L)^d)$ of~\eqref{eq:test-v0} satisfies for all~$p,q$ with $|q-2|,|p-2|\le\eta_0$,
\[\|[\nabla v_L]_2\|_{\Ld^q(Q_L;\Ld^p(\Omega))}\,\lesssim\,\|[g]_2\|_{\Ld^q(Q_L;\Ld^p(\Omega))}.\qedhere\]
\end{theor}

\begin{rem}[Dilute $\Ld^p$ regularity]\label{rem:reg-hofer}
In the dilute regime, the recent work of Höfer~\cite{Hofer-19} on the reflection method
easily yields the following version of Theorem~\ref{CZ}; the proof is a direct adaptation of~\cite{Hofer-19} and is omitted.
This also constitutes a variant of the dilute Green's function estimates in~\cite[Lemma~2.7]{Gloria-19}.\\
{\it Let the random point processes $(\Pc_L)_{L\ge1}$ satisfy the general assumptions~\ref{Hd} for some $\delta>0$, and denote by $\delta_L$ the minimal interparticle distance in $\Pc_L$.
For all $1<p,q<\infty$, there exists a constant $\delta_{q}>0$ (only depending on $d,q$) such that, provided $\Pc_L$ is dilute enough in the sense of $\delta_L\ge\delta_{q}$, we have: Given a random forcing $g\in\Ld^\infty(\Omega;C^\infty_\per(Q_L)^{d\times d})$, the unique solution $v_L\in\Ld^\infty(\Omega;H^1_\per(Q_L)^d)$ of~\eqref{eq:test-v0} satisfies
\[\|\nabla v_L\|_{\Ld^q(Q_L;\Ld^p(\Omega))}\,\lesssim\,\|g\|_{\Ld^q(Q_L;\Ld^p(\Omega))},\]
as well as the following deterministic estimate, for almost all $\w$,
\[\|\nabla v_L^\w\|_{\Ld^q(Q_L)}\,\lesssim\,\|g^\w\|_{\Ld^q(Q_L)}.\qedhere\]
}
\end{rem}

\subsubsection{Localized pressure estimates}
We state the following localized estimate on the pressure for the steady Stokes equation.
This is essentially a consequence of standard pressure estimates in~\cite{Galdi} but it requires some additional care since the prefactor in the estimate is uniform with respect to the size of $D$ although~$\Ic_L$ consists of an unbounded number of components. Note that the same result could be stated in $\Ld^q$ for any $1<q<\infty$, and that the Stokes problem below is tailored to cover both equations~\eqref{eq:StokesL+} and~\eqref{eq:test-v0}.

\begin{lem}[Localized pressure estimates]\label{lem:pres}
Let a (deterministic) point set $\Pc_L=\{x_{n,L}\}_n$ satisfy the hardcore condition in~\ref{Hd} for some $\delta>0$. Given $g\in C^\infty_\per(Q_L)^{d\times d}$ and $e'\in\R^d$, let $v_L\in H^1_\per(Q_L)^d$ denote the unique solution of
\begin{equation*}
\left\{\begin{array}{ll}
-\triangle w_L+\nabla Q_L=\Div g-\alpha_L e,&\text{in $Q_L\setminus\Ic_L$},\\
\Div w_L=0,&\text{in $Q_L\setminus\Ic_L$},\\
\D(w_L)=0,&\text{in $\Ic_L$},\\
e'|I_{n,L}|+\int_{\partial I_{n,L}}\big(g+\sigma(w_L,Q_L)\big)\nu=0,&\forall n,\\
\int_{\partial I_{n,L}}\Theta\nu\cdot\big(g+\sigma(w_L,Q_L)\big)\nu=0,&\forall n,\,\forall\Theta\in\Md^\Skew,
\end{array}\right.
\end{equation*}
with $\int_{Q_L}w_L=0$.
Then there holds for all balls $D\subset Q_L$ with radius $r_D$,
\begin{equation*}
\Big\|Q_L-\fint_{D\setminus\Ic_L}Q_L\Big\|_{\Ld^2(D\setminus\Ic_L)}^2\,\lesssim\,r_D^{d+2}|e|^2+\|\nabla w_{L}\|_{\Ld^2(D)}^2+\|g\|^2_{\Ld^2(D\setminus\Ic_L)}.\qedhere
\end{equation*}
\end{lem}

\begin{proof}
Let a ball $D\subset Q_L$ be fixed with radius $r_D$.
Using the Bogovskii operator as e.g.\@ in~\cite[Step~4.2 of the proof of Proposition~2.1]{DG-19}, we can construct a map $\zeta\in H^1_0(D)$ (implicitly extended by~0 on $Q_L\setminus D$) such that $\zeta|_{I_{n,L}}$ is constant for all $n$ and such that
\begin{gather*}
\Div\zeta\,=\,\Big(Q_L-\fint_{D\setminus\Ic_L}Q_L\Big)\mathds1_{D\setminus\Ic_L},\qquad
\|\nabla\zeta\|_{\Ld^2(D)}\,\lesssim\,\Big\|Q_{L}-\fint_{D\setminus\Ic_L}Q_{L}\Big\|_{\Ld^2(D\setminus\Ic_L)},
\end{gather*}
where we emphasize that the prefactor in the last estimate is uniformly bounded independently of $D$ and $L$.
Arguing as in Step~1 of the proof of Theorem~\ref{th:main1}, we note that the equation for $w_L$ implies in the weak sense on the whole periodic cell $Q_L$,
\[-\triangle w_L+\nabla(Q_L\mathds1_{Q_L\setminus\Ic_L})=\Div(g\mathds1_{Q_L\setminus \Ic_L})-\alpha_L e\mathds1_{Q_L\setminus\Ic_L}-\sum_n\delta_{\partial I_{n,L}}\big(g+\sigma(w_L,Q_L)\big)\nu.\]
Testing this equation with $\zeta$, recalling that $\zeta$ is constant inside particles, and using the boundary conditions for $w_L$, we are led to
\begin{multline*}
\int_{Q_L\setminus\Ic_L}Q_{L}\Div\zeta=\int_{Q_L}\nabla\zeta:\nabla w_{L}+\int_{Q_L\setminus \Ic_L}\nabla\zeta:g
+\alpha_Le\cdot\int_{Q_L\setminus\Ic_L}\zeta-e\cdot\sum_n\int_{I_{n,L}}\zeta.
\end{multline*}
Inserting the definition of $\Div\zeta$, recalling that $|\alpha_L|\lesssim1$, and using Poincar\'e's inequality on~$D$, we deduce
\begin{eqnarray*}
\Big\|Q_L-\fint_{D\setminus\Ic_L}Q_L\Big\|_{\Ld^2(D\setminus\Ic_L)}^2&\lesssim& \|\nabla\zeta\|_{\Ld^2(D)}\Big(\|\nabla w_{L}\|_{\Ld^2(D)}+\|g\|_{\Ld^2(D\setminus\Ic_L)}\Big)+|e|\|\zeta\|_{\Ld^1(D)}\\
&\lesssim& \|\nabla\zeta\|_{\Ld^{2}(D)}\Big(r_D^{d+2}|e|^2+\|\nabla w_{L}\|_{\Ld^2(D)}^2+\|g\|^2_{\Ld^2(D\setminus\Ic_L)}\Big)^\frac12,
\end{eqnarray*}
and the claim follows from the bound on the $\Ld^{2}$-norm of $\nabla\zeta$.
\end{proof}

\subsection{Proof of Proposition~\ref{prop:CLTscaling}}
We shall exploit the multiscale variance inequality~\eqref{eq:SGL}, or its hyperuniform version~\eqref{eq:SGL-hyp}, and appeal to a duality argument.
Let the (deterministic) test function $g\in C^\infty_\per(Q_L)^{d\times d}$ be fixed and for all $\w$ let $v_{L}^\w\in H^1_\per(Q_L)^d$ denote the unique solution of the following auxiliary problem,
\begin{equation}\label{eq:test-v}
\left\{\begin{array}{ll}
-\triangle v_{L}^\w+\nabla P_{L}^\w=\nabla\cdot g,&\text{in $Q_L\setminus\Ic_L^\w$},\\
\Div v_L^\w=0,&\text{in $Q_L\setminus\Ic_L^\w$},\\
\D(v_L^\w)=0,&\text{in $\Ic_{L}^\w$},\\
\int_{\partial I_{n,L}^\w}\big(g+\sigma(v_L^\w,P_L^\w)\big)\nu=0,&\forall n,\\
\int_{\partial I_{n,L}^\w}\Theta\nu\cdot\big(g+\sigma(v_L^\w,P_L^\w)\big)\nu=0,&\forall n,\,\forall\Theta\in\Md^\Skew,
\end{array}\right.
\end{equation}
with $\int_{Q_L}v_{L}^\w=0$.
We split the proof into four main steps: the mixing case~(i) is treated in the first three steps, while the simplifications that appear in the hyperuniform case~(ii) are pointed out in the last step.

\medskip
\step1 Fluctuation scaling outside $\Ic_L$: proof that under~\ref{Mix+} in dimension $d>2$ there holds for all $1\le R\le L$, $q\ge1$, and $1\ll p<\infty$,
\begin{multline}\label{eq:st1-SG}
\Big\|\int_{Q_L\setminus\Ic_L}g:\nabla \phi_{L}\Big\|_{\Ld^{2p}(\Omega)}^{2}
\,\lesssim_p\,L^{-d}\Big\|e\cdot\int_{\Ic_L}v_L\Big\|_{\Ld^p(\Omega)}^2\\
+\|g\|_{\Ld^{\frac{2d}{d+2}}(Q_L)}^2
+\|\avoir g\|_{\Ld^2(Q_L)}^2\,\bigg\|\Big(1+\int_{B_{R}}[\nabla\phi_{L}]_2^{2q}\Big)^\frac1q\bigg\|_{\Ld^p(\Omega)}.
\end{multline}
where alternatively the norm $\|g\|_{\Ld^{\frac{2d}{d+2}}(Q_L)}^2$ can be replaced by $\|\langle\cdot\rangle g\|_{\Ld^2(Q_L)}^2$.

\medskip\noindent
Using the version \eqref{eq:SGL-p} of the multiscale variance inequality~\eqref{eq:SGL} to control higher moments, we obtain
\begin{multline}\label{eq:SG-phi}
\Big\|\int_{Q_L\setminus\Ic_L}g:\nabla \phi_{L}\Big\|_{\Ld^{2p}(\Omega)}^{2}\\
\,\lesssim\,p^2\,\E\bigg[\int_0^L\bigg(\int_{Q_L}\Big(\partial^{\operatorname{osc}}_{\Pc_L,B_\ell(x)}\int_{Q_L\setminus\Ic_L}g:\nabla \phi_{L}\Big)^2dx\bigg)^p\langle\ell\rangle^{-dp}\,\pi(\ell)\,d\ell\bigg]^\frac1p,
\end{multline}
and it remains to estimate the oscillation of $\int_{Q_L\setminus\Ic_L} g:\nabla\phi_L$ with respect to $\Pc_L$ on any ball~$B_\ell(x)$.
Given $0\le \ell\le L$ and $x\in\R^d$, and given a realization, let $\Pc_L'$ be a locally finite point set satisfying the hardcore condition in~\ref{Hd}, with $\Pc_L'|_{Q_L\setminus B_\ell(x)}=\Pc_L|_{Q_L\setminus B_\ell(x)}$, and denote by~$\phi_{L}'$ the corresponding solution of equation~\eqref{eq:StokesL+} with $\Pc_L$ replaced by $\Pc'_L$.
We split the proof into three further substeps.

\medskip
\substep{1.1} Proof that
\begin{multline}\label{ant.1.1}
\Big|\int_{Q_L\setminus\Ic_L}g:\nabla\phi_{L}-\int_{Q_L\setminus\Ic_L'}g:\nabla\phi_{L}'\Big|
\,\lesssim\, L^{-d}\Big|e\cdot\int_{\Ic_L}v_L\Big|+\int_{B_{\ell+1}(x)}|v_L|\\
+\Big(\int_{B_{\ell+2}(x)}|\nabla v_L|^2+|\avoir g|^2\Big)^\frac12\Big(\langle\ell\rangle^{d+2}+\int_{B_{\ell+2}(x)}|\nabla\phi_{L}'|^2\Big)^\frac12.
\end{multline}
First decomposing $\int_{Q_L \setminus\Ic_L'} =\int_{Q_L \setminus\Ic_L}+\int_{\Ic_L\setminus\Ic_L'}-\int_{\Ic_L'\setminus\Ic_L}$, we find
\begin{multline*}
\int_{Q_L\setminus\Ic_L}g:\nabla\phi_{L}-\int_{Q_L\setminus\Ic_L'}g:\nabla\phi_{L}'\\
=\int_{Q_L\setminus\Ic_L}g:\nabla(\phi_{L}-\phi_{L}')+\int_{\Ic_L'\setminus\Ic_L}g:\nabla\phi_{L}'-\int_{\Ic_L\setminus\Ic'_L}g:\nabla\phi_{L}',
\end{multline*}
hence, in view of the inclusion $(\Ic_L\setminus\Ic_L')\cup (\Ic_L'\setminus\Ic_L)\subset B_{\ell+1}(x)$,
\begin{multline}\label{eq:decomp-delphi0}
\Big|\int_{Q_L\setminus\Ic_L}g:\nabla\phi_{L}-\int_{Q_L\setminus\Ic'_L}g:\nabla\phi_{L}'\Big|\\
\lesssim\Big|\int_{Q_L\setminus\Ic_L}g:\nabla(\phi_{L}-\phi_{L}')\Big|+\Big(\int_{B_{\ell+1}(x)}|g|^2\Big)^\frac12\Big(\int_{B_{\ell+1}(x)}|\nabla\phi_{L}'|^2\Big)^\frac12,
\end{multline}
and it remains to examine the first right-hand side term.
Arguing similarly as in Step~1 of the proof of Theorem~\ref{th:main1}, we note that the equation~\eqref{eq:test-v} for $v_L$ implies in the weak sense on the whole periodic cell $Q_L$,
\begin{equation}\label{eq:reform-rd-v+}
-\triangle v_L+\nabla(P_L\mathds1_{Q_L\setminus\Ic_L})=\nabla\cdot(g\mathds1_{Q_L\setminus\Ic_L})-\sum_n\delta_{\partial I_{n,L}}\big(g+\sigma(v_L,P_L)\big)\nu.
\end{equation}
Testing this equation with $\phi_{L}-\phi_{L}'$, and recalling that the pressure $P_L$ is only defined up to additive constant and that $\phi_L$ and $\phi_L'$ are divergence-free on the whole periodic cell $Q_L$,  we obtain for any constant~$c_1\in\R$,
\begin{multline*}
\int_{Q_L\setminus\Ic_L}g:\nabla(\phi_{L}-\phi_{L}')
\,=\,-\int_{Q_L}\nabla v_L:\nabla(\phi_{L}-\phi_{L}')\\
-\sum_n\int_{\partial I_{n,L}}(\phi_{L}-\phi_{L}')\cdot\big(g+\sigma(v_L,P_L-c_1)\big)\nu,
\end{multline*}
hence, in view of the boundary conditions for $\phi_L$, $\phi_L'$, and $v_L$, arguing similarly as in Step~2 of the proof of Theorem~\ref{th:main1},
\begin{multline}\label{ant.1.1.1}
\int_{Q_L\setminus\Ic_L}g:\nabla(\phi_{L}-\phi_{L}')
\,=\,-\int_{Q_L}\nabla v_L:\nabla(\phi_{L}-\phi_{L}')\\
+\sum_{n:x_{n,L}\in B_\ell(x)}\int_{\partial I_{n,L}}\Big(\phi_{L}'-\fint_{ I_{n,L}}\phi_{L}'\Big)\cdot\big(g+\sigma(v_L,P_L-c_1)\big)\nu.
\end{multline}
Likewise, testing with $v_L$ the equation for $\phi_{L}-\phi_{L}'$ in the form~\eqref{eq:reform-sedim}, we get for any constants~$c_2,c_2'\in\R$,
\begin{multline*}
-\int_{Q_L}\nabla v_L:\nabla(\phi_{L}-\phi_{L}')=\alpha_L e\cdot\int_{Q_L\setminus\Ic_L}v_L-\alpha_L' e\cdot\int_{Q_L\setminus\Ic_L'}v_L\\
+\sum_n\int_{\partial I_{n,L}} v_L\cdot\sigma(\phi_L,\Pi_L-c_2)\nu
-\sum_n\int_{\partial I_{n,L}'} v_L\cdot\sigma(\phi_L',\Pi_L'-c_2')\nu,
\end{multline*}
which, in view of the choice $\int_{Q_L}v_L=0$ and of the boundary conditions for $\phi_L$, $\phi_L'$, and~$v_L$, turns into
\begin{multline}\label{ant.1.1.2}
\int_{Q_L}\nabla v_L:\nabla(\phi_{L}-\phi_{L}')=\alpha_L e\cdot\int_{\Ic_L}v_L-\alpha_L' e\cdot\int_{\Ic_L'}v_L\\
+\sum_{n:x_{n,L}\in B_\ell(x)}e\cdot\int_{I_{n,L}}v_L-\sum_{n:x_{n,L}'\in B_\ell(x)}e\cdot\int_{I_{n,L}'}v_L
\\
-\sum_{n:x_{n,L}'\in B_\ell(x)}\int_{\partial I_{n,L}'} \Big(v_L-\fint_{I_{n,L}'}v_L\Big)\cdot\sigma(\phi_L',\Pi_L'-c_2')\nu.
\end{multline}
Combined with \eqref{ant.1.1.1}, this yields
\begingroup\allowdisplaybreaks
\begin{multline}\label{eq:bound-Dphi-00}
\int_{Q_L\setminus\Ic_L}g:\nabla(\phi_{L}-\phi_{L}')=-(\alpha_L-\alpha_L') e\cdot\int_{\Ic_L}v_L\\
-(\alpha_L'+1)\bigg(\sum_{n:x_{n,L}\in B_\ell(x)}e\cdot\int_{I_{n,L}}v_L-\sum_{n:x_{n,L}'\in B_\ell(x)}e\cdot\int_{I_{n,L}'}v_L\bigg)\\
-\sum_{n:x_{n,L}'\in B_\ell(x)}\int_{\partial I_{n,L}'} \Big(v_L-\fint_{I_{n,L}'}v_L\Big)\cdot\sigma(\phi_L',\Pi_L'-c_2')\nu\\
+\sum_{n:x_{n,L}\in B_\ell(x)}\int_{\partial I_{n,L}}\Big(\phi_{L}'-\fint_{ I_{n,L}}\phi_{L}'\Big)\cdot\big(g+\sigma(v_L,P_L-c_1)\big)\nu.
\end{multline}
\endgroup
In order to control the different right-hand side terms localized at particle boundaries, we appeal to the local regularity theory for the Stokes equation. We illustrate this on $v_L$: a trace estimate first yields
\[\int_{\partial I_{n,L}}|\nabla v_L|^2+|P_L-c_1|^2\,\lesssim\,\int_{(I_{n,L}+\frac12\delta B)\setminus I_{n,L}}|\avoir \nabla v_L|^2+|\avoir (P_L-c_1)|^2,\]
while, for any constant $c_3\in\R^d$, replacing $v_L$ by $v_L-c_3$, the local regularity theory for the Stokes equation~\eqref{eq:test-v} satisfied by $v_L-c_3$ in $(I_{n,L}+\delta B)\setminus I_{n,L}$ can then be applied in form of e.g.~\cite[Theorem~IV.5.1]{Galdi},
\begin{multline*}
\int_{\partial I_{n,L}}|\nabla v_L|^2+|P_L-c_1|^2\\
\,\lesssim\,\int_{\partial I_{n,L}}\big|\langle\nabla\rangle (v_L-c_3)|_{I_{n,L}}\big|^2+\int_{(I_{n,L}+\delta B)\setminus I_{n,L}}|\nabla v_L|^2+|P_L-c_1|^2+|\avoir g|^2,
\end{multline*}
which yields by Poincar\'e's inequality, for the choice $c_3:=\fint_{I_{n,L}}v_L$, recalling the linearity of $v_L$ inside particles,
\begin{equation}\label{eq:Stokes-reg-1}
\int_{\partial I_{n,L}}|\nabla v_L|^2+|P_L-c_1|^2\,\lesssim\,\int_{I_{n,L}+\delta B}|\nabla v_L|^2+|P_L-c_1|^2\mathds1_{Q_L\setminus\Ic_L}+| \avoir g|^2.
\end{equation}
Likewise, as $\phi_L'$ satisfies \eqref{eq:StokesL+}, we find
\begin{equation}\label{eq:Stokes-reg-2}
\int_{\partial I_{n,L}'}|\nabla\phi_{L}'|^2+|\Pi'_{L}-c_2'|^2\,\lesssim\,\int_{I_{n,L}'+\delta B}|\nabla\phi_{L}'|^2+|\Pi_{L}'-c_2'|^2\mathds1_{Q_L\setminus\Ic_L'}
+|\alpha'_L|^2.
\end{equation}
Inserting these bounds into~\eqref{eq:bound-Dphi-00}, recalling that $|\alpha_L|,|\alpha'_L|\lesssim 1$, and noting that
\begin{multline*}
\alpha_L-\alpha_L'=\frac{|\Ic_L||Q_L\setminus\Ic_L'|-|\Ic_L'||Q_L\setminus\Ic_L|}{|Q_L\setminus\Ic_L||Q_L\setminus\Ic_L'|}=\frac{|Q_L|\big(|\Ic_L|-|\Ic_L'|\big)}{|Q_L\setminus\Ic_L||Q_L\setminus\Ic_L'|}\\
=\frac{|Q_L|\big(|\Ic_L\cap B_{\ell+1}(x)|-|\Ic_L'\cap B_{\ell+1}(x)|\big)}{|Q_L\setminus\Ic_L||Q_L\setminus\Ic_L'|}\lesssim L^{-d}\langle\ell\rangle^d,
\end{multline*}
we are  led to
\begin{multline*}
\Big|\int_{Q_L\setminus\Ic_L}g:\nabla(\phi_{L}-\phi_{L}')\Big|
\,\lesssim\, L^{-d}\langle\ell\rangle^d\Big|e\cdot\int_{\Ic_L}v_L\Big|+\int_{B_{\ell+1}(x)}|v_L|\\
+\Big(\int_{B_{\ell+2}(x)}\big(|\nabla v_L|^2+|P_L-c_1|^2\mathds1_{Q_L\setminus\Ic_L}+|\avoir g|^2\big)\Big)^\frac12\\
\times\Big(\int_{B_{\ell+2}(x)}\big(|\nabla\phi_{L}'|^2+|\Pi_{L}'-c_2'|^2\mathds1_{Q_L\setminus\Ic_L'}+1\big)\Big)^\frac12.
\end{multline*}
Choosing $c_1:=\fint_{B_{\ell+2}(x)\setminus\Ic_L}P_L$ and $c_2':=\fint_{B_{\ell+2}(x)\setminus\Ic'_L}\Pi_{L}'$, and using the pressure estimate of Lemma~\ref{lem:pres},
we deduce
\begin{multline}\label{eq:preant1.1}
\Big|\int_{Q_L\setminus\Ic_L}g:\nabla(\phi_{L}-\phi_{L}')\Big|
\lesssim L^{-d}\langle\ell\rangle^d\Big|e\cdot\int_{\Ic_L}v_L\Big|+\int_{B_{\ell+1}(x)}|v_L|\\
+\Big(\int_{B_{\ell+2}(x)}|\nabla v_L|^2+| \avoir g|^2\Big)^\frac12\Big(\langle\ell\rangle^{d+2}+\int_{B_{\ell+2}(x)}|\nabla\phi_{L}'|^2\Big)^\frac12.
\end{multline}
Combined with~\eqref{eq:decomp-delphi0}, this yields the claim \eqref{ant.1.1}.

\medskip
\substep{1.2}
Proof that in dimension $d>2$,
\begin{equation}\label{ant.1.3}
\int_{B_{\ell+2}(x)}|\nabla\phi_{L}'|^2\,\lesssim\,\langle\ell\rangle^{d+2}+\int_{B_{\ell+2}(x)}|\nabla\phi_{L}|^2.
\end{equation}
Starting from~\eqref{eq:reform-sedim} and arguing as for \eqref{ant.1.1.2}, the energy identity for $\phi_{L}-\phi_{L}'$ takes the following form, for any constants $c_4,c_4'\in\R$,
\begin{multline*}
\int_{Q_L}|\nabla(\phi_{L}-\phi_{L}')|^2=-(\alpha_L-\alpha_L') e\cdot \int_{Q_L\setminus\Ic_L}(\phi_{L}-\phi_{L}')\\
-(\alpha_L'+1)\bigg(\sum_{n:x_{n,L}\in B_\ell(x)}e\cdot\int_{I_{n,L}}(\phi_{L}-\phi_{L}')-\sum_{n:x_{n,L}'\in B_\ell(x)}e\cdot\int_{I_{n,L}'}(\phi_{L}-\phi_{L}')\bigg)\\
+\sum_{n:x_{n,L}\in B_\ell(x)}\int_{\partial I_{n,L}}\Big(\phi_{L}'-\fint_{I_{n,L}}\phi_{L}'\Big)\cdot\big(2\D(\phi_{L})-(\Pi_{L}-c_4)\Id\big)\nu\\
+\sum_{n:x_{n,L}'\in B_\ell(x)}\int_{\partial I_{n,L}'}\Big(\phi_{L}-\fint_{I_{n,L}'}\phi_{L}\Big)\cdot\big(2\D(\phi_{L}')-(\Pi_{L}'-c_4')\Id\big)\nu.
\end{multline*}
Using the local regularity theory for the Stokes equation in form of~\eqref{eq:Stokes-reg-2} as in Substep~1.1, this leads to
\begin{multline*}
\int_{Q_L}|\nabla(\phi_{L}-\phi_{L}')|^2\,\lesssim\, L^{-d}\langle\ell\rangle^d\int_{Q_L}|\phi_{L}-\phi_{L}'|+\int_{B_{\ell+1}(x)}|\phi_{L}-\phi'_{L}|\\
+\Big(\int_{B_{\ell+2}(x)}|\nabla\phi_{L}'|^2+|\Pi_{L}'-c_4'|^2\mathds1_{Q_L\setminus\Ic_L'}\Big)^\frac12\Big(\int_{B_{\ell+2}(x)}|\nabla\phi_{L}|^2+|\Pi_{L}-c_4|^2\mathds1_{Q_L\setminus\Ic_L}\Big)^\frac12.
\end{multline*}
Hence, choosing $c_4:=\fint_{B_{\ell+2}(x)\setminus\Ic_L}\Pi_{L}$ and $c_4':=\fint_{B_{\ell+2}(x)\setminus\Ic_L'}\Pi'_{L}$
and using the pressure estimate of Lemma~\ref{lem:pres} for both $\Pi_L$ and $\Pi_L'$, we obtain
\begin{multline}\label{eq:preant1.3}
\int_{Q_L}|\nabla(\phi_{L}-\phi_{L}')|^2\,\lesssim\, L^{-d}\langle\ell\rangle^d\int_{Q_L}|\phi_{L}-\phi_{L}'|+\int_{B_{\ell+1}(x)}|\phi_{L}-\phi'_{L}|\\
+\Big(\langle\ell\rangle^{d+2}+\int_{B_{\ell+2}(x)}|\nabla\phi_{L}'|^2\Big)^\frac12\Big(\langle\ell\rangle^{d+2}+\int_{B_{\ell+2}(x)}|\nabla\phi_{L}|^2\Big)^\frac12.
\end{multline}
Using Poincar\'e's inequality in the form
\[L^{-d}\int_{Q_L}|\phi_{L}-\phi_{L}'|\lesssim L^{1-\frac d2}\Big(\int_{Q_L}|\nabla(\phi_{L}-\phi_{L}')|^2\Big)^\frac12,\]
and the Poincar\'e-Sobolev inequality for $d>2$ in the form
\begin{eqnarray}
\int_{B_{\ell+1}(x)}|\phi_{L}-\phi'_{L}|&\lesssim&\langle\ell\rangle^{\frac d2+1}\Big(\int_{Q_L}\,[\phi_{L}-\phi_{L}']_2^{\frac{2d}{d-2}}\Big)^{\frac{d-2}{2d}}\nonumber\\
&\lesssim&\langle\ell\rangle^{\frac d2+1}\Big(\int_{Q_L}|\nabla(\phi_{L}-\phi'_{L})|^2\Big)^\frac12,\label{eq:pre-sob-evit}
\end{eqnarray} 
we deduce in dimension $d>2$,
\begin{multline*}
\int_{Q_L}|\nabla(\phi_{L}-\phi_{L}')|^2\,\lesssim\, \langle\ell\rangle^{\frac d2+1}\Big(\int_{Q_L}|\nabla(\phi_{L}-\phi_{L}')|^2\Big)^\frac12\\
+\Big(\langle\ell\rangle^{d+2}+\int_{B_{\ell+2}(x)}|\nabla\phi_{L}'|^2\Big)^\frac12\Big(\langle\ell\rangle^{d+2}+\int_{B_{\ell+2}(x)}|\nabla\phi_{L}|^2\Big)^\frac12,
\end{multline*}
hence,
\begin{equation*}
\int_{Q_L}|\nabla(\phi_{L}-\phi_{L}')|^2\lesssim\Big(\langle\ell\rangle^{d+2}+\int_{B_{\ell+2}(x)}|\nabla\phi_{L}'|^2\Big)^\frac12\Big(\langle\ell\rangle^{d+2}+\int_{B_{\ell+2}(x)}|\nabla\phi_{L}|^2\Big)^\frac12,
\end{equation*}
and the claim~\eqref{ant.1.3} follows by the triangle inequality.

\medskip
\substep{1.3} Proof of~\eqref{eq:st1-SG}.
\\
Using the result~\eqref{ant.1.3} of Substep~1.2 to replace the perturbed corrector $\phi_{L}'$ by $\phi_{L}$ in the right-hand side of the result~\eqref{ant.1.1} of Substep~1.1, we obtain
\begin{multline*}
\Big|\partial^{\operatorname{osc}}_{\Pc,B_\ell(x)}\int_{Q_L\setminus\Ic_L}g:\nabla\phi_{L}\Big|
\lesssim L^{-d}\langle\ell\rangle^d\Big|e\cdot\int_{\Ic_L}v_L\Big|+\Big|\int_{B_{\ell+1}(x)}v_L\Big|\\
+\Big(\int_{B_{\ell+2}(x)}|\nabla v_L|^2+|\avoir g|^2\Big)^\frac12\Big(\langle\ell\rangle^{d+2}+\int_{B_{\ell+2}(x)}|\nabla\phi_{L}|^2\Big)^\frac12.
\end{multline*}
Inserting this into~\eqref{eq:SG-phi}, we find
for all $p\ge1$,
\begin{multline}\label{eq:pre-bnd-aver-phi}
\Big\|\int_{Q_L\setminus\Ic_L}g:\nabla \phi_{L}\Big\|_{\Ld^{2p}(\Omega)}^{2}
\,\lesssim\,p^2L^{-d}\Big\|e\cdot\int_{\Ic_L}v_L\Big\|_{\Ld^p(\Omega)}^2\Big(\int_0^L\langle\ell\rangle^{dp}\pi(\ell)\,d\ell\Big)^\frac1p\\
+p^2\Big\|\int_{Q_L}|v_L|^2\Big\|_{\Ld^p(\Omega)}\Big(\int_0^L\langle\ell\rangle^{dp}\pi(\ell)\,d\ell\Big)^\frac1p\\
+p^2\,\E\bigg[\int_0^L\bigg(\int_{Q_L}\zeta_\ell(x)^2\Big(\fint_{B_{\ell+2}(x)}|\nabla v_L|^2+|\avoir g|^2\Big)\,dx\bigg)^p\langle\ell\rangle^{dp}\pi(\ell)\,d\ell\bigg]^\frac1p,
\end{multline}
where we have set for abbreviation,
\begin{equation*}
\zeta_\ell(x):=\langle\ell\rangle+\Big(\fint_{B_{\ell+2}(x)}|\nabla\phi_{L}|^2\Big)^\frac12.
\end{equation*}
Before estimating the last right-hand side term in~\eqref{eq:pre-bnd-aver-phi}, we first smuggle in a spatial average at some arbitrary scale $0\le R\le L$,
\begin{multline*}
\int_{Q_L}\zeta_\ell(x)^2\Big(\fint_{B_{\ell+2}(x)}|\nabla v_L|^2+|\avoir g|^2\Big)\,dx\\
\,\lesssim\,\int_{Q_L}\Big(\sup_{B_R(y)}\zeta_\ell^2\Big)\bigg(\fint_{B_{\ell+2}(y)}\Big(\fint_{B_{R+1}(x)}[\nabla v_L]_2^2+[\avoir g]_2^2\Big)\,dx\bigg)\,dy.
\end{multline*}
We then use a duality representation to compute the $\Ld^p(\Omega)$-norm of this expression
(in the following, $X$ denotes a random variable, which is independent of the space variable),
\begin{multline*}
\E\bigg[\bigg(\int_{Q_L}\zeta_\ell(x)^2\Big(\fint_{B_{\ell+2}(x)}|\nabla v_L|^2+|\avoir g|^2\Big)\,dx\bigg)^p\bigg]^\frac1p\\
\hspace{-6cm}\,\lesssim\,\sup_{\|X\|_{\Ld^{2p'}(\Omega)}=1}\,\E\bigg[\int_{Q_L}\Big(\sup_{ B_R(y)}\zeta_\ell^2\Big)\\
\times\bigg(\fint_{B_{\ell+2}(y)}\Big(\fint_{B_{R+1}(x)}[\nabla (Xv_L)]_2^2+X^2[\avoir g]_2^2\Big)\,dx\bigg)\,dy\bigg].
\end{multline*}
By H\"older's inequality and by stationarity of $\zeta_\ell$, we find
\begin{multline*}
\E\bigg[\bigg(\int_{Q_L}\zeta_\ell(x)^2\Big(\fint_{B_{\ell+2}(x)}|\nabla v_L|^2+|\avoir g|^2\Big)\,dx\bigg)^p\bigg]^\frac1p
\,\lesssim\, \E\Big[\sup_{B_R}\zeta_\ell^{2p}\Big]^\frac1p\\
\times\sup_{\|X\|_{\Ld^{2p'}(\Omega)}=1} \int_{Q_L}\E \bigg[\bigg(\fint_{B_{\ell+2}(y)} \Big(\fint_{B_{R+1}(x)} [\nabla (Xv_L)]_2^2+X^2 [\avoir g]_2^2\Big)\,dx\bigg)^{p'} \bigg]^\frac1{p'}dy,
\end{multline*}
hence, since $g$ (unlike $v_L$) is deterministic, using Jensen's inequality,
\begin{multline*}
\E\bigg[\bigg(\int_{Q_L}\zeta_\ell(x)^2\Big(\fint_{B_{\ell+2}(x)}|\nabla v_L|^2+|\avoir g|^2\Big)\,dx\bigg)^p\bigg]^\frac1p\\
\,\lesssim\,\E\Big[\sup_{B_R}\zeta_\ell^{2p}\Big]^\frac1p\bigg(\|\avoir g\|_{\Ld^2(Q_L)}^2+\sup_{\|X\|_{\Ld^{2p'}(\Omega)}=1}\,\|[\nabla (Xv_L)]_2\|_{\Ld^2(Q_L;\Ld^{2p'}(\Omega))}^2\bigg).
\end{multline*}
By perturbative annealed $\Ld^p$ regularity theory in form of Theorem~\ref{th:annealed} (without loss of stochastic integrability!), we deduce for~$p\gg1$,
\begin{equation*}
\E\bigg[\bigg(\int_{Q_L}\zeta_\ell(x)^2\Big(\fint_{B_{\ell+2}(x)}|\nabla v_L|^2+|\avoir g|^2\Big)\,dx\bigg)^p\bigg]^\frac1p\,\lesssim\,\E\Big[\sup_{B_R}\zeta_\ell^{2p}\Big]^\frac1p\|\avoir g\|_{\Ld^2(Q_L)}^2,
\end{equation*}
where the supremum of $\zeta_\ell$ can be estimated as follows, for all $q\ge1$,
\begin{equation*}
\E\Big[\sup_{B_R}\zeta_\ell^{2p}\Big]^\frac1p\,\lesssim\,\langle\ell\rangle^2+\bigg\|\Big(\int_{B_{R+1}}[\nabla\phi_{L}]_2^{2q}\Big)^\frac1q\bigg\|_{\Ld^p(\Omega)},
\end{equation*}
For all $1\le R\le L$, $q\ge1$, and $p\gg1$,
inserting this into~\eqref{eq:pre-bnd-aver-phi} yields
\begin{multline*}
\Big\|\int_{Q_L\setminus\Ic_L}g:\nabla \phi_{L}\Big\|_{\Ld^{2p}(\Omega)}^{2}
\,\lesssim\,p^2L^{-d}\Big\|e\cdot\int_{\Ic_L}v_L\Big\|_{\Ld^p(\Omega)}^2\Big(\int_0^L\langle\ell\rangle^{dp}\pi(\ell)\,d\ell\Big)^\frac1p\\
+p^2\Big\|\int_{Q_L}|v_L|^2\Big\|_{\Ld^p(\Omega)}\Big(\int_0^L\langle\ell\rangle^{dp}\pi(\ell)\,d\ell\Big)^\frac1p\\
+p^2\|\avoir g\|_{\Ld^2(Q_L)}^2\,\bigg\|\Big(1+\int_{B_{R+1}}[\nabla\phi_{L}]_2^{2q}\Big)^\frac1q\bigg\|_{\Ld^p(\Omega)}
\Big(\int_0^L\langle\ell\rangle^{(d+2)p}\pi(\ell)\,d\ell\Big)^\frac1p.
\end{multline*}
Using the Poincar\'e-Sobolev inequality in dimension $d>2$, Jensen's inequality, and the non-perturbative annealed $\Ld^p$ regularity theory in form of Theorem~\ref{CZ}, recalling that $g$ is deterministic, we find for $1<p<\infty$,
\begin{multline}\label{ant.1.4.1}
\Big\|\int_{Q_L}|v_L|^2\Big\|_{\Ld^p(\Omega)}\,\lesssim\,\bigg\|\Big(\int_{Q_L}|\nabla v_L|^\frac{2d}{d+2}\Big)^\frac{d+2}d\bigg\|_{\Ld^p(\Omega)}\\
\,\lesssim\,\|\nabla v_L\|_{\Ld^\frac{2d}{d+2}(Q_L;\Ld^{2p}(\Omega))}^2\,\lesssim_p\,\|g\|_{\Ld^\frac{2d}{d+2}(Q_L)}^2.
\end{multline}
Combining this with the above, and using the superalgebraic decay of the weight~$\pi$,
the claim~\eqref{eq:st1-SG} follows.

\medskip
\substep{1.4} Modification of~\eqref{eq:st1-SG}.\\
The estimate~\eqref{eq:st1-SG} fails to give the correct power of the logarithm in~\eqref{e.ant-need-more} in the critical dimension $d=4$, in which case the estimate needs to be slightly modified.
To this end, we use the following refined version of the Poincar\'e inequality: There exists a universal constant $C>0$ such that for all $L\ge 1$ and all $\zeta \in H^1_\per(Q_L)$ with vanishing average $\int_{Q_L} \zeta =0$, we have
\begin{equation}\label{eq:Poinc-weight}
\int_{Q_L} |\zeta|^2 \,\le\, C \int_{Q_L} |x|^2|\nabla \zeta(x)|^2 dx, 
\end{equation}
in favor of which we presently argue. By scaling it is enough to consider $L=1$, and we proceed by contradiction. Assume there exists a sequence $(\zeta_n)_n \subset H^1_\per(Q)$ with the following properties: $\int_Q |\zeta_n|^2 =1$, $\int_Q \zeta_n=0$, and $\int_{Q} |x|^2 |\nabla \zeta_n(x)|^2 dx \to 0$.
First, by weak compactness, up to a subsequence, $\zeta_n$ converges weakly to some limit $\zeta$ in $\Ld^2_\per(Q)$, with $\int_Q\zeta=0$.
Next, for all $\e>0$, we find $\int_{Q\setminus B_\e}  |\nabla \zeta_n|^2 \le \e^{-2}\int_{Q} |x|^2 |\nabla \zeta_n(x)|^2 dx\to0$, which entails that the limit $\zeta$ must be a constant in $Q$, hence $\zeta=0$ since $\int_Q\zeta=0$.
Using Rellich's theorem e.g.\@ in the annulus $Q\setminus \frac12Q$, we further find that the restriction $\zeta_n|_{Q\setminus\frac12Q}$ converges strongly to $0$ in $\Ld^2(Q)$.
Given a cut-off function $\chi\in C^\infty_c(Q)$ with $\chi|_{\frac12Q}=1$ and $0\le\chi\le1$, we compute
\begin{eqnarray*}
\int_Q\chi^2|\zeta_n|^2&=&\tfrac1d\int_Q\chi(x)^2|\zeta_n(x)|^2\,(\Div\, x)\,dx\\
&=&-\tfrac2d\int_Q\chi(x)^2\zeta_n(x)\otimes x:\nabla\zeta_n(x)\,dx-\tfrac2d\int_Q\chi(x)|\zeta_n(x)|^2x\cdot\nabla\chi(x)\,dx\\
&\lesssim&\Big(\int_Q\chi^2|\zeta_n|^2\Big)^\frac12\Big(\int_Q|x|^2|\nabla\zeta_n(x)|^2\,dx+\int_Q|\nabla\chi|^2|\zeta_n|^2\Big)^\frac12,
\end{eqnarray*}
and thus, using the properties of $\chi$,
\begin{eqnarray*}
\int_{\tfrac12Q}|\zeta_n|^2
&\lesssim&\int_Q|x|^2|\nabla\zeta_n(x)|^2\,dx+\int_{Q\setminus\frac12Q}|\zeta_n|^2.
\end{eqnarray*}
As the restriction $\zeta_n|_{Q\setminus\frac12Q}$ converges strongly to $0$ in $\Ld^2(Q)$, we deduce that $\zeta_n$ also converges strongly to $0$ in $\Ld^2(Q)$, which gives the claimed contradiction.

\medskip\noindent
Applying~\eqref{eq:Poinc-weight} to $v_L$, we thus have 
$
\int_{Q_L}|v_L|^2 \,\lesssim \, \int_{Q_L} \langle\cdot\rangle^2|\nabla v_L|^2,
$
and we now appeal to the non-perturbative weighted annealed $\Ld^p$ regularity theory in form of Theorem~\ref{CZ}, to the effect that
\begin{equation*}
\Big\|\int_{Q_L}|v_L|^2\Big\|_{\Ld^p(\Omega)}\,\lesssim\,\Big\|\int_{Q_L}\langle\cdot\rangle^2|\nabla v_L|^2\Big\|_{\Ld^p(\Omega)}
\,\lesssim_p\,\|\langle\cdot\rangle g\|_{\Ld^2(Q_L)}^2.
\end{equation*}
Using this estimate instead of~\eqref{ant.1.4.1} in Substep~1.3 yields the claimed modification of~\eqref{eq:st1-SG}.

\medskip
\step2
Proof that under~\ref{Mix+} in dimension $d>2$ there holds for all $1\le R\le L$, $q\ge1$, and $1\ll p<\infty$,
\begin{equation}\label{eq:st2-SG}
\Big\|e\cdot\int_{\Ic_L}v_L\Big\|_{\Ld^{2p}(\Omega)}^2\lesssim_p\|g\|_{\Ld^{\frac{2d}{d+2}}(Q_L)}
+\|\avoir g\|_{\Ld^2(Q_L)}^2\,\bigg\|\Big(1+\int_{B_{R}}[\nabla\phi_{L}]_2^{2q}\Big)^\frac1q\bigg\|_{\Ld^p(\Omega)}.
\end{equation}
This estimate allows to upgrade \eqref{eq:st1-SG} to
\begin{multline}\label{eq:st12-SG}
\Big\|\int_{Q_L\setminus\Ic_L}g:\nabla \phi_{L}\Big\|_{\Ld^{2p}(\Omega)}^{2}
\,\lesssim_p\,\|g\|_{\Ld^{\frac{2d}{d+2}}(Q_L)}^2\\
+\|\avoir g\|_{\Ld^2(Q_L)}^2\,\bigg\|\Big(1+\int_{B_{R}}[\nabla\phi_{L}]_2^{2q}\Big)^\frac1q\bigg\|_{\Ld^p(\Omega)}.
\end{multline}
We turn to the argument for~\eqref{eq:st2-SG}.
Using the version~\eqref{eq:SGL-p} of the multiscale variance inequality~\eqref{eq:SGL} to control higher moments, we can write
\begin{equation}\label{eq:SG-phi+}
\Big\|e\cdot\int_{\Ic_L}v_L\Big\|_{\Ld^{2p}(\Omega)}^2
\,\lesssim\,p^2\,\E\bigg[\int_0^L\bigg(\int_{Q_L}\Big(\partial^{\operatorname{osc}}_{\Pc_L,B_\ell(x)}\,e\cdot\int_{\Ic_L}v_L\Big)^2dx\bigg)^p\langle\ell\rangle^{-dp}\,\pi(\ell)\,d\ell\bigg]^\frac1p,
\end{equation}
and it remains to estimate the oscillation.
Given $0\le\ell\le L$ and $x\in Q_L$, and given a realization, let $\Pc_L'$ be a locally finite point set satisfying the hardcore condition in~\ref{Hd}, with $\Pc_L'|_{Q_L\setminus B_\ell(x)}=\Pc_L|_{Q_L\setminus B_\ell(x)}$, and denote by $v_L'$ the corresponding solution of~\eqref{eq:test-v} with~$\Pc_L$ replaced by $\Pc'_L$.
We decompose
\begin{eqnarray}
\Big|e\cdot\int_{\Ic_L}v_L-e\cdot\int_{\Ic_L'}v_L'\Big|&\le&\Big|e\cdot\int_{\Ic_L'}(v_L-v_L')\Big|+\Big|e\cdot\int_{\Ic_L}v_L-e\cdot\int_{\Ic_L'}v_L\Big| \nonumber\\
&\lesssim&\Big|e\cdot\int_{\Ic_L'}(v_L-v_L')\Big|+\int_{B_{\ell+1}(x)}|v_L|.\label{eq:2term-decomp}
\end{eqnarray}
Testing the equation~\eqref{eq:reform-sedim} for $\phi_{L}'$ with $v_L-v_L'$, we obtain for any constant $c_1'\in\R$,
\begin{multline*}
\int_{Q_L}\nabla\phi_{L}':\nabla(v_L-v_L')\\
\,=\,-\alpha_L'e\cdot\int_{Q_L\setminus\Ic_L'}(v_L-v_L')
-\sum_n\int_{\partial I_{n,L}'}(v_L-v_L')\cdot\sigma(\phi_L',\Pi_{L}'-c_1')\nu,
\end{multline*}
hence, in view of the boundary conditions and of the choice $\int_{Q_L}(v_L-v_L')=0$,
\begin{multline*}
(\alpha_L'+1)\,e\cdot\int_{\Ic_L'}(v_L-v_L')
=\int_{Q_L}\nabla\phi_{L}':\nabla(v_L-v_L')\\
+\sum_{n:x_{n,L}'\in B_{\ell}(x)}\int_{\partial I_{n,L}'}\Big(v_L-\fint_{I'_{n,L}}v_L\Big)\cdot\sigma(\phi_{L}',\Pi_{L}'-c_1')\nu.
\end{multline*}
Next, testing the equation~\eqref{eq:reform-rd-v+} for $v_L-v_L'$ with $\phi_{L}'$, we obtain for any constant $c_2\in\R$,
\begin{multline*}
(\alpha_L'+1)\,e\cdot\int_{\Ic_L'}(v_L-v_L')
=-\sum_{n:x_{n,L}\in B_\ell(x)}\int_{\partial I_{n,L}}\Big(\phi_{L}'-\fint_{I_{n,L}}\phi_{L}'\Big)\cdot\big(g+\sigma(v_L,P_L-c_2)\big)\\
+\sum_{n:x_{n,L}'\in B_\ell(x)}\int_{\partial I_{n,L}'}\Big(v_L-\fint_{I'_{n,L}}v_L\Big)\cdot\sigma(\phi_L',\Pi_L'-c_1')\nu.
\end{multline*}
In view of the local regularity theory for the Stokes equation in form of~\eqref{eq:Stokes-reg-1} and~\eqref{eq:Stokes-reg-2}, together with the pressure estimates of Lemma~\ref{lem:pres}, we deduce
\begin{equation*}
\Big|e\cdot\int_{\Ic_L'}(v_L-v_L')\Big|
\,\lesssim\,\Big(\langle\ell\rangle^{d+2}+\int_{B_{\ell+2}(x)}|\nabla\phi_{L}'|^2\Big)^\frac12\Big(\int_{B_{\ell+2}(x)}|\nabla v_L|^2+|\avoir g|^2\Big)^\frac12.
\end{equation*}
Using the bound~\eqref{ant.1.3} of Substep~1.2 to replace $\phi_L'$ by $\phi_L$ in the right-hand side, and inserting this into the decomposition~\eqref{eq:2term-decomp}, we are led to
\begin{multline*}
\Big|\partial^{\operatorname{osc}}_{\Pc_L,B_\ell(x)}e\cdot\int_{\Ic_L}v_L\Big|
\,\lesssim\,\Big|\int_{B_{\ell+1}(x)}v_L\Big|\\
+\Big(\langle\ell\rangle^{d+2}+\int_{B_{\ell+2}(x)}|\nabla\phi_{L}|^2\Big)^\frac12\Big(\int_{B_{\ell+2}(x)}|\nabla v_L|^2+|\avoir g|^2\Big)^\frac12.
\end{multline*}
Inserting this into~\eqref{eq:SG-phi+},   the claim~\eqref{eq:st2-SG} follows as in  Substep~1.3.

\medskip
\step3 Fluctuation scaling on $\Ic_L$: proof that under~\ref{Mix+} in dimension $d>2$ there holds for all $1\le R\le L$, $q\ge1$, and $1\ll p<\infty$,
\begin{equation*}
\Big\|\int_{\Ic_L}g:\nabla\phi_{L}\Big\|_{\Ld^{2p}(\Omega)}^2\,\lesssim_p\,\|g\|_{\Ld^{\frac{2d}{d+2}}(Q_L)}^2
+\|g\|_{\Ld^2(Q_L)}^2\bigg\|\Big(1+\int_{B_{R}}[\nabla\phi_{L}]_2^{2q}\Big)^\frac1q\bigg\|_{\Ld^p(\Omega)},
\end{equation*}
which yields the conclusion~(i) in combination with~\eqref{eq:st12-SG}.

\medskip\noindent
Using the version~\eqref{eq:SGL-p} of the multiscale variance inequality~\eqref{eq:SGL} to control higher moments, we can write
\begin{multline}\label{eq:SG-phi++}
\Big\|\int_{\Ic_L}g:\nabla\phi_{L}\Big\|_{\Ld^{2p}(\Omega)}^2\\
\,\lesssim\,p^2\,\E\bigg[\int_0^L\bigg(\int_{Q_L}\Big(\partial^{\operatorname{osc}}_{\Pc_L,B_\ell(x)}\,\int_{\Ic_L}g:\nabla\phi_{L}\Big)^2dx\bigg)^p\langle\ell\rangle^{-dp}\,\pi(\ell)\,d\ell\bigg]^\frac1p,
\end{multline}
and it remains to estimate the oscillation.
Given $0\le\ell\le L$ and $x\in\R^d$, and given a realization, let $\Pc_L'$ be a locally finite point set satisfying the hardcore condition in~\ref{Hd}, with $\Pc_L'|_{Q_L\setminus B_\ell(x)}=\Pc_L|_{Q_L\setminus B_\ell(x)}$, and denote by $\phi_{L}'$ the corresponding solution of equation~\eqref{eq:StokesL+} with $\Pc_L$ replaced by $\Pc'_L$.
We decompose
\begin{multline*}
{\Big|\int_{\Ic_L}g:\nabla\phi_{L}-\int_{\Ic_L'}g:\nabla\phi_{L}'\Big|}
\,\le\, \bigg| \sum_{n:x_{n,L} \not\in B_\ell(x)} \int_{I_{n,L}} g: (\nabla \phi_L-\nabla \phi_L') \bigg|
\\
+\bigg|\sum_{n:x_{n,L} \in B_\ell(x)} \int_{I_{n,L}} g: \nabla \phi_L -\sum_{n:x'_{n,L} \not\in B_\ell(x)} \int_{I_{n,L}'} g:  \nabla \phi_L' \bigg|.
\end{multline*}
Since $\phi_{L}$ and $\phi_{L}'$ are both affine inside particles $I_{n,L}$'s with $x_{n,L}\notin B_\ell(x)$, we can further write
\begin{multline*}
\Big|\int_{\Ic_L}g:\nabla\phi_{L}-\int_{\Ic_L'}g:\nabla\phi_{L}'\Big|
\,\lesssim\,\bigg|\sum_{n}\Big(\fint_{I_{n,L}}g\Big):\int_{I_{n,L}}(\nabla\phi_{L}-\nabla\phi_{L}')\bigg|\\
+\Big(\int_{B_{\ell+1}(x)}|g|^2\Big)^\frac12\Big(\int_{B_{\ell+1}(x)}|\nabla\phi_{L}|^2+|\nabla\phi_{L}'|^2\Big)^\frac12.
\end{multline*}
Using the bound~\eqref{ant.1.3} of Substep~1.2 to replace $\phi_L'$ by $\phi_L$ in the right-hand side, this yields
\begin{multline}\label{eq:pre-st3-bnd-IL}
\Big|\int_{\Ic_L}g:\nabla\phi_{L}-\int_{\Ic_L'}g:\nabla\phi_{L}'\Big|
\,\lesssim\,\bigg|\sum_{n}\Big(\fint_{I_{n,L}}g\Big):\int_{I_{n,L}}(\nabla\phi_{L}-\nabla\phi_{L}')\bigg|\\
+\Big(\int_{B_{\ell+1}(x)}|g|^2\Big)^\frac12\Big(\langle\ell\rangle^{d+2}+\int_{B_{\ell+2}(x)}|\nabla\phi_{L}|^2\Big)^\frac12.
\end{multline}
It remains to estimate the first right-hand side term.
By a standard use of the Bogovskii operator in form of~\cite[Theorem~III.3.1]{Galdi}, we can construct a divergence-free tensor field $h_L\in H^1_\per(Q_L)^{d\times d}$ such that $h_L=0$ outside $\Ic_L+\delta B=\{x\in Q_L\,:\,\mathrm{dist}(x,\Ic_L)<\delta\}$ and such that for all $n$ there hold
\begin{gather}
h_L|_{I_{n,L}+\delta B}\in H^1_0(I_{n,L}+\delta B)^{d\times d},\qquad\qquad h_L|_{I_{n,L}}=\fint_{I_{n,L}}g,\nonumber\\
\|h_L\|_{H^1(I_{n,L}+\delta B)}\lesssim\|g\|_{\Ld^2(I_{n,L})}.\label{eq:bnd-h-g}
\end{gather}
In these terms, using that $h_L$ is divergence-free, we can write by means of Stokes' formula,
\begin{eqnarray*}
\sum_{n}\Big(\fint_{I_{n,L}}g\Big):\int_{I_{n,L}}(\nabla\phi_{L}-\nabla\phi_{L}')&=&\sum_{n}\Big(\fint_{I_{n,L}}g\Big):\int_{\partial I_{n,L}}(\phi_{L}-\phi_{L}')\otimes\nu\\
&=&\sum_{n}\int_{\partial I_{n,L}}h_L:(\phi_{L}-\phi_{L}')\otimes\nu\\
&=&-\int_{Q_L\setminus\Ic_L}\nabla_i\big(h_L:(\phi_{L}-\phi_{L}')\otimes\ee_i\big)\\
&=&-\int_{Q_L\setminus\Ic_L}h_L:\nabla(\phi_{L}-\phi_{L}'),
\end{eqnarray*}
so that we are reduced to the sensitivity of the gradient $\nabla\phi_L$ outside particles as already studied in Step~1.
Appealing to the bound~\eqref{eq:preant1.1} of Substep~1.1 together with the result~\eqref{ant.1.3} of Substep~1.2, and combining with~\eqref{eq:pre-st3-bnd-IL}, we obtain
\begin{multline*}
\Big|\partial^{\operatorname{osc}}_{\Pc,B_\ell(x)}\int_{\Ic_L}g:\nabla\phi_{L}\Big|
\lesssim L^{-d}\langle\ell\rangle^d\Big|e\cdot\int_{\Ic_L}v_{L,h}\Big|+\int_{B_{\ell+1}(x)}|v_{L,h}|\\
+\Big(\int_{B_{\ell+2}(x)}|\nabla v_{L,h}|^2+|\avoir {h_L}|^2+|g|^2\Big)^\frac12\Big(\langle\ell\rangle^{d+2}+\int_{B_{\ell+2}(x)}|\nabla\phi_{L}|^2\Big)^\frac12,
\end{multline*}
where $v_{L,h}$ denotes the solution of the auxiliary problem~\eqref{eq:test-v} with $g$ replaced by $h_L$.
Inserting this into~\eqref{eq:SG-phi++}, proceeding as in Substep~1.3, and taking advantage of the bound~\eqref{eq:bnd-h-g} in form of the pointwise estimate $[\langle\nabla\rangle h_L]_2(x)\lesssim\int_{B_4(x)}[g]_2^2$, the claim follows.

\medskip
\step4 Hyperuniform case~(ii).\\
The hyperuniformity assumption~\ref{Hyp+} allows to replace the multiscale inequality~\eqref{eq:SGL-p} by its version~\eqref{eq:SGL-hyp-p}, where the oscillation derivative~$\partial^{\operatorname{osc}}$ is replaced by its hyperuniform counterpart~$\partial^{\operatorname{hyp}}:=\partial^{\operatorname{mov}}+L^{-1}\partial^{\operatorname{osc}}$. The main part~$\partial^{\operatorname{mov}}$ only accounts for local perturbations with a fixed number of points, in accordance with the suppression of density fluctuations, while the second part involves the full oscillation~$\partial^{\operatorname{osc}}$ but has the small prefactor $L^{-1}$.
Given a random variable $Y(\Pc_L)$ with $\expec{Y(\Pc_L)}=0$, we write~\eqref{eq:SGL-hyp-p} as
\begin{equation}\label{eq:decomposcetc}
\|Y(\Pc_L)\|_{\Ld^{2p}(\Omega)}^2\,\lesssim\,p^2\,\Ec_p^{\operatorname{hyp}}[Y(\Pc_L)]\,\le\,p^2\,\Ec_p^{\operatorname{mov}}[Y(\Pc_L)]+p^2L^{-2}\Ec_p^{\operatorname{osc}}[Y(\Pc_L)],
\end{equation}
where we have set for abbreviation, for $\star=\operatorname{osc}$, $\operatorname{hyp}$, $\operatorname{mov}$,
\[\Ec_p^\ast[Y(\Pc_L)]\,:=\,\E\bigg[\int_0^L\bigg(\int_{Q_L}\Big(\partial^\star_{\Pc_L,B_\ell(x)}Y(\Pc_L)\Big)^2dx\bigg)^p\langle\ell\rangle^{-dp}\pi(\ell)\,d\ell\bigg]^\frac1p.\]
We split the proof into two substeps, separately considering the contribution of $\Ec_p^{\operatorname{mov}}$ and of $L^{-2}\Ec_p^{\operatorname{osc}}$ for $Y(\Pc_L)=\int_{Q_L}g:\nabla\phi_L$.

\medskip
\substep{4.1} Main contribution $\Ec_p^{\operatorname{mov}}$.\\
In Step~1, given $0\le \ell\le L$ and $x\in\R^d$, and given a realization, we now let $\Pc_L'$ be a locally finite point set satisfying the hardcore condition in $Q_L$, with $\Pc_L'|_{Q_L\setminus B_\ell(x)}=\Pc_L|_{Q_L\setminus B_\ell(x)}$ and with the additional constraint $\sharp\Pc_L'|_{B_\ell(x)}=\sharp\Pc_L|_{B_\ell(x)}$.
In particular, the latter implies $\alpha_L=\alpha_L'$, so that~\eqref{eq:bound-Dphi-00} becomes
\begin{multline}\label{eq:bound-Dphi-00+}
\int_{Q_L\setminus\Ic_L}g:\nabla(\phi_{L}-\phi_{L}')=-(\alpha_L+1)\bigg(\sum_{n:x_{n,L}\in B_\ell(x)}e\cdot\int_{I_{n,L}}v_L-\sum_{n:x_{n,L}'\in B_\ell(x)}e\cdot\int_{I_{n,L}'}v_L\bigg)\\
-\sum_{n:x_{n,L}'\in B_\ell(x)}\int_{\partial I_{n,L}'} \Big(v_L-\fint_{I_{n,L}'}v_L\Big)\cdot\sigma(\phi_L',\Pi_L'-c_2')\nu\\
+\sum_{n:x_{n,L}\in B_\ell(x)}\int_{\partial I_{n,L}}\Big(\phi_{L}'-\fint_{ I_{n,L}}\phi_{L}'\Big)\cdot\big(g+\sigma(v_L,P_L-c_1)\big)\nu.
\end{multline}
The last two terms are estimated as in Substep~1.1, but we argue differently for the first one.
The restriction $\sharp\Pc_L'|_{B_\ell(x)}=\sharp\Pc_L|_{B_\ell(x)}$ allows to write $\Pc_L|_{B_\ell(x)}:=\{x_{n_j,L}\}_{j=1}^m$ and $\Pc_L'|_{B_\ell(x)}:=\{x_{n_j',L}'\}_{j=1}^m$ for some $0\le m\lesssim \langle\ell\rangle^d$. 
We can then easily reformulate as follows the first right-hand side term of~\eqref{eq:bound-Dphi-00+}, using number conservation and disjointness,
\begin{multline*}
\bigg|\sum_{n:x_{n,L}\in B_\ell(x)}e\cdot\int_{I_{n,L}}v_L-\sum_{n:x_{n,L}'\in B_\ell(x)}e\cdot\int_{I_{n,L}'}v_L\bigg|\\
\,\le\,\sum_{j=1}^m\bigg|\int_{B(x_{n_j,L})}v_L-\int_{B(x'_{n_j',L})}v_L\bigg|
\,\lesssim\, \langle\ell\rangle\int_{B_{\ell+1}(x)}|\nabla v_L|,
\end{multline*}
so that the result~\eqref{ant.1.1} of Substep~1.1 is replaced by  
\begin{multline*}
\Big|\int_{Q_L\setminus\Ic_L}g:\nabla\phi_{L}-\int_{Q_L\setminus\Ic_L'}g:\nabla\phi_{L}'\Big|
\,\lesssim\, \langle\ell\rangle \int_{B_{\ell+1}(x)}|\nabla v_L|
\\
+
\Big(\int_{B_{\ell+2}(x)}|\nabla v_L|^2+|\avoir g|^2\Big)^\frac12\Big(\langle\ell\rangle^{d+2}+\int_{B_{\ell+2}(x)}|\nabla\phi_{L}'|^2\Big)^\frac12.
\end{multline*}
Next, a similar argument as above shows that the bound~\eqref{eq:preant1.3} in Substep~1.2 is replaced by
\begin{multline*}
\int_{Q_L}|\nabla(\phi_{L}-\phi_{L}')|^2\,\lesssim\, \langle\ell\rangle \int_{B_{\ell+1}(x)}|\nabla(\phi_{L}-\phi'_{L})|\\
+\Big(\langle\ell\rangle^{d+2}+\int_{B_{\ell+2}(x)}|\nabla\phi_{L}'|^2\Big)^\frac12\Big(\langle\ell\rangle^{d+2}+\int_{B_{\ell+2}(x)}|\nabla\phi_{L}|^2\Big)^\frac12,
\end{multline*}
so that the result~\eqref{ant.1.3} now holds in any dimension $d\ge1$.
Arguing as in Substep~1.3, the result~\eqref{eq:st1-SG} of Step~1 is then replaced by the following: under~\ref{Hyp+} in any dimension, for all $1\le R\le L$, $q\ge1$, and~$1\ll p<\infty$,
\begin{equation*}
\Ec_p^{\operatorname{mov}}\Big[\int_{Q_L\setminus\Ic_L}g:\nabla \phi_{L}\Big]
\,\lesssim_p\,
\|\avoir g\|_{\Ld^2(Q_L)}^2\,\bigg\|\Big(1+\int_{B_{R}}[\nabla\phi_{L}]_2^{2q}\Big)^\frac1q\bigg\|_{\Ld^p(\Omega)}.
\end{equation*}
Likewise, the result of Step~3 is replaced by the following: under~\ref{Hyp+} in any dimension, for all $1\le R\le L$, $q\ge1$, and $1\ll p<\infty$,
\begin{equation*}
\Ec_p^{\operatorname{mov}}\Big[\int_{\Ic_L}g:\nabla\phi_{L}\Big]\,\lesssim_p\,\|g\|_{\Ld^2(Q_L)}^2\bigg\|\Big(1+\int_{B_{R}}[\nabla\phi_{L}]_2^{2q}\Big)^\frac1q\bigg\|_{\Ld^p(\Omega)},
\end{equation*}
and hence,
\begin{equation*}
\Ec_p^{\operatorname{mov}}\Big[\int_{Q_L}g:\nabla\phi_{L}\Big]\,\lesssim_p\,\|\avoir g\|_{\Ld^2(Q_L)}^2\bigg\|\Big(1+\int_{B_{R}}[\nabla\phi_{L}]_2^{2q}\Big)^\frac1q\bigg\|_{\Ld^p(\Omega)}.
\end{equation*}

\medskip

\substep{4.2} Remainder term $L^{-2}\Ec_p^{\operatorname{osc}}$.\\
The contribution of $\Ec_p^{\operatorname{osc}}$ has already been computed in Steps~1--3 in dimension $d>2$, and it remains to show that the same bound hold for $L^{-2}\Ec_p^{\operatorname{osc}}$ in any dimension. We only need to revisit the two places where the restriction to $d>2$ is used, that is, Substep~1.2 and the estimate~\eqref{ant.1.4.1} in Substep~1.3.

\medskip\noindent
We start with revisiting Substep~1.2. Instead of using the Poincar\'e-Sobolev inequality~\eqref{eq:pre-sob-evit}, we use Poincar\'e's inequality in the following form,
\begin{eqnarray*}
\int_{B_{\ell+1}(x)}|\phi_{L}-\phi'_{L}|\lesssim\langle\ell\rangle^\frac d2\Big(\int_{Q_L}|\phi_{L}-\phi'_{L}|^2\Big)^\frac12\lesssim L\langle\ell\rangle^\frac d2\Big(\int_{Q_L}|\nabla(\phi_{L}-\phi'_{L})|^2\Big)^\frac12,
\end{eqnarray*} 
so that the conclusion~\eqref{ant.1.3} is replaced by the following, in any dimension $d\ge1$,
\begin{equation}\label{eq:modif001}
\int_{B_{\ell+2}(x)}|\nabla\phi_{L}'|^2\,\lesssim\,L^2\langle\ell\rangle^{d}+\int_{B_{\ell+2}(x)}|\nabla\phi_{L}|^2.
\end{equation}
Next, we revisit the estimate~\eqref{ant.1.4.1} in Substep~1.3: instead of appealing to the Poincar\'e-Sobolev inequality and to non-perturbative annealed $\Ld^p$ regularity in form of Theorem~\ref{CZ}, we simply use Poincar\'e's inequality and the energy inequality~\eqref{eq:energy}, in any dimension $d\ge1$,
\begin{equation}\label{eq:modif002}
\Big\|\int_{Q_L}|v_L|^2\Big\|_{\Ld^p(\Omega)}\,\lesssim\,L^2\Big\|\int_{Q_L}|\nabla v_L|^2\Big\|_{\Ld^p(\Omega)}\,\lesssim\,L^2\|g\|_{\Ld^2(Q_L)}^2.
\end{equation}
Up to these two modifications~\eqref{eq:modif001} and~\eqref{eq:modif002}, the conclusion of Steps~1--3 becomes the following, for any dimension $d\ge1$, for all $1\le R\le L$, $q\ge1$, and $1\ll p<\infty$,
\[\Ec_p^{\operatorname{osc}}\Big[\int_{Q_L}g:\nabla\phi_L\Big]\,\lesssim_p\,L^2\|\avoir g\|_{\Ld^2(Q_L)}^2\bigg\|\Big(1+\int_{B_{R+1}}[\nabla\phi_L]_2^{2q}\Big)^\frac1q\bigg\|_{\Ld^p(\Omega)}.\]
Multiplying both sides by $L^{-2}$, and inserting this into~\eqref{eq:decomposcetc} together with the result of Substep~4.1, the conclusion~(ii) follows. In contrast with the proof of~(i), we note that~(ii) only requires perturbative annealed $\Ld^p$ regularity in form of Theorem~\ref{th:annealed}.
\qed

\subsection{Proof of Proposition~\ref{prop:interpol}}
Given a ball $D\subset Q_L$ with radius $r_D\ge3$ and given arbitrary constants $c_D\in\R^d$ and $c_D'\in\R$, testing the equation~\eqref{eq:reform-sedim} for $\phi_{L}$ with $\eta_D^2(\phi_{L}-c_D)$,
where $\eta_D$ denotes a cut-off function with $\eta_D=1$ in $D$, $\eta_D=0$ outside $2D$, and $|\nabla\eta_D|\lesssim\frac1{r_D}$, such that $\eta_D$ is constant in $I_{n,L}$ for all $n$, using the boundary conditions and recalling that $\Div \phi_L=0$, we easily obtain the following Caccioppoli type estimate,
\begin{multline*}
\int_{D}|\nabla \phi_{L}|^2
\lesssim\frac{1}{r_D^2}\int_{2D}|\phi_{L}-c_D|^2\\
+\Big(\int_{2D}|\Pi_{L}-c_D'|^2\mathds1_{Q_L\setminus\Ic_L}\Big)^\frac12\Big(\frac{1}{r_D^2}\int_{2D}|\phi_{L}-c_D|^2\Big)^\frac12
+\int_{2D}|\phi_{L}-c_D|.
\end{multline*}
Bounding the last right-hand side term by
\[\int_{2D}|\phi_{L}-c_D|\lesssim r_D^{d+2}+\frac1{r_D^2}\int_{2D}|\phi_{L}-c_D|^2,\]
choosing $c_D':=\fint_{2D\setminus\Ic_L}\Pi_{L}$, and applying the pressure estimate of Lemma~\ref{lem:pres}, we obtain for all $K\ge1$,
\begin{equation}\label{eq:bound-nabphi-cacc}
\fint_{D}|\nabla \phi_{L}|^2
\lesssim \frac{K^2}{r_D^2}\fint_{2D}|\phi_{L}-c_D|^2+\frac1{K^2}\fint_{2D}|\nabla \phi_{L}|^2+r_D^{2}.
\end{equation}
Using the Poincar\'e-Sobolev inequality to estimate the first right-hand side term, with the choice $c_D:=\fint_{2D}\phi_{L}$, we deduce 
\[\Big(\fint_{D}|\nabla \phi_{L}|^2\Big)^\frac12
\lesssim K\Big(\fint_{2D}|\nabla\phi_{L}|^\frac{2d}{d+2}\Big)^\frac{d+2}{2d}+\frac1{K}\Big(\fint_{2D}|\nabla \phi_{L}|^2\Big)^\frac12+r_D.\]
While this is proven here for all balls $D$ with radius $r_D\ge3$, smuggling in local quadratic averages at scale~1 allows to infer that for all balls $D$ (with any radius $r_D>0$) and $K\ge1$, 
\[\Big(\fint_{D}[\nabla \phi_{L}]_2^2\Big)^\frac12
\lesssim K\Big(\fint_{3D}[\nabla\phi_{L}]_2^\frac{2d}{d+2}\Big)^\frac{d+2}{2d}+\frac1{K}\Big(\fint_{3D}[\nabla \phi_{L}]_2^2\Big)^\frac12+ r_D + 1.\]
Choosing $K$ large enough and applying Gehring's lemma~\cite{Gehring,GiaMo},
we deduce the following Meyers type estimate: there exists some $\eta_0>0$ (only depending on $d,\delta$) such that for all $1\le q\le 1+\eta_0$ and all~$R\ge3$,
\begin{equation*}
\Big(\fint_{B_R}[\nabla \phi_{L}]_2^{2q}\Big)^\frac1{q}
\lesssim R^2+\fint_{B_{3R}}[\nabla\phi_{L}]_2^2.
\end{equation*}
Combining this with~\eqref{eq:bound-nabphi-cacc}, we obtain for all $K\ge1$, for any constant $c_R\in\R^d$,
\begin{equation}\label{gl.2}
\Big(\fint_{B_R}[\nabla \phi_{L}]_2^{2q}\Big)^\frac1{q}
\lesssim \frac{K^2}{R^2}\fint_{B_{8R}}|\phi_{L}-c_R|^2+\frac1{K^2}\fint_{B_{8R}}|\nabla \phi_{L}|^2+R^2.
\end{equation}
For $1\le r\le R$, choosing $c_R:=\fint_{B_{8R}}\chi_r\ast\phi_{L}$, Poincar\'e's inequality yields
\begin{eqnarray*}
\fint_{B_{8R}}|\phi_{L}-c_R|^2&\lesssim &\fint_{B_{8R}}|\phi_{L}-\chi_r\ast\phi_{L}|^2+\fint_{B_{8R}}|\chi_r\ast\phi_{L}-c_R|^2\\
&\lesssim_\chi&r^2\fint_{B_{8R}}|\nabla\phi_{L}|^2+R^2\fint_{B_{8R}}|\chi_r\ast\nabla\phi_{L}|^2.
\end{eqnarray*}
Inserting this into \eqref{gl.2}, we find
\begin{equation*}
\Big(\fint_{B_R}|\nabla \phi_{L}|^{2q}\Big)^\frac1q
\lesssim \Big(K^2\frac{r^2}{R^2}+\frac1{K^2}\Big)\fint_{B_{8R}}|\nabla\phi_{L}|^2+K^2\fint_{B_{8R}}|\chi_r\ast\nabla\phi_{L}|^2+R^{2}.
\end{equation*}
Since stationarity and Jensen's inequality yield
\[\Big\|\fint_{B_{8R}}|\nabla \phi_{L}|^2\Big\|_{\Ld^p(\Omega)}\lesssim\Big\|\fint_{B_{R}}|\nabla \phi_{L}|^2\Big\|_{\Ld^p(\Omega)}\lesssim\bigg\|\Big(\fint_{B_{R}}|\nabla \phi_{L}|^{2q}\Big)^\frac1q\bigg\|_{\Ld^p(\Omega)}\]
and
$$
\expec{\Big(\fint_{B_{8R}} |\chi_r * \nabla \phi_L|^2\Big)^p} \le \expec{\fint_{B_{8R}} |\chi_r * \nabla \phi_L|^{2p}}= \expec{ |\chi_r * \nabla \phi_L|^{2p}},
$$
this implies
\begin{multline*}
\bigg\|\Big(\fint_{B_{R}}|\nabla \phi_{L}|^{2q}\Big)^\frac1q\bigg\|_{\Ld^p(\Omega)}\\
\lesssim_\chi\Big(K^2 \frac{r^2}{R^2}+\frac1{K^2}\Big)\bigg\|\Big(\fint_{B_{R}}|\nabla \phi_{L}|^{2q}\Big)^\frac1q\bigg\|_{\Ld^p(\Omega)}
+K^2\Big\|\int_{Q_L}\chi_r\nabla\phi_{L}\Big\|_{\Ld^{2p}(\Omega)}^2+R^{2}.
\end{multline*}
Choosing $K\gg1$ and $R\gg_{\chi,K} r$, the first right-hand side term can be absorbed and the conclusion follows.
\qed


\medskip
\section{Homogenization result}\label{sec:hom}

The proof of Theorem~\ref{th:hom} combines the quantitative estimates of Theorems~\ref{th:main1}--\ref{th:main2} together with the homogenization result for colloidal (non-sedimenting) suspensions in~\cite[Theorem~1]{DG-19}.
In particular, this qualitative result requires mixing assumptions and quantitative estimates.

\begin{proof}[Proof of Theorem~\ref{th:hom}]
We start with a suitable splitting of the Stokes problem~\eqref{eq:Stokes}.
In terms of the renormalized pressure $\tilde P_\e:=P_\e-\frac1\e \lambda e \cdot x$, rewriting the boundary conditions as follows,
\begin{eqnarray*}
0&=&\e^{d-1}e |I_n^\w|+  \int_{\e\partial I_n^\w}\sigma(u_\e^\w,P_\e^\w)\nu
\\
&=& \e^{d-1}e |I_n^\w|- \int_{\e\partial I_n^\w} \tfrac{1}\e( \lambda e\cdot x)\nu+    \int_{\e\partial I_n^\w}\sigma(u_\e^\w,\tilde P_\e^\w)\nu
\\
&=&\e^{d-1} (1-\lambda)e |I_n^\w| +    \int_{\e\partial I_n^\w}\sigma(u_\e^\w,\tilde P_\e^\w)\nu,
\end{eqnarray*}
and for all $\Theta\in\Md^\Skew$,
\begin{equation*}
0\,=\, \int_{\e\partial I_n^\w}\Theta\nu\cdot\sigma(u_\e^\w,P_\e^\w)\nu\,=\,\int_{\e\partial I_n^\w}\Theta\nu\cdot\sigma(u_\e^\w,\tilde P_\e^\w)\nu,
\end{equation*}
we can rewrite~\eqref{eq:Stokes} in the following equivalent form,
\begin{equation}\label{eq:Stokes++}
\left\{\begin{array}{ll}
-\triangle u_\e^\w+\nabla \tilde P_\e^\w=f- \frac1\e \lambda e,&\text{in $U\setminus\Ic_\e^\w(U)$},\\
\Div u_\e^\w=0,&\text{in $U\setminus\Ic_\e^\w(U)$},\\
u_\e^\w=0,&\text{on $\partial U$},\\
\D(u_\e^\w)=0,&\text{in $\Ic_\e^\w(U)$},\\
\e^{d-1}(1-\lambda) e |I_n^\w|+\int_{\e\partial I_n^\w}\sigma(u_\e^\w,\tilde P_\e^\w)\nu=0,&\forall n\in\Nc_\e^\w(U),\\
\int_{\e\partial I_n^\w}\Theta\nu\cdot\sigma(u_\e^\w,\tilde P_\e^\w)\nu=0,&\forall n\in\Nc_\e^\w(U),\,\forall\Theta\in\Md^\Skew,
\end{array}\right.
\end{equation}
By linearity and since $\alpha=\frac \lambda{1-\lambda}$,
we may then decompose the solution into two parts,
\[u_\e^\w=u_{\e,1}^\w+(1-\lambda)u_{\e,2}^\w,\qquad \tilde P_\e^\w=P_{\e,1}^\w+(1-\lambda)P_{\e,2}^\w,\]
where $(u_{\e,1}^\w,P_{\e,1}^\w)$ solves
\begin{equation}\label{eq:Stokes-1.0}
\left\{\begin{array}{ll}
-\triangle u_{\e,1}^\w+\nabla P_{\e,1}^\w=f ,&\text{in $U\setminus\Ic_\e^\w(U)$},\\
\Div u_{\e,1}^\w=0,&\text{in $U\setminus\Ic_\e^\w(U)$},\\
u_{\e,1}^\w=0,&\text{on $\partial U$},\\
\D(u_{\e,1}^\w)=0,&\text{in $\Ic_\e^\w(U)$},\\
\int_{\e\partial I_n^\w}\sigma(u_{\e,1}^\w,P_{\e,1}^\w)\nu=0,&\forall n\in\Nc_\e^\w(U),\\
\int_{\e\partial I_n^\w}\Theta\nu\cdot\sigma(u_{\e,1}^\w,P_{\e,1}^\w)\nu=0,&\forall n\in\Nc_\e^\w(U),\,\forall\Theta\in\Md^\Skew,
\end{array}\right.
\end{equation}
while $(u_{\e,2}^\w,P_{\e,2}^\w)$ is a rescaled proxy with Dirichlet boundary conditions on $U$ for the ``sedimentation corrector'' $(\phi,\Pi)$ in Theorem~\ref{th:main1}, cf.~\eqref{eq:cor-sed},
\begin{equation}\label{eq:Stokes-2.0}
\left\{\begin{array}{ll}
-\triangle u_{\e,2}^\w+\nabla P_{\e,2}^\w=-\frac1\e \alpha e,&\text{in $U\setminus\Ic_\e^\w(U)$},\\
\Div u_{\e,2}^\w=0,&\text{in $U\setminus\Ic_\e^\w(U)$},\\
u_{\e,2}^\w=0,&\text{on $\partial U$},\\
\D(u_{\e,2}^\w)=0,&\text{in $\Ic_\e^\w(U)$},\\
\e^{d-1}e|I_n^\w|+\int_{\e\partial I_n^\w}\sigma(u_{\e,2}^\w,P_{\e,2}^\w)\nu=0,&\forall n\in\Nc_\e^\w(U),\\
\int_{\e\partial I_n^\w}\Theta\nu\cdot\sigma(u_{\e,2}^\w,P_{\e,2}^\w)\nu=0,&\forall n\in\Nc_\e^\w(U),\,\forall\Theta\in\Md^\Skew,
\end{array}\right.
\end{equation}
where we recall $\alpha=\frac{\lambda}{1-\lambda}$.
We split the proof into two steps, and analyze the two contributions separately.

\medskip
\step1 Homogenization of~\eqref{eq:Stokes-1.0}.\\
The system~\eqref{eq:Stokes-1.0} coincides with the equations for a steady Stokes fluid with a colloidal (non-sedimenting) suspension, as we already studied in \cite{DG-19}.
In view of~\cite[Theorem~1]{DG-19}, for almost all $\w$, there holds
$u_{\e,1}^\w\rightharpoonup \bar u$ weakly in $H^1_0(U)$ and 
\[\Big(P_{\e,1}^\w-\fint_{U\setminus\Ic_\e^\w(U)}P_{\e,1}^\w\Big)\mathds1_{U\setminus\Ic_\e^\w(U)}-\Big(\bar P+\bb:\D(\bar u)-\fint_U\bar P\Big)\mathds1_{U\setminus\Ic_\e^\w(U)}~\cvf~0,\qquad\text{weakly in $\Ld^2(U)$},\]
where $(\bar u,\bar P)$ denotes the unique weak solution of the homogenized problem~\eqref{eq:Stokes-hom}.
Moreover, provided $f\in\Ld^p(U)$ for some $p>d$, for almost all $\w$, in view of~\cite[Theorem~1]{DG-19}, a corrector result holds for the velocity field in form of
\begin{equation}\label{e.corr-stokes-1.0}
\Big\|u_{\e,1}^\w-\bar u-\e\sum_{E\in\Ec}\psi_E^\w(\tfrac\cdot\e)\nabla_E\bar u\Big\|_{H^1(U)}\to0,
\end{equation}
and for the pressure field in form of 
\begin{equation}\label{e.corr-stokes-1.1}
\inf_{\kappa\in\R}\Big\|P_{\e,1}^\w-\bar P-\bb:\D(\bar u)-\sum_{E\in\Ec}(\Sigma_E^\w\mathds1_{\R^d\setminus\Ic^\w})(\tfrac\cdot\e)\nabla_E\bar u-\kappa\Big\|_{\Ld^2({U\setminus\Ic_\e^\w(U)})}\to0.
\end{equation}

\medskip
\step2 Analysis of~\eqref{eq:Stokes-2.0}: proof that, under~\ref{Mix+} for $d>2$ or under~\ref{Hyp+} for any $d\ge1$, there holds for almost all $\omega$,
\begin{gather}\label{eq:claim-ueps2}
\|u_{\e,2}^\w-\e \phi^\w(\tfrac \cdot \e)\|_{H^1(U)} \to 0,\\
\inf_{\kappa\in\R}\big\|P_{\e,2}^\w-(\Pi^\w\mathds1_{\R^d\setminus\Ic^\w})(\tfrac\cdot\e)-\kappa\big\|_{\Ld^2(U\setminus\Ic^\w_\e(U))}^2\to0,\nonumber
\end{gather}
where we recall that $(\phi,\Pi)$ is the unique solution of the infinite-volume problem~\eqref{eq:cor-sed} as given by Theorem~\ref{th:main1}.
More precisely, we claim that for all $1\le p<\infty$,
\begin{gather}\label{eq:convLp}
\|u_{\e,2}-\e\phi(\tfrac\cdot\e)\|_{\Ld^p(\Omega; H^1(U))}^2\lesssim_p\e\mu_d(\tfrac1\e)^\frac32,\\
\inf_{\kappa\in\R}\big\|\big(P_{\e,2}-(\Pi\mathds1_{\R^d\setminus\Ic})(\tfrac\cdot\e)-\kappa\big)\mathds1_{U\setminus\Ic_\e(U)}\big\|_{\Ld^p(\Omega;\Ld^2(U))}^2\lesssim_p \e\mu_d(\tfrac1\e)^\frac32,\nonumber
\end{gather}
and that for all $\kappa>0$ there exists a random variable $\mathcal X_\kappa$ with bounded moments 
such that for almost all $\w$,
\begin{gather}\label{eq:conv-as}
\|u_{\e,2}^\w-\e\phi^\w(\tfrac\cdot\e)\|_{H^1(U)}^2\lesssim  (\mathcal X^\w_\kappa)^2 \e^{1-2\kappa} \mu_d(\tfrac1\e)^{\frac32} ,\\
\inf_{\kappa\in\R}\big\|\big(P_{\e,2}^\w-(\Pi^\w\mathds1_{\R^d\setminus\Ic^\w})(\tfrac\cdot\e)-\kappa\big)\mathds1_{U\setminus\Ic_\e(U)}\big\|_{ \Ld^2(U)}^2\lesssim  (\mathcal X^\w_\kappa)^2 \e^{1-2\kappa} \mu_d(\tfrac1\e)^{\frac32},\nonumber
\end{gather}
in terms of
\[\mu_d(r):=\left\{\begin{array}{lll}
1&:&\text{under~\ref{Mix+} with $d>4$, or under~\ref{Hyp+} with $d>2$};\\
\log(2+ r)^\frac12&:&\text{under~\ref{Mix+} with $d=4$, or under~\ref{Hyp+} with $d=2$};\\
\langle r\rangle^\frac12&:&\text{under~\ref{Mix+} with $d=3$, or under~\ref{Hyp+} with $d=1$}.
\end{array}\right.\]

\medskip
\noindent
We focus on the convergence of $u_{\e,2}$.  The corresponding convergence of the pressure~$P_{\e,2}$ is obtained similarly, further using the Bogovskii operator as in the proof of Lemma~\ref{lem:pres} (see also~\cite[Substep~8.3 of Section~3]{DG-19}); details are omitted. We split the proof into three further substeps.

\medskip
\substep{2.1} Proof that for all  $1\le R\le \frac1\e$,
\begin{equation}\label{eq:preergthm}
\|u_{\e,2}^\w-\e\phi^\w(\tfrac\cdot\e)\|_{H^1(U)}^2\lesssim\e R^3+\e^d\int_{\partial_{R}U_\e}\Big(\frac1{R^2}|\phi^\w|^2+|\nabla\phi^\w|^2\Big),
\end{equation}
where we use the short-hand notation $U_\e:=\frac1\e U$ and $\partial_{R}U_\e:=\{x\in U_\e:\operatorname{dist}(x,\partial U_\e)<R\}$.

\medskip\noindent
We start by rescaling the equations: the functions $v_\e(x):=\frac1\e u_{\e,2}(\e x)$ and $Q_\e(x):=P_{\e,2}(\e x)$ on $U_\e$ satisfy
\begin{equation}\label{eq:Stokes-3.0}
\left\{\begin{array}{ll}
-\triangle v_{\e}^\w+\nabla Q_{\e}^\w=- \alpha e,&\text{in $U_\e\setminus\Ic^\w(U_\e)$},\\
\Div v_{\e}^\w=0,&\text{in $U_\e\setminus\Ic^\w(U_\e)$},\\
v^\w_{\e}=0,&\text{on $\partial U_\e$},\\
\D(v^\w_{\e})=0,&\text{in $\Ic^\w(U_\e)$},\\
e |I_n^\w|+\int_{\partial I_n^\w}\sigma(v_{\e}^\w,Q_{\e}^\w)\nu=0,&\forall n\in\Nc^\w(U_\e),\\
\int_{\partial I_n^\w}\Theta\nu\cdot\sigma(v_{\e}^\w,Q_{\e}^\w)\nu=0,&\forall n\in\Nc^\w(U_\e),\,\forall\Theta\in\Md^\Skew,
\end{array}\right.
\end{equation}
where we use the short-hand notation $\Ic^\w(U_\e):=\Ic_1^\w(U_\e)$ and $\Nc^\w(U_\e):=\Nc_1^\w(U_\e)$.
This Stokes system~\eqref{eq:Stokes-3.0} is formally the approximation of the infinite-volume problem~\eqref{eq:cor-sed} on the set~$U_\e$ with homogeneous Dirichlet boundary conditions (and discarding particles that are close to the boundary of $U_\e$); the claim~\eqref{eq:claim-ueps2} is therefore not surprising.
Arguing as in Step~1 of the proof of Theorem~\ref{th:main1}, we note that the Stokes equation for $v_\e^\w$ implies in the weak sense on the whole rescaled domain $U_\e$,
\begin{equation}\label{eq:Stokes-4.0}
-\triangle v_{\e}^\w+\nabla (Q_{\e}^\w\mathds1_{U_\e\setminus\Ic^\w(U_\e)}) =- \alpha e \mathds1_{U_\e\setminus\Ic^\w(U_\e)}- \sum_{n\in\Nc^\w(U_\e)}\delta_{\partial I_n^\w}\sigma(v_{\e}^\w,Q_{\e}^\w)\nu,
\end{equation}
which we compare to the corresponding equation for $\phi^\w$ on the whole space $\R^d$, cf.~\eqref{eq:cor-sed},
\[-\triangle\phi^\w+\nabla(\Pi^\w\mathds1_{\R^d\setminus\Ic^\w})=-\alpha e\mathds1_{\R^d\setminus\Ic^\w}-\sum_n\delta_{\partial I_n^\w}\sigma(\phi^\w,\Pi^\w)\nu.\]
We choose a smooth cut-off function $\eta_\e^\w:\R^d \to [0,1]$ such that $\eta_\e^\w$ is supported in $U_\e$, $\eta_\e^\w$~is constant inside the particles $I_n^\w$ and vanishes in $I_n^\w$ for $n\notin\Nc^\w(U_\e)$. In addition, given some $1\le R\le \frac1\e$ (that will be chosen later depending on $\e$ and $d$), we assume that $\eta_\e^\w$ satisfies $\eta_\e^\w(x)=1$ for all $x\in U_\e$ with $\operatorname{dist}(x,\partial U_\e)\ge R$, and $|\nabla\eta_\e^\w|\lesssim\frac1{R}$.
The above equation for $\phi^\w$ entails that $\eta_\e^\w \phi^\w$ satisfies the following in the weak sense on the whole space $\R^d$,
\begin{multline}\label{eq:Stokes-5.0}
-\triangle (\eta_\e^\w \phi^\w) +\nabla(\eta_\e^\w\Pi^\w\mathds1_{\R^d\setminus\Ic^\w }) =- \eta_\e^\w \alpha e \mathds1_{\R^d\setminus\Ic^\w}- \eta_\e^\w\sum_{n}\delta_{\partial I_n^\w}\sigma(\phi^\w,\Pi^\w)\nu\\
-\big(\nabla\phi^\w - \Pi^\w\mathds1_{\R^d\setminus\Ic^\w }\big)\nabla\eta_\e^\w-\nabla \cdot (\phi^\w\otimes\nabla \eta_\e^\w).
\end{multline}
Substracting~\eqref{eq:Stokes-5.0} from~\eqref{eq:Stokes-4.0}, and adding arbitrary constants to the pressures, we obtain for any $c,c'\in\R$,
\begin{multline*}
 -\triangle (v_{\e}^\w-\eta_\e^\w \phi^\w)\,=\,\nabla \Big(\eta_\e^\w(\Pi^\w-c)\mathds1_{\R^d\setminus\Ic^\w}-(Q_{\e}^\w-c')\mathds1_{U_\e\setminus\Ic^\w(U_\e)}\Big)\\
 + e \big(\alpha\eta_\e^\w \mathds1_{\R^d\setminus\Ic^\w}-\alpha \mathds1_{U_\e\setminus\Ic^\w(U_\e)}\big)
+\big(\nabla\phi^\w - (\Pi^\w-c)\mathds1_{\R^d\setminus\Ic^\w }\big)\nabla\eta_\e^\w+\nabla \cdot (\phi^\w\otimes\nabla \eta_\e^\w)\\
 +\sum_{n\in\Nc^\w(U_\e)}\delta_{\partial I_n^\w}\Big(\eta_\e^\w\sigma(\phi^\w,\Pi^\w-c)-\sigma(v_{\e}^\w,Q_{\e}^\w-c')\Big)\nu.
\end{multline*}
Testing this equation with $v_{\e}^\w-\eta_\e^\w \phi^\w\in H^1_0(U_\e)$, recalling that both $v_{\e}^\w$ and  $\phi^\w$ are divergence-free, noting that $v_{\e}^\w-\eta_\e^\w \phi^\w$ is constant inside the particles $I_n^\w$ with $n\in\Nc^\w(U_\e)$, using the boundary conditions for $v_\e^\w$ and $\phi^\w$, and using the properties of $\eta_\e^\w$,
we are led to
\begin{multline}\label{eq:est-homog-qu}
\int_{U_\e}|\nabla(v_\e^\w-\eta_\e^\w\phi^\w)|^2\lesssim\int_{\partial_{R}U_\e}|v_\e^\w-\eta_\e^\w\phi^\w|\\
+\frac1{R}\int_{\partial_{R}U_\e}|\phi^\w|\big(|Q_\e^\w-c'|\mathds1_{U_\e\setminus\Ic^\w(U_\e)}+|\Pi^\w-c|\mathds1_{\R^d\setminus\Ic^\w}\big)\\
+\frac1{R}\int_{\partial_{R}U_\e}|v_\e^\w-\eta_\e^\w\phi^\w|\big(|\nabla\phi^\w|+|\Pi^\w-c|\mathds1_{\R^d\setminus\Ic^\w}\big)\\
+\frac1{R}\int_{\partial_{R}U_\e}|\nabla(v_\e^\w-\eta_\e^\w\phi^\w)||\phi^\w|.
\end{multline}
We separately estimate the different right-hand side terms, and we start with the first one.  
Using Cauchy-Schwarz' inequality and Poincar\'e's inequality on $H^1_0(U_\e)$ restricted to the annulus $\partial_{R}U_\e$ (with Poincaré constant $O(R)$), we find for all $K\ge1$,
\begin{eqnarray*}
\int_{\partial_{R}U_\e} |v_{\e}^\w-\eta_\e^\w \phi^\w| &\lesssim& (\e^{1-d}R)^\frac12 \Big(\int_{\partial_{R}U_\e} |v_{\e}^\w-\eta_\e^\w \phi^\w| ^2\Big)^\frac12\\
&\lesssim& (\e^{1-d}R^3)^\frac12 \Big(\int_{\partial_{R}U_\e} |\nabla(v_{\e}^\w-\eta_\e^\w \phi^\w)| ^2\Big)^\frac12\\
&\lesssim& K^2\e^{1-d}R^3 +\frac1{K^2}\int_{\partial_{R}U_\e} |\nabla(v_{\e}^\w-\eta_\e^\w \phi^\w)| ^2.
\end{eqnarray*}
We turn to the second right-hand side term in~\eqref{eq:est-homog-qu}.
Using Cauchy-Schwarz' inequality and the pressure estimate of Lemma~\ref{lem:pres} with $c:=\fint_{\partial_{R}U_\e\setminus\Ic^\w}\Pi^\w$ and $c':=\fint_{\partial_{R}U_\e\setminus\Ic^\w(U_\e)}Q_\e^\w$ (the proof of Lemma~\ref{lem:pres} needs to be repeated on the annulus $\partial_{R}U_\e$, using Poincar\'e's inequality as above), we obtain for all $K\ge1$,
\begin{eqnarray*}
\lefteqn{\frac1{R}\int_{\partial_{R}U_\e}|\phi^\w|\big(|Q_\e^\w-c'|\mathds1_{U_\e\setminus\Ic^\w(U_\e)}+|\Pi^\w-c|\mathds1_{\R^d\setminus\Ic^\w}\big)}\\
&\lesssim&\Big(\frac1{R^2}\int_{\partial_{R}U_\e}|\phi^\w|^2\Big)^\frac12\Big(\int_{\partial_{R}U_\e}|Q_\e^\w-c'|^2\mathds1_{U_\e\setminus\Ic^\w(U_\e)}+|\Pi^\w-c|^2\mathds1_{\R^d\setminus\Ic^\w}\Big)^\frac12\\
&\lesssim&\Big(\frac1{R^2}\int_{\partial_{R}U_\e}|\phi^\w|^2\Big)^\frac12\Big(\e^{1-d}R^3+\int_{\partial_{R}U_\e}|\nabla\phi^\w|^2+|\nabla v_\e^\w|^2\Big)^\frac12\\
&\lesssim&\e^{1-d}R^3+K^2\int_{\partial_{R}U_\e}\Big(\frac1{R^2}|\phi^\w|^2+|\nabla\phi^\w|^2\Big)+\frac1{K^2}\int_{\partial_{R}U_\e}|\nabla (v_\e^\w-\eta_\e^\w\phi^\w)|^2.
\end{eqnarray*}
It remains to analyze the last two right-hand side terms in~\eqref{eq:est-homog-qu}. Proceeding similarly as above, we obtain for all $K\ge1$,
\begin{multline*}
\frac1{R}\int_{\partial_{R}U_\e}|v_\e^\w-\eta_\e^\w\phi^\w|\big(|\nabla\phi^\w|+|\Pi^\w-c|\mathds1_{\R^d\setminus\Ic^\w}\big)\\
\,\lesssim\,K^2\e^{1-d}R^3+K^2\int_{\partial_{R}U_\e}|\nabla\phi^\w|^2+\frac1{K^2}\int_{\partial_{R}U_\e}|\nabla(v_\e^\w-\eta_\e^\w\phi^\w)|^2,
\end{multline*}
and also
\begin{equation*}
\frac1{R}\int_{\partial_{R}U_\e}|\nabla(v_\e^\w-\eta_\e^\w\phi^\w)||\phi^\w|\,\lesssim\,K^2\int_{\partial_{R}U_\e}\frac1{R^2}|\phi^\w|^2+\frac1{K^2}\int_{\partial_{R}U_\e}|\nabla(v_\e^\w-\eta_\e^\w\phi^\w)|^2.
\end{equation*}
Inserting the above estimates into~\eqref{eq:est-homog-qu} and choosing $K$ large enough to absorb part of the right-hand side terms into the left-hand side, we are led to
\begin{equation*}
\int_{U_\e}|\nabla(v_\e^\w-\eta_\e^\w\phi^\w)|^2\lesssim\e^{1-d}R^3+\int_{\partial_{R}U_\e}\Big(\frac1{R^2}|\phi^\w|^2+|\nabla\phi^\w|^2\Big),
\end{equation*}
which yields \eqref{eq:preergthm} after rescaling and using Poincar\'e's inequality.

\medskip
\substep{2.2} Proof of \eqref{eq:convLp}.\\
Taking the expectation of the $p$-th power of  \eqref{eq:preergthm} and using the quantitative estimates of Theorem~\ref{th:main2}, under~\ref{Mix+} for $d>2$ or under~\ref{Hyp+} for any $d\ge1$, we deduce for all $1\le p<\infty$ and~$1\le R\le\frac1\e$,
\begin{equation*}
\|u_{\e,2}-\e\phi(\tfrac\cdot\e)\|_{\Ld^p(\Omega; H^1(U))}^2\lesssim_p \e\Big(\frac1{R}\mu_d(\tfrac1\e)^2+ R^3\Big),
\end{equation*}
and the claim~\eqref{eq:convLp} follows for the choice $R=\mu_d(\frac1\e)^\frac12$.

\medskip
\substep{2.3} Proof of \eqref{eq:conv-as}.\\
Starting point is again~\eqref{eq:preergthm} for $R=\mu_d(\frac1\e)^{1/2}$.
We need to convert the optimal annealed bounds  \eqref{eq:nablaphi-bnd}--\eqref{eq:phi-bnd} and \eqref{eq:nablaphi-bnd+}--\eqref{eq:phi-bnd+} into suboptimal quenched bounds and prove that for all $\kappa>0$ there exists
a random variable $\mathcal X_\kappa$ with finite algebraic moments such that for all~$x\in \R^d$,
\begin{equation}\label{ag++1}
|\phi^\w(x)|\,\le\, \mathcal X^\w_\kappa \langle x\rangle^\kappa\mu_d(|x|), \qquad |\nabla \phi^\w(x)|\,\le\, \mathcal X^\w_\kappa \langle x\rangle^\kappa.
\end{equation}
Indeed, equipped with these bounds, 
we deduce for all $0<\kappa\ll 1$ and almost all $\omega$,
$$
\e^d\int_{\partial_{R}U_\e}\Big(\frac1{R^2}|\phi^\w|^2+|\nabla\phi^\w|^2\Big) \,\lesssim \, (\mathcal X^\w_\kappa)^2 \e^{1-2\kappa} \mu_d(\tfrac1\e)^{\frac32} \xrightarrow{\e \downarrow 0} 0,
$$
so that~\eqref{eq:conv-as} follows from~\eqref{eq:preergthm}.
To prove \eqref{ag++1}, we set 
$$
 \mathcal X^\w_\kappa := \sup_{x \in \R^d} \Big(\langle x\rangle^{-\kappa}\mu_d(|x|)^{-1}|\phi^\w(x)|+\langle x\rangle^{-\kappa}|\nabla \phi^\w(x)|\Big),
$$
and it suffices to check that this random variable has bounded moments. By local regularity in form of~\eqref{eq:loc-reg00}, a covering argument yields
$$
 \mathcal X^\w_\kappa \lesssim \sup_{x \in \frac1{\sqrt d}\Z^d} \bigg(\langle x\rangle^{-\kappa}\mu_d(|x|)^{-1} \Big(\fint_{B_2(x)} |\phi^\w|^2\Big)^\frac12+\langle x\rangle^{-\kappa}\Big(\fint_{B_2(x)} |\nabla \phi^\w|^2\Big)^\frac12\bigg).
$$
Hence, given $\kappa>0$, since the function $x\mapsto \langle x\rangle^{-p\kappa}$ is integrable on $\R^d$ for $p>d/\kappa$,
the moment bounds~\eqref{eq:nablaphi-bnd}--\eqref{eq:phi-bnd} and \eqref{eq:nablaphi-bnd+}--\eqref{eq:phi-bnd+} lead (after bounding the supremum on $\frac1{\sqrt d}\Z^d$ by the sum) to
\begin{multline*}
\expec{(\mathcal X_\kappa)^p}\,\lesssim_p\,\sum_{x \in \frac1{\sqrt d}\Z^d}\langle x\rangle^{-p\kappa}\,\E\bigg[{ \mu_d(|x|)^{-p} \Big(\fint_{B_2(x)} |\phi|^2\Big)^\frac p2+\Big(\fint_{B_2(x)} |\nabla \phi|^2\Big)^\frac p2}\bigg]\\
\,\lesssim_p\,   \sum_{x \in \frac1{\sqrt d}\Z^d} \langle x\rangle^{-p\kappa}\,\lesssim_p\,1,
\end{multline*}
that is, $\mathcal X_\kappa \in \Ld^p(\Omega)$ (which is therefore almost surely finite).
\end{proof}


\medskip\appendix
\section{Functional-analytic version of hyperuniformity}\label{sec:hyper}
The present appendix is devoted to a more detailed discussion and motivation of the hyperuniformity assumptions~\ref{Hyp} and~\ref{Hyp+}.
Pioneered by Lebowitz~\cite{Lebowitz-83,JLM-93} in the physical literature for Coulomb systems,
the notion of hyperuniformity for a point process~$\Pc$ on $\R^d$ was first coined and theorized by Torquato and Stillinger~\cite{Torquato-Stillinger-03} (see also~\cite{Torquato-16,Ghosh-Lebowitz-17}) as the suppression of density fluctuations. More precisely, while for a Poisson point process one has~$\var{\sharp(\Pc \cap {B_R})}\propto|B_R|$, the process $\Pc$ is said to be hyperuniform if rather
\begin{equation}\label{eq:hyper-def}
\lim_{R\uparrow\infty} \frac{\var{\sharp(\Pc \cap {B_R})}}{|B_R|}=0.
\end{equation}
Typically, this concerns processes for which number fluctuations are a boundary effect, that is, $\var{\sharp(\Pc \cap {B_R})}\lesssim |\partial B_R|$.
Hyperuniformity can be interpreted as a hidden form of order on large scales and has been observed in various types of physical and biological systems, see e.g.~\cite{Torquato-Stillinger-03,Torquato-16}. For Coulomb gases, rigorous results on the hyperuniformity of the Gibbs state have been recently obtained in~\cite{BBNY-17,Leble-Serfaty-18,Serfaty-20}.
The simplest examples of hyperuniform processes are given by perturbed lattices, e.g.\@ $\Pc:=\{z+U_z:z\in\Z^d\}$ where the lattice points in $\Z^d$ are pertubed by iid random variables $\{U_z\}_{z\in\Z^d}$  (this model however only enjoys discrete stationarity due to the lattice structure; see~\cite{Peres-Sly-14,Yakir-20} for refined properties).

\medskip\noindent
Alternatively, hyperuniformity is known to be equivalent to the vanishing of the structure factor in the small-wavenumber limit, that is,
\[\lim_{k\to0}S(k)=0,\]
where the structure factor is defined as the Fourier transform $S(k):=\widehat h_{2}(k)$ of the total pair correlation function $h_2$, cf.~Definition~\ref{def:corr-fct}. If the pair correlation function is integrable, this can equivalently be written as
\begin{equation}\label{eq:hyp-equiv0}
S(0)=\int_{\R^d}h_2=0.
\end{equation}
The advantage of this reformulation in terms of the structure factor $S(k)$ is that the latter can be directly observed in diffraction experiments. This property of vanishing structure factor is reminiscent of crystals, and indeed hyperuniform processes share crystalline properties on large scales, although they can be statistically isotropic like gases, thereby leading to a new state of matter~\cite{Torquato-18}.

\medskip\noindent
In the spirit of~\eqref{eq:hyp-equiv0}, in our periodized setting, for a family $\{\Pc_L\}_{L\ge1}$ of random point processes $\Pc_L$ on $Q_L$,
we consider the following slightly relaxed definition of hyperuniformity, cf.~\ref{Hyp},
\begin{equation}\label{eq:hypL}
\sup_{L\ge1}\,L^2\,\Big|\int_{Q_L}h_{2,L}\Big|\,<\, \infty,
\end{equation}
which is viewed as the approximate vanishing of the corresponding structure factors at~$0$ in the limit $L\uparrow\infty$. The precise rate $O(L^{-2})$ is chosen in view of Lemma~\ref{lem:var-lin} below.
As claimed, such a definition of hyperuniformity in terms of structure factors implies the suppression of density fluctuations in the following sense.
\begin{lem}[Density fluctuations~\cite{Torquato-Stillinger-03,Ghosh-Lebowitz-17}]
Let a family $\{\Pc_L\}_{L\ge1}$ of random point processes $\Pc_L$ on $Q_L$ be hyperuniform in the sense of~\eqref{eq:hypL}.
Then, for all $1\le R\le L$,
\[\var{\sharp(\Pc_L \cap {B_R})}\,\lesssim\, \rho_L^2|B_R|\,\Big(L^{-2}+\int_{Q_L}\big(1\wedge\tfrac{|x|_L}R\big)\,|g_{2,L}(x)|\,dx\Big).\]
In particular, provided that the pair correlation function $g_{2,L}$ has fast enough decay in the sense of $\sup_{L\ge1}\int_{Q_L}|x|_L|g_{2,L}(x)|\,dx<\infty$, we deduce for all $1\le R\le L$,
\[\var{\sharp(\Pc_L \cap {B_R})}\,\lesssim\,\rho_L^2|\partial B_R|.\qedhere\]
\end{lem}

\begin{proof}
Number fluctuations are computed as follows,
\begin{eqnarray*}
\var{\sharp(\Pc_L \cap {B_R})}\,=\,\Var\bigg[{\sum_n\mathds1_{B_R}(x_{n,L})}\bigg]&=&\rho_L^2\iint_{B_R\times B_R}h_{2,L}(x-y)\,dxdy\\
&=&\rho_L^2\int_{Q_L}|B_R(-x)\cap B_R|\,h_{2,L}(x)\,dx.
\end{eqnarray*}
Hence, decomposing
\[|B_R(-x)\cap B_R|=|B_R|-|B_R\setminus B_R(-x)|=|B_R|+\big(1\wedge\tfrac{|x|_L}R\big)\,O(|B_R|),\]
where the last summand is a continuous function that vanishes at $x=0$, 
we deduce
\begin{gather*}
\bigg|\frac1{|B_R|}\var{\sharp(\Pc_L \cap {B_R})}-\rho_L^2\int_{Q_L}h_{2,L}\bigg|
\,\lesssim\,\rho_L^2\int_{Q_L}\big(1\wedge\tfrac{|x|_L}R\big)\,|g_{2,L}(x)|\,dx,
\end{gather*}
and the conclusion follows.
\end{proof}

Recall that a Poisson point process $\Pc=\{x_n\}_n$ on $\R^d$ satisfies $\var{\sum_n\zeta(x_n)}\propto\int_{\R^d}|\zeta|^2$ for all $\zeta\in C^\infty_c(\R^d)$, and that similarly any point process $\Pc$ with integrable pair correlation function satisfies
\begin{equation}\label{eq:var-est00}
\Var\bigg[{\sum_n\zeta(x_n)}\bigg]\,=\,\rho^2\iint_{\R^d\times\R^d}\zeta(x)\zeta(y)\,h_2(x-y)\,dxdy\,\lesssim\,\rho^2\int_{\R^d}|\zeta|^2.
\end{equation}
The suppression of density fluctuations under hyperuniformity is naturally expected to lead to an improved version of such a variance estimate. Indeed, given an independent copy~$\{x_{n}'\}_n$ of $\Pc=\{x_{n}\}_n$, we may represent
\[\Var\bigg[\sum_n \zeta(x_{n})\bigg]\,=\,\E\E'\bigg[\frac12\Big(\sum_n \zeta(x_{n})-\sum_n \zeta(x_{n}')\Big)^2\bigg],\]
and the suppression of density fluctuations would formally allow to locally couple the random point sets~$\{x_{n}'\}_n$ and $\{x_{n}\}_n$, only comparing points of the two realizations one to one locally, which would ideally translate into the gain of a derivative: for all $\zeta\in C^\infty_c(\R^d)$,
\begin{equation}\label{eq:withhyp+}
\Var\bigg[\sum_n \zeta(x_{n})\bigg]\,\lesssim\,\rho^2\int_{\R^d}|\nabla\zeta|^2.
\end{equation}
Indeed, provided that the pair correlation function has fast enough decay, it can be checked that hyperuniformity~\eqref{eq:hyper-def} is equivalent to this improved variance inequality~\eqref{eq:withhyp+}.
In our periodized setting, a rigorous statement is as follows.

\begin{lem}[Functional characterization of hyperuniformity]\label{lem:var-lin}
Consider a family $\{\Pc_L\}_{L\ge1}$ of random point processes $\Pc_L=\{x_{n,L}\}_n$ on $Q_L$ and assume that the pair correlation function $g_{2,L}$ has fast enough decay in the sense of
\[\sup_{L\ge1}\int_{Q_L}|x|_L^2|g_{2,L}(x)|\,dx<\infty.\]
Then $\{\Pc_L\}_{L\ge1}$ is hyperuniform in the sense of~\eqref{eq:hypL} if and only if  for all $L\ge1$ and $\zeta\in C_\per^\infty(Q_L)$ we have
\begin{equation}\label{var-estim-lin}
\Var\bigg[{\sum_n\zeta(x_{n,L})}\bigg]\,\lesssim\,\rho_L^2\int_{Q_L}|\nabla\zeta|^2+L^{-2}\rho_L^2\int_{Q_L}|\zeta|^2.
\end{equation}
In particular, the latter implies for all $\zeta\in C_\per^\infty(Q_L)$ with $\expec{\sum_n\zeta(x_{n,L})}=\rho_L\int_{Q_L}\zeta=0$,
\[\Var\bigg[{\sum_n\zeta(x_{n,L})}\bigg]\,\lesssim\,\rho_L^2\int_{Q_L}|\nabla\zeta|^2.\qedhere\]
\end{lem}
\begin{proof}
By the definition of the total pair correlation function $h_{2,L}$, cf.~Definition~\ref{def:corr-fct}, recall that
\[\Var\bigg[{\sum_n\zeta(x_{n,L})}\bigg]=\rho_L^2\iint_{Q_L\times Q_L}\zeta(x)\zeta(y)\,h_{2,L}(x-y)\,dxdy.\]
Choosing $\zeta=1$, the variance inequality~\eqref{var-estim-lin} yields
\[\Big|\int_{Q_L}h_{2,L}\Big|\lesssim L^{-2},\]
that is, our definition~\eqref{eq:hypL} of hyperuniformity, and it remains to prove the converse implication.
Recomposing the square, the above identity for the variance takes the form
\[\Var\bigg[{\sum_n\zeta(x_{n,L})}\bigg]=-\frac12\rho_L^2\iint_{Q_L\times Q_L}|\zeta(x)-\zeta(y)|^2\,g_{2,L}(x-y)\,dxdy
+\Big(\rho_L^2\int_{Q_L}|\zeta|^2\Big)\Big(\int_{Q_L}h_{2,L}\Big).\]
Using the decay of correlations to estimate the first right-hand side term, and
hyperuniformity~\eqref{eq:hypL} to estimate the last one, the variance inequality~\eqref{var-estim-lin} follows.
\end{proof}

While the above variance inequality is restricted to \emph{linear} functionals $Y_L=\sum_n\zeta(x_{n,L})$ of the point process, the analysis of nonlinear multibody interactions requires a corresponding tool for general nonlinear functionals.
For general functionals $Y=Y(\Pc)$ of a Poisson point process $\Pc$ with unit intensity on $\R^d$, the following variance inequality is known to hold~\cite{Wu-00,Last-Penrose-11a},
\[\var{Y(\Pc)}\le\E\bigg[\int_{\R^d}\big(\partial^{\operatorname{add}}_yY(\Pc)\big)^2dy\bigg],\qquad \partial^{\operatorname{add}}_yY(\Pc):=Y(\Pc\cup\{y\})-Y(\Pc),\]
where the difference operator $\partial^{\operatorname{add}}$ is known as the {\it add-one-point operator}.
More general versions of this type of functional inequality have been considered in the literature as a convenient quantification of nonlinear mixing in order to cover various classes of examples. In this spirit, our improved mixing assumption~\ref{Mix+} is formulated in terms of the multiscale variance inequality~\eqref{eq:SGL} of~\cite{DG1,DG2}.
As shown in~\cite[Section~3]{DG2}, this covers most examples of interest in materials science~\cite{Torquato-02}, including for instance (periodized) hardcore Poisson processes and random parking processes.
Applied to a linear functional $Y_L=\sum_n\zeta(x_{n,L})$, this variance inequality~\eqref{eq:SGL} clearly reduces to~\eqref{eq:var-est00}, so that~\eqref{eq:SGL} can indeed be viewed as a nonlinear version of~\eqref{eq:var-est00}.

\medskip\noindent
In the hyperuniform setting, as number fluctuations are suppressed, the {\it add-one-point operator} in the above or the general {\it oscillation} in~\ref{Mix+} could be intuitively replaced by a suitable {\it ``move-point'' operator}, only allowing to locally {\it move} points of the process, but not add or remove any. A general version of this idea is formalized as the improved hyperuniformity assumption~\ref{Hyp+} in form of~\eqref{eq:SGL-hyp}.
Again, applied to a linear functional $Y_L=\sum_n\zeta(x_{n,L})$, this new variance inequality~\eqref{eq:SGL-hyp} clearly reduces to~\eqref{var-estim-lin}, so that~\eqref{eq:SGL-hyp} can be viewed as a nonlinear version.
We believe that this new functional inequality is of independent interest.

\begin{ex}[Perturbed lattices]
We briefly show that assumption~\ref{Hyp+} in terms of the hyperuniform multiscale variance inequality~\eqref{eq:SGL-hyp} is not empty.
For that purpose, we consider the simplest example of a hyperuniform process on $Q_L$, that is, the perturbed lattice~$\Pc_L:=\{z+U_z:z\in\Z^d\cap Q_L\}$, where the lattice points in $\Z^d \cap Q_L$ are perturbed by iid random variables $\{U_z\}_{z\in\Z^d\cap Q_L}$, say with values in the unit ball $B$.
This model $\Pc_L$ is easily checked to satisfy the following stronger version of the variance inequality~\eqref{eq:SGL-hyp}:  for all $\sigma(\Pc_L)$-measurable random variables~$Y(\Pc_L)$,
\begin{equation}\label{eq:pert-lat-SG+}
\var{Y(\Pc_L)}\,\le\,\frac12\E\bigg[\int_{Q_L}\Big(\partial^{\operatorname{mov}}_{\Pc_L,B_{1+\sqrt d/2}(z)}Y(\Pc_L)\Big)^2\,dz\bigg],
\end{equation}
where we recall that the move-point derivative $\partial^{\operatorname{mov}}$ is defined in~\ref{Hyp+}.
This is indeed a direct consequence of the Efron-Stein inequality~\cite{Efron-Stein-81} for the iid sequence $\{U_z\}_{z\in\Z^d\cap Q_L}$ in the following form: for $\Pc_{L,z}:=\{y+U_y\}_{y:y\ne z}\cup \{z+U'_z\}$ with $\{U'_z\}_z$ an iid copy of $\{U_z\}_z$,
\begin{equation*}
\var{Y(\Pc_L)}\,\le \,\frac12\E\bigg[\sum_{z\in \Z^d\cap Q_L} \big(Y(\Pc_L)-Y(\Pc_{L,z})\big)^2\bigg],
\end{equation*}
while $Y(\Pc_L)-Y(\Pc_{L,z})$ can be bounded by $\partial^{\operatorname{mov}}_{\Pc,B(z)}Y(\Pc)$, thus leading to~\eqref{eq:pert-lat-SG+}.
\end{ex}


\section*{Acknowledgements}
MD acknowledges financial support from the CNRS-Momentum program,
and AG from the European Research Council (ERC) under the European Union's Horizon 2020 research and innovation programme (Grant Agreement n$^\circ$~864066).

\bibliographystyle{plain}

\def\cprime{$'$} \def\cprime{$'$} \def\cprime{$'$}

\end{document}